\documentclass[letter,final,twosided]{amsart}
\usepackage{amsmath}
\usepackage{paralist}
\usepackage[english]{babel}
\usepackage{amsfonts}
\usepackage{amscd}
\usepackage{amssymb}
\usepackage{amsmath}
\usepackage{amstext}
\usepackage{latexsym}
\usepackage{xspace}
\usepackage[all]{xy}
\usepackage{hyperref}

\newcommand{\baseRing}[1]{\ensuremath{\mathbb{#1}}}
\newcommand{\R}{\baseRing{R}}
\newcommand{\C}{\baseRing{C}}
\newcommand{\N}{\baseRing{N}}

\newcommand{\jdef}[1]{\index{#1}\emph{#1}}
\newcommand{\stext}[1]{\ensuremath{\quad\text{#1}\quad}}

\newcommand{\pd}[2]{{\frac{\partial {#1}}{\partial {#2}}}}

\renewcommand{\iff}{\ensuremath{\Leftrightarrow}\xspace}

\newcommand{\CC}{\ensuremath{{\mathcal{A}}}}
\newcommand{\CCp}[1]{\ensuremath{{\mathcal{A}^{#1}}}}
\newcommand{\DC}{\ensuremath{{\mathcal{A}_d}}}
\newcommand{\DCp}[1]{\ensuremath{{\mathcal{A}_d^{#1}}}}

\newcommand{\DCT}{\ensuremath{\DCp{T_2}}}
\newcommand{\HL}{\ensuremath{h_d}}
\newcommand{\HLds}[1]{\ensuremath{\overline{\HL^{{#1}}}}}
\newcommand{\HLd}[1]{\ensuremath{{\HL^{{#1}}}}}
\newcommand{\SG}{\ensuremath{G}\xspace}
\newcommand{\jgsg}{\ensuremath{\mathfrak{g}}\xspace}

\numberwithin{equation}{section}

\providecommand{\abs}[1]{\ensuremath{\left\lvert{#1}\right\rvert}}
\providecommand{\norm}[1]{\ensuremath{\left\lVert{#1}\right\rVert}}

\DeclareMathOperator{\re}{Re}
\DeclareMathOperator{\im}{Im}

\DeclareMathOperator{\Ob}{ob}
\DeclareMathOperator{\Mor}{mor}
\DeclareMathOperator{\Diff}{Diff}

\DeclareMathOperator{\FBTS}{\textit{E}}
\DeclareMathOperator{\FBBS}{\textit{M}}
\DeclareMathOperator{\FBM}{\phi}
\DeclareMathOperator{\DLPSC}{{\mathfrak{LP}_\textit{d}}}
\DeclareMathOperator{\DLPSMor}{\Upsilon}

\DeclareMathOperator{\IVCM}{\mathcal{P}}
\DeclareMathOperator{\Lagr}{\textit{L}_\textit{d}}
\DeclareMathOperator{\DLPS}{\mathcal{M}}

\newcommand{\ie}{\textsl{i.e.}\xspace}
\newcommand{\ti}[1]{\widetilde{#1}}
\newcommand{\conj}{\overline}

\newtheorem{theorem}{Theorem}[section]
\newtheorem{corollary}[theorem]{Corollary}

\newtheorem{lemma}[theorem]{Lemma}
\newtheorem{proposition}[theorem]{Proposition}

\theoremstyle{definition}
\newtheorem{definition}[theorem]{Definition}
\newtheorem{remark}[theorem]{Remark}
\newtheorem{example}[theorem]{Example}

\title[Discrete lagrangian reduction by stages]{Lagrangian reduction of\\
  discrete mechanical systems by stages}

\author[Javier Fern\'andez, Cora Tori and Marcela Zuccalli]{}



\email{jfernand@ib.edu.ar}
\email{cora@mate.unlp.edu.ar}
\email{marce@mate.unlp.edu.ar}

\thanks{This research was partially supported by grants from the
   Universidad Nacional de Cuyo, Universidad Nacional de La Plata and
   CONICET}

\begin{document}

\bibliographystyle{amsplain}

\maketitle

\centerline{\scshape Javier Fern\'andez}

\medskip {\footnotesize

  \centerline{Instituto Balseiro, Universidad Nacional de Cuyo --
    C.N.E.A.}

  \centerline{ Av. Bustillo 9500, San Carlos de
    Bariloche, R8402AGP, Rep\'ublica Argentina} }

\medskip 

\centerline{\scshape Cora Tori and Marcela Zuccalli} \medskip
{\footnotesize
  \centerline{Departamento de Matem\'atica, Facultad de Ciencias Exactas}
      \centerline{Universidad Nacional de La Plata}
  \centerline{50 y 115, La Plata, Buenos Aires, 1900, Rep\'ublica
    Argentina}
}


\bigskip

\begin{abstract}
  In this work we introduce a category of discrete Lagrange--Poincar\'e
  systems $\DLPSC$ and study some of its properties. In particular, we
  show that the discrete mechanical systems and the discrete
  mechanical systems obtained by the Lagrangian reduction of symmetric
  discrete mechanical systems are objects in $\DLPSC$. We introduce a
  notion of symmetry groups for objects of $\DLPSC$ and introduce a
  reduction procedure that is closed in the category
  $\DLPSC$. Furthermore, under some conditions, we show that the
  reduction in two steps (first by a closed normal subgroup of the
  symmetry group and then by the residual symmetry group) is
  isomorphic in $\DLPSC$ to the reduction by the full symmetry group.
\end{abstract}


\section{Introduction}
\label{sec:introduction}

The study of mechanical systems with symmetries is a classical
subject. A standard technique used in the area is the construction of
a certain dynamical system ---the \jdef{reduced system}--- where some
or all of the original symmetries have been eliminated and whose
trajectories can be used to obtain the trajectories of the original
system. This general idea has been developed and used in many
different contexts. In the Lagrangian formulation of Classical
Mechanics, one such approach is given by
E. Routh~\cite{bo:routh-stability_of_a_given_state_of_motion},
although it was implicit in Lagrange's original ideas. A modern
treatment, including nonholonomic constraints, is given by H. Cendra,
J. Marsden and T. Ratiu
in~\cite{ar:cendra_marsden_ratiu-geometric_mechanics_lagrangian_reduction_and_nonholonomic_systems}. In
the modern Hamiltonian case, there are the original works of
V. Arnold~\cite{ar:arnlod-sur_la_geometrie_differentielle_des_groupes_de_lie_de_dimension_infinie_et_ses_applications_a_l'hydrodynamique_des_fluides_parfaits},
S. Smale~\cite{ar:smale-topology_and_mechanics_1,ar:smale-topology_and_mechanics_2},
K. Meyer~\cite{ar:meyer-symmetries_and_integrals_in_mechanics},
J. Marsden and
A. Weinstein~\cite{ar:marsden_weinstein-reduction_of_symplectic_manifolds_with_symmetry}
and, among the recent literature, J. Marsden et
al.~\cite{bo:marsden_misiolek_ortega_perlmutter_ratiu-hamiltonian_reduction_by_stages}. There
is also a field-theoretic version as explained, for instance, by
M. Castrill\'on Lopez and T. Ratiu
in~\cite{ar:castrillonlopez_ratiu-reduction_of_principal_bundles_covariant_lagrange_poincare_equations}. In
the case of discrete mechanical systems (DMS), different versions of
reduction theory have been considered, among others, by S. Jalnapurkar
et
al. in~\cite{ar:jalnapurkar_leok_marsden_west-discrete_routh_reducion},
R. McLachlan and M. Perlmutter
in~\cite{ar:mclachlan_perlmutter-integrators_for_nonholonomic_mechanical_systems}
and the authors
in~\cite{ar:fernandez_tori_zuccalli-lagrangian_reduction_of_discrete_mechanical_systems}. Reduction
theory has also been developed in the context of Lie groupoids and Lie
algebroids as discussed by J. C. Marrero et al
in~\cite{ar:marrero_martin_martinez-discrete_lagrangian_and_hamiltonian_mechanics_on_lie_groupoids}
and by D. Iglesias et al
in~\cite{ar:iglesias_marrero_martin_martinez_padron-reduction_of_symplectic_lie_algebroids_by_a_lie_subalgebroid_and_a_symmetry_lie_group}.

In some cases, if $\SG$ is a symmetry group of a mechanical system, it
may be convenient to consider a partial reduction, that is, the
reduction of the system by a subgroup $H\subset \SG$ and, eventually,
as a second step, the reduction of any remaining symmetries in the
associated reduced system. This process is called \jdef{reduction by
  (two) stages}. A problem is that, in general, the reduced system
associated to a symmetric mechanical system is a dynamical system that
is not a mechanical system. Therefore, a second reduction cannot be
performed in the framework of mechanical systems. For continuous time
systems, the solution found by different authors consists of enlarging
the class of systems considered beyond the mechanical ones, developing
a reduction theory for those generalized systems that extends the
original reduction of mechanical systems and, eventually, considering
the reduction by stages in this generalized framework. This is the
case, for instance, in the Lagrangian context,
of~\cite{bo:cendra_marsden_ratiu-lagrangian_reduction_by_stages}, and,
for Lagrangian systems with nonholonomic constraints,
of~\cite{ar:cendra_diaz-lagrange_dalembert_poincare_equations_by_several_stages,phd:diaz-tesis_doctorado}. In
the Hamiltonian case, it is extensively analyzed
in~\cite{bo:marsden_misiolek_ortega_perlmutter_ratiu-hamiltonian_reduction_by_stages}. It
should be remarked that in the Lie algebroid or Lie groupoid contexts
the problem described in this paragraph does not arise as the
reduction of an object in one of these categories lies within the same
category.

The purpose of the present work is to introduce a generalized
framework to study the reduction of DMS by stages. In this sense, it
parallels~\cite{bo:cendra_marsden_ratiu-lagrangian_reduction_by_stages}
for discrete time mechanical systems. More precisely, a category
$\DLPSC$ of \jdef{discrete Lagrange--Poincar\'e systems} (DLPS) is
introduced and it is shown that DMSs are among its objects in a
natural way. Also, the reduced systems associated to symmetric DMSs
are in $\DLPSC$. The dynamics of a DLPS is defined via a variational
principle that, for a DMS, reduces to the discrete Hamilton
Principle. Then, a reduction theory for symmetric DLPSs is
developed. It is shown that this theory, when applied to DMS,
coincides with the one defined
in~\cite{ar:fernandez_tori_zuccalli-lagrangian_reduction_of_discrete_mechanical_systems}. At
this stage, we can prove the main result of the paper,
Theorem~\ref{th:isomorphism_reduced_spaces_in_stages}, showing that,
under appropriate conditions, the reduction in two stages is feasible
and isomorphic in $\DLPSC$ to the full reduction in one step.

The construction of reduced systems considered
in~\cite{bo:cendra_marsden_ratiu-lagrangian_reduction_by_stages},
\cite{ar:cendra_diaz-lagrange_dalembert_poincare_equations_by_several_stages,phd:diaz-tesis_doctorado},
\cite{ar:leok_marsden_weinstein-a_discrete_theory_of_connections_on_principal_bundles},
\cite{ar:fernandez_tori_zuccalli-lagrangian_reduction_of_discrete_mechanical_systems}
and here, all require additional data: a connection or a
\jdef{discrete connection} on a certain principal bundle. We prove
that, even though the reduced DLPSs depend on the specific discrete
connection used, any two choices lead to DLPSs that are isomorphic in
$\DLPSC$. Last, we prove that under fairly general conditions discrete
connections on the appropriate principal bundles satisfying the
conditions required by the reduction by stages results exist.

It is well known and very useful that DMSs carry natural symplectic
structures. But, in general, their associated reduced systems are not
symplectic; instead, they carry Poisson structures. In contrast,
general DLPSs do not carry a natural symplectic or Poisson structure;
we prove that, when a Poisson structure is added to a DLPS, then it
descends to any reduced system associated to it. A consequence of this
fact is that all DLPSs obtained by a finite number of reductions from
a symmetric DMS, have a ``natural'' Poisson structure coming from the
symplectic structure of the original DMS.

The plan for the paper is as
follows. Section~\ref{sec:general_recollections} reviews the notion of
discrete connection on a principal bundle and some of its basic
properties. Section~\ref{sec:DLPS} introduces discrete
Lagrange--Poincar\'e systems, their dynamics and explores some
examples. Section~\ref{sec:categorical_formulation} defines a category
whose objects are the
DLPSs. Section~\ref{sec:reduction_of_generalized_discrete_mechanical_systems}
introduces the notion of symmetry group of a DLPS and, then,
constructs a ``reduced'' DLPS associated to any symmetric DLPS and
discrete connection; it also compares the dynamics of the reduced DLPS
and that of the original DLPS, proving that the trajectories of one
system can be obtained from those of the other. An example of
reduction process is analyzed in
Section~\ref{sec:example_T2}. Section~\ref{sec:reduction_in_two_stages}
considers the reduction of DLPSs in two
stages. Section~\ref{sec:poisson_structures} studies some aspects of
Poisson structures on DLPSs. The paper closes with
Section~\ref{sec:appendix}, where we list some basic and general
results on group actions on manifolds and on principal bundles; most
of this material is standard and it is included to have a unified
notation and reference point.

\vskip .3cm

Finally, we wish to thank Hern\'an Cendra for his continuous interest in
this work and many very useful discussions.


\section{General recollections}
\label{sec:general_recollections}

Let $\SG$ be a Lie group acting on the left on the manifold $Q$ by
$l^Q$ in such a way that the quotient map
$\pi^{Q,\SG}:Q\rightarrow Q/\SG$ be a principal $\SG$-bundle; we also
consider the induced diagonal $\SG$-action $l^{Q\times Q}$ on
$Q\times Q$. Leok, Marsden and Weinstein introduced
in~\cite{ar:leok_marsden_weinstein-a_discrete_theory_of_connections_on_principal_bundles}
a notion of discrete connection on principal $\SG$-bundles. The
following definition comes
from~\cite{ar:fernandez_zuccalli-a_geometric_approach_to_discrete_connections_on_principal_bundles},
which the reader should refer to for further details on discrete
connections.

\begin{definition}
  \label{def:discrete_connection}
  Let $Hor\subset Q\times Q$ be an $l^{Q\times Q}$-invariant
  submanifold containing the diagonal $\Delta_Q\subset Q\times Q$.  We
  say that $Hor$ defines the \jdef{discrete connection} $\DC$ on the
  principal bundle $\pi^{Q,\SG}:Q\rightarrow Q/\SG$ if $(id_Q\times
  \pi^{Q,\SG})|_{Hor}:Hor\rightarrow Q\times (Q/\SG)$ is an injective
  local diffeomorphism. We denote $Hor$ by $Hor_{\DC}$.
\end{definition}

When $Hor_\DC$ is a discrete connection on
$\pi^{Q,\SG}:Q\rightarrow Q/\SG$, it is easy to see that for any
$(q_0,q_1)\in Q\times Q$, there is a unique $g\in \SG$ such that
$(q_0,l^Q_{g^{-1}}(q_1))\in Hor_\DC$, where $l^Q_g(q) := l^Q(g,q)$. In
this case, the \jdef{discrete connection form}
$\DC:Q\times Q\rightarrow \SG$ is defined by $\DC(q_0,q_1):=g$.

\begin{remark}\label{rem:Ad_not_everywhere_but_ignored}
  It is well known that when the principal $\SG$-bundle is not
  trivial, the existence of $g$ stated above cannot be assured in
  general. It is true, though, when $(q_0,q_1)$ is in a certain open
  subset of $Q\times Q$ containing the diagonal, known as the
  \jdef{domain} of the discrete connection. Still, in what follows, we
  omit this technical detail in order to keep the notation simple.
\end{remark}

As in the case of connections on a principal $\SG$-bundle, discrete
connections define a notion of discrete horizontal lift, that we
introduce below.

\begin{definition}
  Let $\DC$ be a discrete connection on the principal $\SG$-bundle
  $\pi^{Q,\SG}:Q\rightarrow Q/\SG$. The \jdef{discrete horizontal
    lift} $\HLd{} : Q\times (Q/\SG) \rightarrow Q\times Q$ is the
  inverse map of the injective local diffeomorphism
  $(id_Q\times \pi)|_{Hor_\DC}:Hor_\DC\rightarrow Q\times
  (Q/\SG)$. Explicitly
  \begin{equation*}
    \HLd{q_0}(r_1) \:= \HL(q_0,r_1) := (q_0,q_1) \quad \iff \quad (q_0,q_1)\in
    Hor_{\DC} \ \text{ and }\  \pi^{Q,\SG}(q_1)=r_1.
  \end{equation*}
  We define $\HLds{} :=p_2\circ \HLd{}$ and
  $\HLds{q_0} :=p_2\circ \HLd{q_0}$, where
  $p_2:Q\times Q\rightarrow Q$ is the projection on the second
  variable. More generally,
  $p_j:X_1\times\cdots\times X_k\rightarrow X_j$ is the projection
  from the Cartesian product onto the $j$-th component.
\end{definition}

\begin{remark}
  In the same spirit of
  Remark~\ref{rem:Ad_not_everywhere_but_ignored}, $\HLd{}$ may not be
  defined on all of $Q\times (Q/\SG)$ but only on an open subset. We
  ignore this fact in what follows.
\end{remark}

Discrete connections and their lifts satisfy a number of
properties. The next result reviews some of them.
\begin{proposition}\label{prop:dcs_and_lifts}
  Let $\DC$ be a discrete connection on the principal $\SG$-bundle
  $\pi^{Q,\SG}:Q\rightarrow Q/\SG$. Then,
  \begin{enumerate}
  \item \label{it:dcs_and_lifts-smooth} the discrete connection form
    $\DC$ and the discrete horizontal lift $\HLd{}$ are smooth
    functions and,
  \item \label{it:dcs_and_lifts-equivariance} if we consider the left
    $\SG$-actions on $\SG$ and $Q\times (Q/\SG)$ given by
    \begin{equation*}
      l^\SG_g(g'):= g g' g^{-1} \stext{ and } l^{Q\times
        (Q/\SG)}_g(q_0,r_1) := (l^Q_g(q_0),r_1)
    \end{equation*}
    as well as the diagonal action on $Q\times Q$ then $\DC$, $\HLd{}$
    and $\HLds{}$ are $\SG$-equivariant.
  \item \label{it:dcs_and_lifts-GxG} More generally, for any
    $g_0,g_1\in\SG$,
    \begin{equation*}
      \DC(l^Q_{g_0}(q_0),l^Q_{g_1}(q_1)) = g_1 \DC(q_0,q_1) g_0^{-1} 
      \stext{ for all } q_0,q_1\in Q.
    \end{equation*}
  \end{enumerate}
\end{proposition}

\begin{proof}
  See Lemma 3.2 and Theorems 3.4 and 4.4
  in~\cite{ar:fernandez_zuccalli-a_geometric_approach_to_discrete_connections_on_principal_bundles}.
\end{proof}

In what follows we use several notions that are reviewed in the
Appendix (Section~\ref{sec:appendix}). For instance, fiber bundles and
their maps are introduced in Definitions~\ref{def:fiber_bundle}
and~\ref{def:bundle_map}, while the action of a Lie group on a fiber
bundle is introduced in
Definition~\ref{def:group_acts_on_fiber_bundle}.

When a Lie group $\SG$ acts on the fiber bundle $(E,M,\phi,F)$ and
$F_2$ is a right $\SG$-manifold, it is possible to construct an
\jdef{associated bundle} on $M/\SG$ with total space
$(E\times F_2)/\SG$ and fiber $F\times F_2$. The special case when
$F_2 = \SG$ acting on itself by $r_g(h):=g^{-1} h g$ is known as the
\jdef{conjugate bundle} and is denoted by $\ti{\SG}_E$ (see
Example~\ref{ex:extended_associated_bundle-fiber_bundle}).

\begin{proposition}\label{prop:generalized_isomorphisms_associated_to_DC}
  Let $\SG$ be a Lie group that acts on the fiber bundle
  $(E,M,\phi,F)$ and $\DC$ be a discrete connection on the principal
  $\SG$-bundle $\pi^{M,\SG}:M\rightarrow M/\SG$. Define\footnote{As we
    mentioned in Remark~\ref{rem:Ad_not_everywhere_but_ignored}, the
    discrete connection $\DC$ need not be defined on $M\times M$ but,
    rather, on an open subset. This restricts the domain of
    $\ti{\Phi}_\DC$ and $\ti{\Psi}_\DC$ to appropriate open subsets,
    where the results of
    Proposition~\ref{prop:generalized_isomorphisms_associated_to_DC}
    hold. Still, we ignore this point and keep working as if $\DC$
    were globally defined in order to avoid a more involved notation.}
  $\ti{\Phi}_\DC:E\times M\rightarrow E\times \SG\times (M/\SG)$ and
  $\ti{\Psi}_\DC:E\times\SG\times(M/\SG)\rightarrow E\times M$ by
  \begin{gather*}
    \ti{\Phi}_\DC(\epsilon,m):=(\epsilon,\DC(\phi(\epsilon),m),\pi^{M,\SG}(m)),\\
    \ti{\Psi}_\DC(\epsilon,w,r):=(\epsilon,l^M_w(\HLds{\phi(\epsilon)}(r))).
  \end{gather*}
  Then, $\ti{\Phi}_\DC$ and $\ti{\Psi}_\DC$ are smooth functions,
  inverses of each other. If we view $E\times M$ and
  $E\times \SG\times (M/\SG)$ as fiber bundles over $M$ via
  $\phi\circ p_1$, then $\ti{\Phi}_\DC$ and $\ti{\Psi}_\DC$ are bundle
  maps (over the identity). In addition, if we consider the left
  $\SG$-actions $l^{E\times M}$ and $l^{E\times\SG\times(M/\SG)}$
  defined by
  \begin{equation*}
    l^{E\times M}_g(\epsilon,m):=(l^E_g(\epsilon),l^M_g(m)) \stext{ and } 
    l^{E\times\SG\times(M/\SG)}_g(\epsilon,w,r):=(l^E_g(\epsilon),l^\SG_g(w),r),
  \end{equation*}
  then, $\ti{\Phi}_\DC$ and $\ti{\Psi}_\DC$ are $\SG$-equivariant, so
  they induce diffeomorphisms $\Phi_\DC:(E\times M)/\SG\rightarrow
  \ti{\SG}_E \times (M/\SG)$ and $\Psi_\DC:\ti{\SG}_E \times (M/\SG)
  \rightarrow (E\times M)/\SG$.
\end{proposition}

\begin{proof}
  Being composition of smooth functions (see
  part~\ref{it:dcs_and_lifts-smooth} in
  Proposition~\ref{prop:dcs_and_lifts}), $\ti{\Phi}_\DC$ and
  $\ti{\Psi}_\DC$ are smooth; direct computations involving
  part~\ref{it:dcs_and_lifts-GxG} of
  Proposition~\ref{prop:dcs_and_lifts} show that $\ti{\Phi}_\DC$ and
  $\ti{\Psi}_\DC$ are inverses of each other. Using
  part~\ref{it:dcs_and_lifts-equivariance} from
  Proposition~\ref{prop:dcs_and_lifts} it is easy to verify the
  $\SG$-equivariance of $\ti{\Phi}_\DC$. Being
  $\ti{\Psi}_\DC = (\ti{\Phi}_\DC)^{-1}$, its $\SG$-equivariance
  follows.  The last part of the statement is derived from
  Corollary~\ref{cor:quotient_maps-G_equivariant_map}.
\end{proof}

\begin{remark}
  Notice that $(E\times M)/\SG$ can be seen as a fiber bundle over
  $M/\SG$ via $\check{\phi\circ p_1}$, corresponding to the associated
  fiber bundle $\ti{M}_E$ constructed in
  Example~\ref{ex:extended_associated_bundle-fiber_bundle}. Similarly
  $\ti{\SG}_E\times(M/\SG)$ is a fiber bundle over $M/\SG$ via
  $\check{\phi\circ p_1}$. In this context, $\Phi_\DC$ and $\Psi_\DC$
  are bundle maps (over the identity).
\end{remark}

After
Proposition~\ref{prop:generalized_isomorphisms_associated_to_DC}, we
have the commutative diagram
\begin{equation}
  \label{eq:diagram_ExM_to_reduced}
  \xymatrix{ {E\times M} \ar[r]^(.4){\ti{\Phi_\DC}}_(.4){\sim} \ar[d]_{\pi^{E\times M,\SG}} 
    \ar[dr]^(.55){\Upsilon_\DC} & 
    {(E\times \SG) \times (M/\SG)} \ar[d]^{(\pi^{E\times \SG,\SG}\circ p_1)\times p_2} \\
    {(E\times M)/\SG} \ar[r]_(.45){\Phi_\DC}^(.45){\sim} & {\ti{\SG}_E\times (M
      /\SG)}
  }
\end{equation}
consisting of manifolds and smooth maps (or bundle maps, understanding
that the top row are bundles over $M$ and the bottom row are bundles
over $M/\SG$). In~\eqref{eq:diagram_ExM_to_reduced}, we have defined
\begin{equation}
  \label{eq:Upsilon_DC-def}
  \Upsilon_\DC:=\Phi_\DC\circ \pi^{E\times M,\SG}
  = ((\pi^{E\times \SG,\SG}\circ p_1)\times p_2)\circ \ti{\Phi_\DC}.
\end{equation}

\begin{lemma}\label{le:Upsilon_DC_is_ppal_bundle}
  Let $\SG$ act on the fiber bundle $(E,M,\phi,F)$ and $\DC$ be a
  discrete connection on the principal $\SG$-bundle
  $\pi^{M,\SG}:M\rightarrow M/\SG$. Then, $\Upsilon_\DC:E\times
  M\rightarrow (\ti{\SG}_E\times (M/\SG))$ defined
  by~\eqref{eq:Upsilon_DC-def} is a principal $\SG$-bundle.
\end{lemma}

\begin{proof}
  Since $\Phi_\DC$ is a diffeomorphism and $\pi^{E\times M,\SG}$ is a
  surjective submersion, $\Upsilon_\DC$ is also a surjective
  submersion. Also, as
  $\Upsilon_\DC^{-1}(\Upsilon_\DC(\epsilon_0,m_1))=l^{E\times
    M}_\SG(\epsilon_0,m_1)$, by
  Theorem~\ref{thm:ppal_bundles_via_embedding}, we conclude that
  $(E\times M,\ti{\SG}_E\times (M/\SG),\Upsilon_\DC,\SG)$ is a
  principal $\SG$-bundle.
\end{proof}

When $\SG$ acts on the fiber bundle $(E,M,\phi,F)$ and $\DC$ is a
discrete connection on the principal $\SG$-bundle
$\pi^{M,\SG}:M\rightarrow M/\SG$, we have the following commutative
diagram involving the conjugate bundle $\ti{\SG}_E$.
\begin{equation}\label{eq:ExM_and_tiGxM/G}
  \xymatrix{
    {E} \ar[d]_{\pi^{M,\SG}\circ \phi} & 
    {E\times M} \ar[l]_{p_1} \ar[dr]^{\Upsilon_\DC} \ar[r]^(.35){\ti{\Phi_\DC}}
    & {(E\times \SG)\times (M/\SG)} \ar[d]^{(\pi^{E\times\SG,\SG} \circ p_1)\times p_2}\\
    {M/\SG} & {} & {\ti{\SG}_E \times (M/\SG)} \ar[ll]^{p^{M/\SG}\circ p_1}\\
  }
\end{equation}


\section{Discrete Lagrange--Poincar\'e systems}
\label{sec:DLPS}

A discrete mechanical system (DMS) as
in~\cite{ar:marsden_west-discrete_mechanics_and_variational_integrators}
is a pair $(Q,L_d)$ where $Q$ is a finite dimensional manifold known
as the \jdef{configuration space} and $L_d:Q\times Q\rightarrow \R$ is
a smooth function called the \jdef{discrete lagrangian}. Trajectories
of such a system are critical points of an action function determined
by $L_d$.

In this section we introduce an extended notion of DMS as a dynamical
system whose dynamics arises from a variational principle. In
addition, we find the corresponding equations of motion. In
Section~\ref{sec:categorical_formulation}, we formulate a categorical
framework that contains the extended systems.


\subsection{Discrete Lagrange--Poincar\'e systems and dynamics}
\label{sec:generalized_systems_and_dynamics}

The reduction procedure introduced
in~\cite{ar:fernandez_tori_zuccalli-lagrangian_reduction_of_discrete_mechanical_systems}
and reviewed in the unconstrained situation later, in
Section~\ref{sec:reduced_system_associated_to_a_symmetric_mechanical_system},
has a shortcoming in that, in most cases, when applied to a DMS, the
resulting dynamical system is not a DMS. The main objective of this
paper is to overcome this problem by considering a larger class of
discrete mechanical systems that is closed by the reduction
procedure. In order to define the larger class of DMSs we will
consider more general ``discrete velocity'' phase spaces than $Q\times
Q$; concretely, we will consider spaces of the form $E\times M$, where
$\phi:E\rightarrow M$ is a fiber bundle. Furthermore, we will consider
discrete time dynamical systems on such spaces, whose dynamics will be
defined using a lagrangian function and a variational principle. In
this section we study the extended discrete velocity phase spaces and
discrete lagrangian systems on them.

The motivation for the notion of extended discrete velocity phase
space that we consider in this paper comes from the type of space
obtained by the reduction process introduced
in~\cite{ar:fernandez_tori_zuccalli-lagrangian_reduction_of_discrete_mechanical_systems}. There,
the reduced space associated to a discrete system on $Q\times Q$ with
symmetry group $\SG$ is the space
$(Q/\SG)\times (Q/\SG)\times_{Q/\SG} \ti{\SG}$, that is a fibered
product of the pair bundle $(Q/\SG)\times(Q/\SG)$ and the fiber bundle
$\ti{\SG}\rightarrow Q/\SG$, where $\ti{\SG}=\ti{\SG}_E$ for $E$ the
fiber bundle $id_Q:Q\rightarrow Q$ (see
Example~\ref{ex:extended_associated_bundle-fiber_bundle}). This space
is not a standard space for a DMS due to the presence of
$\ti{\SG}$. Therefore, it seems reasonable to enlarge the class of
spaces to be considered by looking at spaces that are the fibered
product of a pair bundle $M\times M$ and a fiber bundle
$E\rightarrow M$. In fact, for continuous mechanical systems, this is
the approach
of~\cite{bo:cendra_marsden_ratiu-lagrangian_reduction_by_stages},
where their extended velocity phase space is of the form $TQ\oplus V$
and $V$ is a vector bundle over $Q$. Yet, we will consider a minor
variation of the preceding idea: instead of $(M\times M)\times_M E$,
we will consider $E\times M$ that, as fiber bundles over $M$ (by
$\phi\circ p_1$ in the second space) are isomorphic. The technical
advantage of using this last space is that it is easier to work with a
product manifold rather than with a fibered product.

Given a fiber bundle $\phi:E\rightarrow M$ we will denote
$C'(E):=E\times M$, seen as a fiber bundle over $M$ by
$\phi\circ p_1$. Similarly, we define the \jdef{discrete second order
  manifold}
$C''(E):=(E\times M)\times_{p_2,(\phi\circ p_1)}(E\times M)$ that we
view as a fiber bundle over $M$ via the map induced by $p_2$.

\begin{remark}\label{rem:isomorphism_C''(E)_with_ExExM}
  Given a fiber bundle $\phi:E\rightarrow M$, the second order
  manifold $C''(E)\rightarrow M$ is isomorphic as a fiber bundle to
  the fiber bundle $\phi\circ p_2:E\times E\times M\rightarrow M$ via
  $F_E((\epsilon_0,m_1),(\epsilon_1,m_2)) :=
  (\epsilon_0,\epsilon_1,m_2)$.
\end{remark}

\begin{definition}
  Let $\phi:E\rightarrow M$ be a fiber bundle. A \jdef{discrete path}
  in $C'(E)$ is a collection $(\epsilon_\cdot,m_\cdot) =
  ((\epsilon_0,m_1),\ldots,(\epsilon_{N-1},m_N))$ where $((\epsilon_k,
  m_{k+1}),(\epsilon_{k+1},m_{k+2}))\in C''(E)$ for $k=0,\ldots,N-2$.
\end{definition}

\begin{definition}
  Let $\phi:E\rightarrow M$ be a fiber bundle. An \jdef{infinitesimal
    variation chaining map} $\IVCM$ on $E$ is a homomorphism of vector
  bundles over $\ti{p_1}$, according to the following commutative
  diagram (of vector bundles)
  \begin{equation*}
    \xymatrix{
      {TE} \ar[d] & {\ti{p_3}^*(TE)} \ar[l] \ar[d] \ar[r]^{\IVCM} & 
      {\ker(d\phi)} \ar[d] \ar@{^{(}->}[r] & {TE} \ar[dl]\\
      {E} & {C''(E)} 
      \ar[l]^{\ti{p_3}} \ar[r]_{\ti{p_1}} & {E} & {}
    }
  \end{equation*}
  where $\ti{p_1}((\epsilon_0,m_1),(\epsilon_1,m_2)):=\epsilon_0$ and
  $\ti{p_3}((\epsilon_0,m_1),(\epsilon_1,m_2)):=\epsilon_1$. Notice
  that since $\phi:E\rightarrow M$ is a fiber bundle, $\ker(d\phi)$
  has constant rank, so it defines a vector subbundle of
  $TE\rightarrow E$.
\end{definition}

\begin{definition}
  Let $\phi:E\rightarrow M$ be a fiber bundle. A \jdef{discrete
    Lagrange--Poincar\'e system} (DLPS) over $E$ is a triple
  $\mathcal{M} := (E,L_d,\IVCM)$ where $L_d:C'(E)\rightarrow \R$ is a
  smooth function and $\IVCM$ is an infinitesimal variation chaining
  map on $E$.
\end{definition}

\begin{definition}\label{def:infinitesimal_variation_with_fixed_endpoints}
  Let $(E,L_d,\IVCM)$ be a DLPS and
  $(\epsilon_\cdot,m_\cdot) =
  ((\epsilon_0,m_1),\ldots,(\epsilon_{N-1},m_N))$
  be a discrete path in $C'(E)$. An \jdef{infinitesimal variation} on
  $(\epsilon_\cdot,m_\cdot)$ is a tangent vector
  $(\delta \epsilon_\cdot, \delta m_\cdot) = ((\delta \epsilon_0,
  \delta m_1),\ldots,(\delta \epsilon_{N-1}, \delta m_N))\in
  T_{(\epsilon_\cdot,m_\cdot)} C'(E)^N$ such that
  \begin{equation}\label{eq:gdms_variation_free-def}
    \delta m_{k} = d\phi(\epsilon_{k})(\delta \epsilon_{k}) \stext{ for }
    k=1,\ldots,N-1 
  \end{equation}
  or, equivalently, that
  $((\delta \epsilon_{k-1},\delta m_k),(\delta \epsilon_k, \delta
  m_{k+1})) \in T C''(E)$
  for $ k=1,\ldots,N-1$.  An \jdef{infinitesimal variation on
    $(\epsilon_\cdot,m_\cdot)$ with fixed endpoints} is an
  infinitesimal variation $(\delta \epsilon_\cdot, \delta m_\cdot)$ on
  $(\epsilon_\cdot,m_\cdot)$ such that
  \begin{equation}\label{eq:gdms_variation_fep-def}
    \begin{split}
      \delta m_N =& 0,\\
      \delta \epsilon_{N-1} =& \ti{\delta \epsilon_{N-1}},\\
      \delta \epsilon_k =& \ti{\delta \epsilon_k} +
      \IVCM((\epsilon_{k},m_{k+1}),(\epsilon_{k+1},m_{k+2}))(\ti{\delta
      \epsilon_{k+1}}), \stext{ if } k=1,\ldots,N-2,\\
      \delta \epsilon_0 =&
      \IVCM((\epsilon_{0},m_{1}),(\epsilon_{1},m_{2}))(\ti{\delta
      \epsilon_{1}}),
    \end{split}
  \end{equation}
  where $\ti{\delta \epsilon_k} \in T_{\epsilon_k}E$ is arbitrary for
  $k=1,\ldots, N-1$.
\end{definition}

\begin{remark}
  The name ``infinitesimal variation with fixed endpoints'' is not
  entirely accurate in
  Definition~\ref{def:infinitesimal_variation_with_fixed_endpoints}. Certainly,
  $\delta m_N=0$ means that $m_N$ remains fixed. On the other hand,
  $\delta \epsilon_0$ does not necessarily vanish, but neither it is
  arbitrary, as it is determined by $\ti{\delta \epsilon_1}$ through
  $\IVCM$. As the $\ti{\delta \epsilon_k}$ are arbitrary for
  $k=1,\ldots,N-1$, given $\delta \epsilon_k$ for $k=1,\ldots, N-1$, it
  is possible to find $\ti{\delta \epsilon_k}$ for $k=1,\ldots,N-1$
  such that $\delta \epsilon_{N-1} = \ti{\delta \epsilon_{N-1}}$ and
  $\delta \epsilon_k = \ti{\delta \epsilon_k} +
  \IVCM((\epsilon_{k},m_{k+1}),(\epsilon_{k+1},m_{k+2}))(\ti{\delta
    \epsilon_{k+1}})$
  for all $k=1,\ldots,N-2$. In this case, $\delta \epsilon_0$ turns
  out to be a function of all
  $\delta \epsilon_1,\ldots,\delta \epsilon_{N-1}$.
\end{remark}

\begin{definition}\label{def:dynamics-generalized}
  Let $\mathcal{M}=(E,L_d,\IVCM)$ be a DLPS. The \jdef{discrete
    action} of $\mathcal{M}$ is a function from the space of all
  discrete curves on $C'(E)$ to $\R$ defined by
  $S_d(\epsilon_\cdot,m_\cdot):=\sum_{k=0}^{N-1}
  L_d(\epsilon_k,m_{k+1})$. A \jdef{trajectory} of $\mathcal{M}$ is a
  discrete curve $(\epsilon_\cdot,m_\cdot)$ in $C'(E)$ such that
  \begin{equation}\label{eq:critical_action_condition}
    dS_d(\epsilon_\cdot,m_\cdot)(\delta \epsilon_\cdot,\delta m_\cdot) =0
  \end{equation}
  for all infinitesimal variations $(\delta \epsilon_\cdot,\delta
  m_\cdot)$ on $(\epsilon_\cdot,m_\cdot)$ with fixed endpoints, that
  is, satisfying~\eqref{eq:gdms_variation_free-def}
  and~\eqref{eq:gdms_variation_fep-def}.
\end{definition}

The following Proposition characterizes the trajectories of a DLPS
in terms of (algebraic) equations.

\begin{proposition}\label{prop:eqs_of_motion_gdms}
  Let $\mathcal{M} = (E,L_d,\IVCM)$ be a DLPS and
  $(\epsilon_\cdot,m_\cdot)$ be a discrete path in $C'(E)$. Then,
  $(\epsilon_\cdot,m_\cdot)$ is a trajectory of $\mathcal{M}$ if and
  only if for all $k=1,\ldots,N-1$,
  \begin{equation}
    \label{eq:equation_of_motion-generalized-in_epsilon_m}
    \begin{split}
      D_1L_d(\epsilon_k,m_{k+1}) +& D_2L_d(\epsilon_{k-1},m_k)\circ 
      d\phi(\epsilon_k) \\+&D_1L_d(\epsilon_{k-1},m_{k}) \circ
      \IVCM((\epsilon_{k-1},m_{k}),(\epsilon_{k},m_{k+1})) = 0
    \end{split}
  \end{equation}
  in $T^*_{\epsilon_k}E$, where $D_j$ denotes the restriction to the
  $j$-th component of the exterior differential on a Cartesian
  product.
\end{proposition}

\begin{proof}
  Equation~\eqref{eq:equation_of_motion-generalized-in_epsilon_m} is
  obtained from the standard variational computation of
  $dS_d(\epsilon_\cdot, m_\cdot)(\delta \epsilon_\cdot, \delta
  m_\cdot)$, taking into account that the fixed endpoint infinitesimal
  variations $(\delta \epsilon_\cdot, \delta m_\cdot)$ over
  $(\epsilon_\cdot, m_\cdot)$
  satisfy~\eqref{eq:gdms_variation_fep-def} for arbitrary $\ti{\delta
    \epsilon_k}\in T_{\epsilon_k}E$.
\end{proof}

Next, we introduce sufficient conditions for the existence of a flow
on a DLPS $\mathcal{M} = (E,L_d,\IVCM)$. Consider the commutative
diagram (of smooth maps)
\begin{equation*}
  \xymatrix{
    {} & {T^*E} \ar[d]\\
    {E\times E\times M} \ar[ur]^{\mathcal{E}} \ar[r]_(.65){p_2} & {E}
  }
\end{equation*}
where
\begin{equation*}
  \begin{split}
    \mathcal{E}(\epsilon_0,\epsilon_1,m_2) := &D_1L_d(\epsilon_1,m_{2}) +
    D_1L_d(\epsilon_{0},\phi(\epsilon_1)) \circ 
    \IVCM((\epsilon_{0},\phi(\epsilon_1)),(\epsilon_{1},m_{2})) \\
    &+ D_2L_d(\epsilon_{0},\phi(\epsilon_1)) \circ d\phi(\epsilon_1).
  \end{split}
\end{equation*}
Notice that all trajectories $((\epsilon_0,m_1),(\epsilon_1,m_2))$ of
$\mathcal{M}$ satisfy $\mathcal{E}(\epsilon_0,\epsilon_1,m_2) =
0_{\epsilon_1} \in T_{\epsilon_1}^*E$. Conversely, given
$(\epsilon_0,\epsilon_1,m_2)\in E\times E\times M$ such that
$\mathcal{E}(\epsilon_0,\epsilon_1,m_2)=0_{\epsilon_1}$, then
$((\epsilon_0,\phi(\epsilon_1)),(\epsilon_1,m_2))$ is a trajectory of
$\mathcal{M}$.

Let $\mathcal{Z}\subset T^*E$ be the image of the zero section of the
canonical projection $T^*E\rightarrow E$. It is easy to check that
$\mathcal{Z}\subset T^*E$ is an embedded submanifold.

\begin{proposition}\label{prop:existence_trajectories_DLPS}
  Let $\mathcal{M}$, $\mathcal{E}$ and $\mathcal{Z}$ be as above.
  \begin{enumerate}
  \item \label{it:existence_trajectories_DLPS-submanifold} Assume that
    $(\epsilon_0,\epsilon_1,m_2)\in E\times E\times M$ is such that
    $\mathcal{E}(\epsilon_0,\epsilon_1,m_2)=0_{\epsilon_1}$ and that
    $\im(d\mathcal{E}(\epsilon_0,\epsilon_1,m_2)) +
    T_{0_{\epsilon_1}}\mathcal{Z} = T_{0_{\epsilon_1}} T^*E$. Then,
    there is an open subset $U\subset E\times E\times M$ with
    $(\epsilon_0,\epsilon_1,m_2)\in U$ and such that $\mathcal{E}_U :=
    U\cap \mathcal{E}^{-1}(\mathcal{Z})$ is an embedded submanifold of
    $E\times E\times M$ with $\dim(\mathcal{E}_U) = \dim(E) +
    \dim(M)$.
  \item \label{it:existence_trajectories_DLPS-flow} Consider the
    smooth map $p_1\times (\phi\circ p_2) : E\times E\times
    M\rightarrow C'(E)$. In addition to what was assumed in
    part~\ref{it:existence_trajectories_DLPS-submanifold}, suppose that
    \begin{enumerate}[(i)]
    \item \label{it:existence_trajectories_DLPS-1st_proj} $d(p_1\times
      (\phi\circ p_2))|_{\mathcal{E}_U}(\epsilon_0,\epsilon_1,m_2)
      \in\hom(T_{(\epsilon_0,\epsilon_1,m_2)}\mathcal{E}_U,
      T_{(\epsilon_0,\phi(\epsilon_1))}C'(E))$ is injective and 
    \item \label{it:existence_trajectories_DLPS-2nd_proj} $d(p_2\times
      p_3)|_{\mathcal{E}_U}(\epsilon_0,\epsilon_1,m_2)
      \in\hom(T_{(\epsilon_0,\epsilon_1,m_2)}\mathcal{E}_U,T_{(\epsilon_1,m_2)}C'(E))$
      is injective.
    \end{enumerate}
    Then, there are open sets $V_1,V_2\subset C'(E)$ such that
    $(\epsilon_0,\phi(\epsilon_1)) \in V_1$ and $(\epsilon_1,m_2)\in
    V_2$ and a diffeomorphism $F_{\DLPS}:V_1\rightarrow V_2$ such that
    $F_{\DLPS}(\epsilon_0,\phi(\epsilon_1)) = (\epsilon_1,m_2)$ and,
    for all $(\epsilon_0',m_1')\in V_1$,
    $((\epsilon_0',m_1'),F_{\DLPS}(\epsilon_0',m_1'))$ is a trajectory
    of $\mathcal{M}$.
  \end{enumerate}
\end{proposition}

\begin{proof}
  Part~\ref{it:existence_trajectories_DLPS-submanifold} follows
  immediately from the transversality argument on page 28
  of~\cite{bo:Guillemin-Pollack-differential_topology}, applied to the
  point $(\epsilon_0,\epsilon_1,m_2)$. Notice that, as
  $(\epsilon_0,\epsilon_1,m_2) \in \mathcal{E}_U$, it is not the empty
  set.

  Let
  $P:=(p_1\times(\phi\circ p_2))|_{\mathcal{E}_U}
  :\mathcal{E}_U\rightarrow C'(E)$.
  As $\dim(\mathcal{E}_U) = \dim(C'(E))$,
  condition~\ref{it:existence_trajectories_DLPS-1st_proj} implies that
  $dP(\epsilon_0,\epsilon_1,m_2)$ is an isomorphism and, consequently,
  $P$ is a local diffeomorphism at
  $(\epsilon_0,\epsilon_1,m_2)$. Hence, there are open sets
  $V_1\subset C'(E)$ and $V_2'\subset \mathcal{E}_U$ such that
  $(\epsilon_0,\phi(\epsilon_1))\in V_1$ and
  $(\epsilon_0,\epsilon_1,m_2)\in V_2'$ where $P|_{V_2'}$ is a
  diffeomorphism onto $V_1$. In addition, as
  $\dim(\mathcal{E}_U) = \dim(C'(E))$,
  condition~\ref{it:existence_trajectories_DLPS-2nd_proj} implies that
  $d(p_2\times p_3)|_{\mathcal{E}_U}(\epsilon_0,\epsilon_1,m_2)$ is a
  local diffeomorphism at $(\epsilon_0,\epsilon_1,m_2)$ so that
  (eventually shrinking $V_2'$)
  $V_2:=(p_2\times p_3)(V_2')\subset C'(E)$ is open and
  $(p_2\times p_3)|_{V_2'}:V_2'\rightarrow V_2$ is a diffeomorphism.

  Let $F_{\DLPS}:V_1\rightarrow V_2$ be the diffeomorphism
  $F_{\DLPS}:=(p_2\times p_3)\circ (P|_{V_2'})^{-1}$. By construction,
  $F_{\DLPS}(\epsilon_0,\phi(\epsilon_1)) = (\epsilon_1,
  m_2)$.
  Furthermore, for $(\epsilon_0',m_1')\in V_1$, if we let
  $(\epsilon_1',m_2'):=F_{\DLPS}(\epsilon_0',m_1')$, it follows
  readily that $((\epsilon_0',m_1'),(\epsilon_1',m_2')) \in C''(E)$ is
  a trajectory of $\mathcal{M}$.
\end{proof}

\begin{definition}
  The function $F_{\mathcal{M}}$ that appears in
  part~\ref{it:existence_trajectories_DLPS-flow} of
  Proposition~\ref{prop:existence_trajectories_DLPS} is the
  \jdef{discrete Lagrangian flow} of $\mathcal{M}$.
\end{definition}

\begin{remark}\label{rem:sub_trajectories}
  If $(\epsilon_\cdot, m_\cdot) = ((\epsilon_0,m_1),\ldots,
  (\epsilon_{N-1},m_N))$ is a trajectory of the DLPS $\mathcal{M}$,
  then it
  satisfies~\eqref{eq:equation_of_motion-generalized-in_epsilon_m} for
  $k=0, \ldots, N-1$. But then, if $j=0,\ldots,N-2$,
  $((\epsilon_j,m_{j+1}),(\epsilon_{j+1},m_{j+2}))$ also
  satisfies~\eqref{eq:equation_of_motion-generalized-in_epsilon_m}
  (for $k=j,j+1$) and, by Proposition~\ref{prop:eqs_of_motion_gdms},
  is also a trajectory of $\mathcal{M}$. That is, contiguous points of
  a trajectory of $\mathcal{M}$, form a trajectory of $\mathcal{M}$.
\end{remark}

The following example shows how a DMS can be seen as a DLPS.

\begin{example}\label{ex:mechanical_system_as_generalized_mechanical_system}
  Let $(Q,L_d)$ be a DMS. Define the fiber bundle
  $\FBM:\FBTS\rightarrow\FBBS$ by $id_Q:Q\rightarrow Q$, so that
  $\Lagr$ defines a lagrangian function on $C'(id_Q:Q\rightarrow Q) =
  Q\times Q$. Next, let $\IVCM((q_{k-1},q_k),(q_k,q_{k+1}))(\delta
  q_k) := 0$ for all $\delta q_k\in T_{q_k}Q$.  We define the DLPS
  $\DLPS:= (\FBTS, \Lagr, \IVCM)$.

  Discrete paths $(\epsilon_\cdot,m_\cdot)$ of $\DLPS$ are, in the
  current context, the same as discrete paths $q_\cdot$ in
  $Q$\footnote{\label{it:discrete_path_in_manifold}A discrete path
    $x_\cdot$ in a manifold $X$ is an element of the Cartesian product
    $X^N$, for some $N\in\N$.}. Such discrete paths are trajectories
  of $\DLPS$ if and only if they
  satisfy~\eqref{eq:equation_of_motion-generalized-in_epsilon_m} that,
  in this case, becomes 
  \begin{equation}
    \label{eq:DEL_for_DMS}
    D_1\Lagr(q_k,q_{k+1}) + D_2\Lagr(q_{k-1},q_{k})=0
  \end{equation}
  for all $k$, that is the usual discrete Euler--Lagrange equation
  (see equation (1.3.3)
  in~\cite{ar:marsden_west-discrete_mechanics_and_variational_integrators})
  that characterizes the trajectories of $(Q,L_d)$. Hence, all DMSs
  can be seen as DLPSs whose dynamics coincide with those of the
  original systems.
\end{example}

\begin{remark}
  As, by
  Example~\ref{ex:mechanical_system_as_generalized_mechanical_system},
  all DMSs are DLPSs, we can specialize
  Proposition~\ref{prop:existence_trajectories_DLPS} to the case of a
  DMS $(Q,L_d)$. A simple analysis provides the following
  statement. Let $(q_0,q_1,q_2)\in Q\times Q\times Q$ be a solution
  of~\eqref{eq:DEL_for_DMS} (for $k=1$) such that $L_d$ is
  regular\footnote{Regularity at $(q_0,q_1)$ means that, with respect
    to local coordinates $q_j^a$ (for $j=0,1$ and
    $a=1,\ldots,n:=\dim(Q)$) neat $q_0$ and $q_1$, the matrix
    $\frac{\partial^2 L_d(q_0,q_1)}{\partial q_0^j \partial q_1^k} \in
    \R^{n\times n}$
    be invertible.} at $(q_0,q_1)$ and $(q_1,q_2)$. Then there are
  open sets $V_1, V_2\subset Q\times Q$ with $(q_0,q_1)\in V_1$ and
  $(q_1,q_2)\in V_2$ and a diffeomorphism $F_{L_d}:V_1\rightarrow V_2$
  such that $F_{L_d}(q_0,q_1) = (q_1,q_2)$ and that
  $(q_0',F_{L_d}(q_0',q_1'))$ is a solution of~\eqref{eq:DEL_for_DMS}
  (for $k=1$) for all $(q_0',q_1')\in V_1$.  We emphasize that the
  existence of a trajectory $(q_0,q_1,q_2)$ as a staring point cannot
  be avoided. For example, when $Q=\R$ and
  $L_d(q_0,q_1):= \frac{1}{2}(q_1-q_0)^2 - \eta (q_0+q_1)^3$ for
  $\eta>0$, we have that $L_d$ is regular at $(q_0,q_1)$ and
  $(q_1,q_2)$ for $q_0,q_1,q_2 < -\frac{1}{6\eta}$. But, it is easy to
  check that, if $q_1< -\frac{1}{24\eta}$, there is no trajectory of
  the form $(q_0,q_1,q_2)$.
\end{remark}

The dynamical system obtained by the reduction process of a symmetric
DMS can be seen as a DLPS, as we describe in the following section.

\subsection{Reduced system associated to a symmetric discrete
  mechanical system}
\label{sec:reduced_system_associated_to_a_symmetric_mechanical_system}

We say that the Lie group $\SG$ is a symmetry group of the DMS
$(Q,L_d)$ if $\SG$ acts on $Q$ in such a way that the quotient mapping
$\pi^{Q,\SG}:Q\rightarrow Q/\SG$ is a principal $\SG$-bundle and
$L_d\circ l^{Q\times Q}_g = L_d$ for all $g\in \SG$. Given such a
system we can construct a discrete time dynamical system called the
\jdef{reduced system} whose dynamics captures the essential behavior
of the original dynamics. First we review the construction of the
reduced system and, then, compare the dynamics of the reduced to that
of the unreduced system. After that, we prove that the reduced system
can be seen as a DLPS with the same trajectories.


Given a discrete connection $\DC$ on the principal $\SG$-bundle
$\pi^{Q,\SG}:Q\rightarrow Q/\SG$, we can specialize the commutative
diagram~\eqref{eq:diagram_ExM_to_reduced} to the case where
$\phi:E\rightarrow M$ is $id_Q:Q\rightarrow Q$:
\begin{equation*}
  \xymatrix{
    {Q\times Q} \ar[r]^(.4){\ti{\Phi_\DC}} \ar[d]_{\pi^{Q\times Q,\SG}} 
    \ar[rd]^{\Upsilon_\DC} &
    {(Q\times \SG)\times (Q/\SG)} \ar[d]^{(\pi^{Q\times \SG,\SG} \circ p_1)\times p_2}\\
    {(Q\times Q)/\SG} \ar[r]_{\Phi_{\DC}} & {\ti{\SG}\times (Q/\SG)}
  }
\end{equation*}
where $\ti{\SG} = (Q\times \SG)/\SG$ with $\SG$ acting on $Q$ by $l^Q$
and on $\SG$ by conjugation and, explicitly,
\begin{equation}\label{eq:Upsilon_reduced_dms}
  \Upsilon_\DC(q_0,q_1) = (\pi^{Q\times
    \SG,\SG}(q_0,\DC(q_0,q_1)),\pi^{Q,\SG}(q_1)).
\end{equation}

By the $\SG$-invariance of $L_d$, there is a well defined map
$\check{L}_d:\ti{\SG}\times (Q/\SG)\rightarrow \R$ such that
$\check{L}_d(v_0,r_1) = L_d(q_0,q_1)$ whenever $(q_0,q_1)\in Q\times
Q$ satisfies $(v_0,r_1) = \Upsilon_\DC(q_0,q_1)$. The action
associated to $\check{L}_d$ is $\check{S}_d(v_\cdot,r_\cdot) := \sum_k
\check{L}_d(v_k,r_{k+1})$.

The following result
from~\cite{ar:fernandez_tori_zuccalli-lagrangian_reduction_of_discrete_mechanical_systems}\footnote{In
  fact, Theorem~\ref{thm:four_point_theorem-unconstrained-QxQ} here is
  part of Theorem 5.11
  in~\cite{ar:fernandez_tori_zuccalli-lagrangian_reduction_of_discrete_mechanical_systems}, specialized to the unconstrained case, and where we have adapted
  the notation slightly to match the one used in the present
  paper. } relates the dynamics of the original system to a
variational principle for a system on $\ti{\SG}\times (Q/\SG)$.

\begin{theorem} \label{thm:four_point_theorem-unconstrained-QxQ} Let
  $\SG$ be a symmetry group of the DMS $(Q,L_d)$. Fix a discrete
  connection $\DC$ on the principal $\SG$-bundle
  $\pi^{Q,\SG}:Q\rightarrow Q/\SG$. Let $q_\cdot$ be a discrete path
  in $Q$, $r_k:=\pi^{Q,\SG}(q_k)$, $w_k:=\DC(q_k,q_{k+1})$ and
  $v_k:=\pi^{Q\times\SG,\SG}(q_k,w_k)$ be the corresponding discrete
  paths in $Q/\SG$, $\SG$ and $\ti{\SG}$ (see
  footnote~\ref{it:discrete_path_in_manifold}). Then, the
  following statements are equivalent.
  \begin{enumerate}
  \item \label{it:var_pple-general} $q_\cdot$ satisfies the
    variational principle $dS_d(q_\cdot)(\delta q_\cdot) = 0$ for all
    vanishing endpoints variations $\delta q_\cdot$ over $q_\cdot$.

  \item \label{it:red_var_pple-general} $d\check{S}_d(r_\cdot,v_\cdot)
    (\delta r_\cdot, \delta v_\cdot) = 0$ for all variations $(\delta
    v_\cdot, \delta r_\cdot)$ such that
    \begin{equation}
      \label{eq:delta_vk_with_delta_qk-def}
      \begin{split}
        (\delta v_k,\delta r_{k+1}) :=
        d\Upsilon_\DC(q_k,q_{k+1})(\delta q_k,\delta q_{k+1})
      \end{split}
    \end{equation}
    for $k=0,\ldots,N-1$ and where $\delta q_\cdot$ is a fixed
    endpoints variation over $q_\cdot$. 
  \end{enumerate}
\end{theorem}

\begin{remark}
  The more general Theorem 5.11
  in~\cite{ar:fernandez_tori_zuccalli-lagrangian_reduction_of_discrete_mechanical_systems}
  requires the additional data of a connection $\CC$ on the principal
  $\SG$-bundle $\pi^{Q,\SG}:Q\rightarrow Q/\SG$. With this additional
  information the variations $\delta q_\cdot$ are decomposed in
  $\CC$-horizontal and $\CC$-vertical parts.
\end{remark}

The reduced system associated to $(Q,L_d)$ is the discrete dynamical
system on $\ti{\SG}\times (Q/\SG)$ whose trajectories are the discrete
paths that satisfy the variational principle stated in
point~\ref{it:red_var_pple-general} of
Theorem~\ref{thm:four_point_theorem-unconstrained-QxQ}.  A DLPS
$\DLPS := (\FBTS, \check{L}_d, \IVCM)$ is associated to this reduced
system; we prove later that the trajectories of both systems
coincide. Define the fiber bundle $\FBM:\FBTS\rightarrow \FBBS$ as
the \jdef{conjugate bundle} $p^{Q/\SG}:\ti{\SG}\rightarrow Q/\SG$,
where $p^{Q/\SG}(\pi^{Q\times \SG,\SG}(q,w)):= \pi^{Q,\SG}(q)$.  The
reduced Lagrangian $\check{L}_d:\ti{\SG}\times (Q/\SG)\rightarrow\R$
defines a real valued function on $C'(\FBTS) = \FBTS\times \FBBS$.

In order to define the infinitesimal variation chaining function, we
consider $\Upsilon_\DC:Q\times Q\rightarrow \ti{\SG}\times (Q/\SG)$
defined by~\eqref{eq:Upsilon_reduced_dms}. Then define
$\IVCM \in\hom(p_3^*(T\ti{\SG}),\ker(dp^{Q/\SG}))$ by
\begin{equation}\label{eq:definition_of_ivcm-QxQ}
  \IVCM((v_{0},r_1),(v_1,r_{2}))(\delta v_1) :=
  D_2(p_1\circ \Upsilon_\DC)(q_{0},q_1)(\delta q_1) \in T_{v_{0}}\ti{\SG}
\end{equation}
where $(q_{0},q_{1},q_{2})$ are such that
$(v_{0},r_{1})=\Upsilon_\DC(q_{0},q_{1})$ and
$(v_1,r_{2})=\Upsilon_\DC(q_1,q_{2})$, and $\delta q_1 \in T_{q_1}Q$
is such that
$D_1(p_1\circ \Upsilon_\DC)(q_{1},q_{2})(\delta q_1) = \delta
v_1$.
Lemma~\ref{le:definition_of_ivcm-QxQ} proves that $\IVCM$ is well
defined.

\begin{lemma}\label{le:definition_of_ivcm-QxQ}
  Let $Q$, $\DC$ and $\Upsilon_\DC$ be as before. Then, the following
  assertions are true.
  \begin{enumerate}
  \item \label{it:definition_of_ivcm-QxQ-iso} For
    $(q_0,q_1)\in Q\times Q$,
    $D_1(p_1\circ \Upsilon_\DC)(q_0,q_1) :
    T_{(q_0,q_1)}(Q\times\{q_1\}) \rightarrow
    T_{\Upsilon_\DC(q_0,q_1)}\ti{\SG}$
    is an isomorphism of vector spaces.
  \item \label{it:definition_of_ivcm-QxQ-well_def} For
    $((v_0,r_{1}),(v_{1},r_{2})) \in C''(\FBTS)$ and
    $\delta v_{1}\in T_{v_{1}}\ti{\SG}$ define
    $\IVCM((v_{0},r_{1}),(v_{1},r_{2}))(\delta v_{1})$
    using~\eqref{eq:definition_of_ivcm-QxQ}. Then, $\IVCM$ is well
    defined. In addition, $\IVCM$ is linear in $\delta v_{1}$.
  \item \label{it:definition_of_ivcm-QxQ-in_ker} For
    $((v_0,r_{1}),(v_{1},r_{2})) \in C''(\FBTS)$ and
    $\delta v_{1}\in T_{v_{1}}\ti{\SG}$ we have
    \begin{equation*}
      dp^{Q/\SG}(v_0)(\IVCM((v_0,r_1),(v_1,r_2))(\delta v_1)) = 0.
    \end{equation*}
  \end{enumerate}
\end{lemma}

We skip the proof of Lemma~\ref{le:definition_of_ivcm-QxQ} as we will
be proving more general statements later: see
point~\ref{it:properties_Upsilon^2-isomorphism} in
Lemma~\ref{le:properties_Upsilon^2} for
point~\ref{it:definition_of_ivcm-QxQ-iso} and
Lemma~\ref{le:reduced_IVCM_is_well_def} for
points~\ref{it:definition_of_ivcm-QxQ-well_def}
and~\ref{it:definition_of_ivcm-QxQ-in_ker}.

Next, we compare discrete trajectories of $\DLPS$ with the reduced
trajectories given by part~\ref{it:red_var_pple-general} of
Theorem~\ref{thm:four_point_theorem-unconstrained-QxQ}. We denote
points in $\FBTS=\ti{\SG}$ with $v$ and in $\FBBS = Q/\SG$ with
$r$. The following result proves that all discrete paths in
$C'(\FBTS)$ arise from discrete paths in $Q$.

\begin{lemma}\label{le:lifiting_of_reduced_paths-DMS}
  Let $(v_\cdot,r_\cdot)$ be a discrete path in
  $C'(\FBTS)$ and $q_0\in Q$ such that $p^{Q/\SG}(v_0)
  = \pi^{Q,\SG}(q_0)$. Then, there exists a unique discrete path in
  $C'(id_Q:Q\rightarrow Q)$ such that $\Upsilon_\DC(q_k,q_{k+1}) =
  (v_k,r_{k+1})$ for all $k=0,\ldots,N-1$.
\end{lemma}

\begin{proof}
  See Proposition~\ref{prop:lifting_reduced_discrete_paths}, that is
  the same result, in a more general context.
\end{proof}

A trajectory $(v_\cdot,r_\cdot) = ((v_0,r_1),\ldots, (v_{N-1},r_N))$
of $\DLPS$ is a pair of discrete paths $v_\cdot$ and $r_\cdot$ such
that $\FBM(v_k) = p^{Q/\SG}(v_k) = r_k$ for $k=1,\ldots,N-1$, and
satisfies $dS_d(v_\cdot,r_\cdot)(\delta v_\cdot,\delta r_\cdot) = 0$
for all infinitesimal variations $(\delta v_\cdot, \delta r_\cdot)$ on
$(v_\cdot, r_\cdot)$ with fixed endpoints. Those infinitesimal
variations are given by~\eqref{eq:gdms_variation_free-def}
and~\eqref{eq:gdms_variation_fep-def}.

In what follows, we fix discrete paths $(v_\cdot,r_\cdot)$ in
$\DLPS$ and $q_\cdot$ in $Q$ such that $(v_k,r_{k+1})
= \Upsilon_\DC(q_k, q_{k+1})$ for all $k$. The following result
compares the infinitesimal variations over $(v_\cdot,r_\cdot)$ in
$\DLPS$ to those coming
from~\eqref{eq:delta_vk_with_delta_qk-def}.

\begin{proposition}\label{prop:inf_var_compar-QxQ_reduced}
  With the notation as above, the following statements are true.
  \begin{enumerate}
  \item \label{it:inf_var_compar-QxQ_reduced-lowering_variations}
    Given a fixed endpoint variation $\delta q_\cdot$ over the
    discrete path $q_\cdot$ in $Q$, the infinitesimal variation
    $(\delta v_\cdot, \delta r_\cdot)$ defined
    by~\eqref{eq:delta_vk_with_delta_qk-def} is an infinitesimal
    variation with fixed endpoints over $(v_\cdot, r_\cdot)$ in
    $\DLPS$.
  \item \label{it:inf_var_compar-QxQ_reduced-lift_variations} Given a
    discrete variation $(\delta v_\cdot, \delta r_\cdot)$ over
    $(v_\cdot, r_\cdot)$ with fixed endpoints, there is a fixed
    endpoints variation $\delta q_\cdot$ over the discrete path
    $q_\cdot$ such that~\eqref{eq:delta_vk_with_delta_qk-def} holds
    for all $k$.
  \end{enumerate}
\end{proposition}

\begin{proof}
  \begin{enumerate}
  \item Let $(\delta v_\cdot, \delta r_\cdot)$ be the variation
    defined by~\eqref{eq:delta_vk_with_delta_qk-def} in terms of
    $\delta q_\cdot$. Let
    $\ti{\delta v_k} := D_1(p_1\circ \Upsilon_\DC)(q_k,q_{k+1})(\delta
    q_k) \in T_{v_k}\ti{\SG}$
    for $k=0,\ldots,N-1$. Direct computations
    using~\eqref{eq:delta_vk_with_delta_qk-def} prove that
    $(\delta v_\cdot, \delta r_\cdot)$
    satisfies~\eqref{eq:gdms_variation_free-def}
    and~\eqref{eq:gdms_variation_fep-def}. Thus, it is an
    infinitesimal variation with fixed endpoints in $\DLPS$ on
    $(v_\cdot,r_\cdot)$.

  \item Write $(\delta v_\cdot, \delta r_\cdot)$ according
    to~\eqref{eq:gdms_variation_free-def}
    and~\eqref{eq:gdms_variation_fep-def} for some vectors
    $\ti{\delta v_k}\in T_{v_k}\ti{\SG}$ and $k=1,\ldots,N-1$. Let
    $\delta q_N:=0 \in T_{q_N}Q$, $\delta q_0:=0 \in T_{q_0}Q$ and,
    for each $k=1,\ldots, N-1$, using
    point~\ref{it:definition_of_ivcm-QxQ-iso} in
    Lemma~\ref{le:definition_of_ivcm-QxQ}, let
    $\delta q_k \in T_{q_{k}}Q$ be such that
    $D_1(p_1\circ \Upsilon_\DC)(q_k,q_{k+1})(\delta q_k) = \ti{\delta
      v_k}$.
    Straightforward computations
    using~\eqref{eq:gdms_variation_free-def}
    and~\eqref{eq:gdms_variation_fep-def} now show that
    $\delta q_\cdot$ as constructed is an infinitesimal variation over
    $q_\cdot$ with fixed endpoints and
    that~\eqref{eq:delta_vk_with_delta_qk-def} holds.
  \end{enumerate}
\end{proof}

\begin{corollary}
  A discrete path $(v_\cdot,r_\cdot)$ is a trajectory of $\DLPS$ if
  and only if it is a trajectory of the reduced system according to
  point~\ref{it:red_var_pple-general} in
  Theorem~\ref{thm:four_point_theorem-unconstrained-QxQ}.
\end{corollary}

Hence, the family of DLPSs contains in a natural way all DMSs as well
as all the dynamical systems obtained by reduction of symmetric DMSs.


\section{Categorical formulation}
\label{sec:categorical_formulation}

In many circumstances it is useful to be able to consider ``maps''
between mechanical systems. One example in the area of interest of
this paper is the reduction process, seen as a map from a symmetric
system to a reduced one. Another example is the comparison of
different reductions of the same symmetric system. More generally, a
symmetry could be seen as a map from a system to itself. A common
framework for considering spaces together with their maps is provided
by constructing a category (see, for
instance,~\cite{bo:cendra_marsden_ratiu-lagrangian_reduction_by_stages}). In
this section we study the basic properties of DLPSs and their
morphisms in this categorical context.

\begin{definition}\label{def:DLPSC}
  We define the \jdef{category of discrete Lagrange--Poincar\'e
    systems} as the category $\DLPSC$ whose objects are DLPSs. Given
  $\mathcal{M}, \mathcal{M}' \in \Ob_{\DLPSC}$ with
  $\mathcal{M} = (E,L_d,\IVCM)$ and $\mathcal{M}' = (E',L_d',\IVCM')$,
  a map $\Upsilon:C'(E)\rightarrow C'(E')$ is in
  $\Mor_{\DLPSC}(\mathcal{M},\mathcal{M}')$ if
  \begin{enumerate}
  \item \label{it:DLPSC_mor-surj_subm} $\Upsilon$ is a surjective submersion,
  \item \label{it:DLPSC_mor-D_1_onto} $D_1(p_1\circ
    \Upsilon)(\epsilon_0,m_1):T_{(\epsilon_0,m_1)}(E\times\{m_1\})\rightarrow
    T_{p_1(\Upsilon(\epsilon_0,m_1))}E'$ is onto for all
    $(\epsilon_0,m_1)\in C'(E)$,
  \item \label{it:DLPSC_mor-D_1_null} $D_1(p_2\circ
    \Upsilon)(\epsilon_0,m_1)=0$ for all $(\epsilon_0,m_1)\in C'(E)$
  \item \label{it:DLPSC_mor-C''_well_def} as maps from $C''(E)$ to $M'$, 
    \begin{equation}
      \label{eq:DLPSC_mor-C''_well_def}
      p_2^{C'(E'),M'}\circ \Upsilon\circ
      p_1^{C''(E),C'(E)} = \phi' \circ p_1^{C'(E'),E'} \circ \Upsilon
      \circ p_2^{C''(E),C'(E)},
    \end{equation}
    where $p_j^{A,B}:A\rightarrow B$ are the maps induced by the
    canonical projections of a Cartesian product onto its factors,
  \item \label{it:DLPSC_mor-L_d} $L_d = L_d' \circ \Upsilon$,
  \item \label{it:DLPSC_mor-IVCM} For all
    $((\epsilon_0,m_1),(\epsilon_1,m_2),\delta \epsilon_1) \in
    p_3^*(TE)$,
    \begin{equation}\label{eq:DLPSC_mor-IVCM}
      \begin{split}
        \IVCM'(\Upsilon^{(2)}((\epsilon_0,m_1),
        &(\epsilon_1,m_2)))(D_1(p_1\circ
        \Upsilon)(\epsilon_1,m_2)(\delta \epsilon_1)) =\\ &d(p_1\circ
        \Upsilon)(\epsilon_0,m_1)(\IVCM((\epsilon_0,m_1),(\epsilon_1,m_2))(\delta
        \epsilon_1),d\phi(\epsilon_1)(\delta \epsilon_1))
      \end{split}
    \end{equation}
    (see Remark~\ref{rem:Upsilon2_well_defined} below).
  \end{enumerate}
\end{definition}

\begin{remark}\label{rem:Upsilon2_well_defined}
  If $\Upsilon\in\Mor_{\DLPSC}(\mathcal{M},\mathcal{M}')$, by
  point~\ref{it:DLPSC_mor-C''_well_def}, $\Upsilon\times \Upsilon$
  defines a map $\Upsilon^{(2)}:C''(E)\rightarrow C''(E')$, which is
  used in point~\ref{it:DLPSC_mor-IVCM}.
\end{remark}

\begin{lemma}\label{le:basic_DLPSC_mor_props}
  Let $\Upsilon\in\Mor_{\DLPSC}(\mathcal{M},\mathcal{M}')$,
  $((\epsilon_0,m_1),(\epsilon_1,m_2))\in C''(E)$ and
  $(\epsilon'_0,m'_1) := \Upsilon(\epsilon_0,m_1)$. The following
  assertions are true.
  \begin{enumerate}
  \item \label{it:basic_DLPSC_mor_props-two_vars} Given $\delta
    \epsilon'_0\in T_{\epsilon'_0}E'$, if $D_1(p_1\circ
    \Upsilon)(\epsilon_0,m_1)(\delta \epsilon_0) = \delta \epsilon'_0$
    for some $\delta \epsilon_0\in T_{\epsilon_0} E$, then
    $d\Upsilon(\epsilon_0,m_1)(\delta \epsilon_0,0) = (\delta
    \epsilon'_0, 0)$.
  \item \label{it:basic_DLPSC_mor_props-derivative} If $\delta
    \epsilon_1\in T_{\epsilon_1} E$,
    \begin{equation}\label{eq:DLPSC_mor-C''_well_def-derivative}
      D_2(p_2\circ \Upsilon)(\epsilon_0, m_1) 
      (d\phi(\epsilon_1)(\delta \epsilon_1)) =
      d\phi'(\epsilon'_1)(D_1(p_1\circ \Upsilon)(\epsilon_1, 
      m_2)(\delta \epsilon_1)).
    \end{equation}
  \end{enumerate}
\end{lemma}

\begin{proof}
  Point~\ref{it:basic_DLPSC_mor_props-two_vars} follows from
  morphism's condition~\ref{it:DLPSC_mor-D_1_null}, satisfied by
  $\Upsilon$. Point~\ref{it:basic_DLPSC_mor_props-derivative} follows
  by noticing that $(0,d\phi(\epsilon_1)(\delta \epsilon_1),\delta
  \epsilon_1,0) \in T_{((\epsilon_0,m_1),(\epsilon_1,m_2))}C''(E)$
  and, then, using~\eqref{eq:DLPSC_mor-C''_well_def}.
\end{proof}

\begin{proposition}\label{prop:DLPSC_is_a_category}
  $\DLPSC$ is a category considering the standard composition of
  functions and identity mappings.
\end{proposition}

\begin{proof}
  In order to prove that the given data defines a category one has to
  check that the composition mapping is associative and the identities
  are left and right identities for the composition mapping. The
  composition of functions and the identity mappings meet those
  requirements, so the only thing left to prove is that $\circ$ is
  well defined in $\DLPSC$, that is, that $\circ :
  \Mor_{\DLPSC}(\mathcal{M}',\mathcal{M}'')\times
  \Mor_{\DLPSC}(\mathcal{M},\mathcal{M}') \rightarrow
  \Mor_{\DLPSC}(\mathcal{M},\mathcal{M}'')$ and that $id_\mathcal{M}
  := id_{C'(E)} \in \Mor_{\DLPSC}(\mathcal{M},\mathcal{M})$.  Both
  properties follow in a lengthy but straightforward manner.
\end{proof}

\begin{lemma}\label{le:comm_triangle_side_function_imp_side_morph}
  Let $\Upsilon'\in\Mor_{\DLPSC}(\mathcal{M},\mathcal{M}')$ and
  $\Upsilon''\in\Mor_{\DLPSC}(\mathcal{M},\mathcal{M}'')$ where
  $\mathcal{M}=(E,L_d,\IVCM)$, $\mathcal{M}'=(E',L_d',\IVCM')$ and
  $\mathcal{M}''=(E'',L_d',\IVCM'')$. If $F:C'(E')\rightarrow
  C'(E'')$ is a smooth map such that the diagram
  \begin{equation*}
    \xymatrix{ {} & {C'(E)} \ar[dl]_{\Upsilon'} \ar[dr]^{\Upsilon''} & {}\\
      {C'(E')} \ar[rr]_{F} & {} & {C'(E'')}
    }
  \end{equation*}
  is commutative, then
  $F\in\Mor_{\DLPSC}(\mathcal{M}',\mathcal{M}'')$. Furthermore, if $F$
  is a diffeomorphism, then $F$ is an isomorphism in $\DLPSC$.
\end{lemma}

\begin{proof}
  That $F$ satisfies morphism's
  conditions~\ref{it:DLPSC_mor-surj_subm}
  and~\ref{it:DLPSC_mor-D_1_onto} follows easily using the
  corresponding property of the morphism $\Upsilon''$ to lift the data
  (point or tangent vector) to $C'(E)$ and, then, using $\Upsilon'$ to
  push down to $C'(E')$. 

  Given $(\epsilon_0',m_1')\in C'(E')$ and $\delta \epsilon_0' \in
  T_{\epsilon_0'} E'$, let $(\epsilon_0,m_1)\in C'(E)$ and $\delta
  \epsilon_0\in T_{\epsilon_0}E$ such that $\Upsilon'(\epsilon_0, m_1)
  = (\epsilon_0', m_1')$ and $D_1(p_1\circ
  \Upsilon')(\epsilon_0,m_1)(\delta \epsilon_0) = \delta \epsilon_0'$,
  by point~\ref{it:basic_DLPSC_mor_props-two_vars} in
  Lemma~\ref{le:basic_DLPSC_mor_props},
  $d\Upsilon'(\epsilon_0,m_1)(\delta \epsilon_0,0) = (\delta
  \epsilon_0',0)$. As $p_2 \circ \Upsilon'' = p_2 \circ F \circ
  \Upsilon'$, taking differentials and evaluating at
  $(\epsilon_0,m_1)$ we get 
  \begin{equation*}
    \begin{split}
      D_1(p_2\circ F)(\epsilon_0',m_1')(\delta \epsilon_0') =&
      d(p_2\circ F)(\epsilon_0',m_1')(\delta \epsilon_0',0) =
      d(p_2\circ \Upsilon'')(\epsilon_0,m_1)(\delta \epsilon_0,0) \\=&
      D_1(p_2\circ \Upsilon'')(\epsilon_0,m_1)(\delta \epsilon_0) =0,
    \end{split}
  \end{equation*}
  where the last identity holds because
  $\Upsilon''\in\Mor_{\DLPSC}(\mathcal{M},\mathcal{M}'')$. Thus, $F$
  satisfies morphism's condition~\ref{it:DLPSC_mor-D_1_null}.

  The remaining conditions follow in a similar fashion, and we
  conclude that $F\in\Mor_{\DLPSC}(\mathcal{M}',\mathcal{M}'')$.

  The last assertion of the statement follows easily as the first part
  of the Lemma proves that $F^{-1}$ is a morphism in $\DLPSC$ and
  since, as functions, $F$ and $F^{-1}$ are mutually inverses, they
  have the same property as morphisms in $\DLPSC$.
\end{proof}

\begin{lemma}\label{le:morphism_DLPSC_and_diffeo_imp_iso_DLPSC}
  Let $\Upsilon\in \Mor_{\DLPSC}(\mathcal{M}, \mathcal{M}')$ for
  $\mathcal{M}=(E,L_d,\IVCM)$ and $\mathcal{M}'=(E',L_d',\IVCM')$
  such that $\Upsilon:C'(E)\rightarrow C'(E')$ is a
  diffeomorphism. Then $\Upsilon$ is an isomorphism of $\DLPSC$.
\end{lemma}

\begin{proof}
  As $\Upsilon\in \Mor_{\DLPSC}(\mathcal{M}, \mathcal{M}')$ and, by
  Proposition~\ref{prop:DLPSC_is_a_category},
  $id_{C'(E)}\in \Mor_{\DLPSC}(\mathcal{M}, \mathcal{M})$, the result
  follows from
  Lemma~\ref{le:comm_triangle_side_function_imp_side_morph} with
  $\mathcal{M}'':=\mathcal{M}$ and $F:=\Upsilon^{-1}$.
\end{proof}

The following result exposes the relation between trajectories of a
DLPS and their images under a morphism in $\DLPSC$.

\begin{theorem}\label{thm:morphisms_and_trajectories}
  Given $\Upsilon\in \Mor_{\DLPSC}(\mathcal{M},\mathcal{M}')$ with
  $\mathcal{M}=(E,L_d,\IVCM)$ and $\mathcal{M}'=(E',L_d',\IVCM')$, let
  $(\epsilon_\cdot,m_\cdot) =
  ((\epsilon_0,m_1),\ldots,(\epsilon_{N-1},m_N))$
  be a discrete path in $C'(E)$ and define
  $(\epsilon'_k,m'_{k+1}) := \Upsilon(\epsilon_k,m_{k+1})$ for
  $k=0,\ldots, N-1$. Then, $(\epsilon_\cdot, m_\cdot)$ is a trajectory
  of $\mathcal{M}$ if and only if $(\epsilon'_\cdot, m'_\cdot)$ is a
  trajectory of $\mathcal{M}'$.
\end{theorem}

\begin{proof}
  Assume that $(\delta \epsilon_\cdot,\delta m_\cdot)$ is an
  infinitesimal variation in $\mathcal{M}$ over
  $(\epsilon_\cdot,m_\cdot)$ and that $(\delta \epsilon'_\cdot,\delta
  m'_\cdot)$ is an infinitesimal variation in $\mathcal{M}'$ over
  $(\epsilon'_\cdot,m'_\cdot)$ satisfying
  \begin{equation}\label{eq:image_of_variation_under_morphism}
    d\Upsilon(\epsilon_k,m_{k+1})(\delta \epsilon_k,\delta m_{k+1}) = 
    (\delta \epsilon'_k,\delta m'_{k+1}) \stext{ for } k=0,\ldots,N-1.
  \end{equation}
  Then, using the chain rule, we see that 
  \begin{equation}\label{eq:comparison_actions_Upsilon_related}
    dS_d(\epsilon_\cdot,m_\cdot)(\delta \epsilon_\cdot,\delta
    m_\cdot) = dS_d'(\epsilon'_\cdot,m'_\cdot)(\delta
    \epsilon'_\cdot, \delta m'_\cdot).
  \end{equation}

  Next we prove the equivalence of the assertions in the statement.

  Assume that $(\epsilon_\cdot,m_\cdot)$ is a trajectory of
  $\mathcal{M}$. Let $(\delta \epsilon'_\cdot, \delta m'_\cdot)$ be
  an infinitesimal variation with fixed endpoints in $\mathcal{M}'$
  over the path $(\epsilon'_\cdot,m'_\cdot)$. That is, there are
  $\ti{\delta \epsilon'_k}\in T_{\epsilon'_k} E'$ for $k=1,\ldots,N-1$
  such that~\eqref{eq:gdms_variation_free-def}
  and~\eqref{eq:gdms_variation_fep-def} hold with $\delta \epsilon_k'$
  and $\ti{\delta \epsilon_k'}$ instead of $\delta \epsilon_k$ and
  $\ti{\delta \epsilon_k}$.

  By morphism's property~\ref{it:DLPSC_mor-D_1_onto} applied to
  $\Upsilon$, there exist $\ti{\delta \epsilon_k}\in T_{\epsilon_k}E$
  such that $D_1(p_1\circ \Upsilon)(\epsilon_k,m_{k+1})(\ti{\delta
    \epsilon_k}) = \ti{\delta \epsilon'_k}$ for $k=1,\ldots,N-1$; we
  fix one such vector for each $k$. Next
  apply~\eqref{eq:gdms_variation_free-def}
  and~\eqref{eq:gdms_variation_fep-def} to define an infinitesimal
  variation $(\delta \epsilon_\cdot, \delta m_\cdot)$ on
  $(\epsilon_\cdot, m_\cdot)$ with fixed endpoints based on the
  $\ti{\delta \epsilon_\cdot}$ constructed above.

  Direct computations using the morphism properties of $\Upsilon$ show
  that condition~\eqref{eq:image_of_variation_under_morphism} holds
  for these variations. Then,
  using~\eqref{eq:comparison_actions_Upsilon_related},
  \begin{equation*}
    dS_d'(\epsilon'_\cdot,m'_\cdot)(\delta \epsilon'_\cdot, \delta
    m'_\cdot) = dS_d(\epsilon_\cdot,m_\cdot)(\delta \epsilon_\cdot,\delta
    m_\cdot) = 0,
  \end{equation*}
  where the last equality holds because $(\delta \epsilon_\cdot,\delta
  m_\cdot)$ is an infinitesimal variation with fixed endpoints in
  $\mathcal{M}$ over $(\epsilon_\cdot,m_\cdot)$, that is a trajectory
  of $\mathcal{M}$. Finally, as $(\delta \epsilon'_\cdot, \delta
  m'_\cdot)$ was an arbitrary infinitesimal variation with fixed
  endpoints in $\mathcal{M}'$ over the path
  $(\epsilon'_\cdot,m'_\cdot)$, we conclude that
  $(\epsilon'_\cdot,m'_\cdot)$ is a trajectory of $\mathcal{M}'$.

  A similar argument shows that if $(\epsilon_\cdot',m_\cdot')$ is a
  trajectory of $\mathcal{M}'$, then $(\epsilon_\cdot,m_\cdot)$ is a
  trajectory of $\mathcal{M}$.
\end{proof}


\section{Reduction of discrete Lagrange--Poincar\'e systems}
\label{sec:reduction_of_generalized_discrete_mechanical_systems}

The purpose of this section is to define what is meant by a group of
symmetries of a DLPS. Also, a reduction result is studied.


\subsection{Symmetry groups of discrete Lagrange--Poincar\'e systems}
\label{sec:symmetry_groups_of_generalized_discrete_mechanical_systems}

Recall that a $\SG$-action on a fiber bundle consists of a pair of
$\SG$-actions $l^E$ and $l^M$, satisfying a number of conditions
(Definition~\ref{def:group_acts_on_fiber_bundle}). We can use these
actions to define ``diagonal'' $\SG$-actions on the fiber bundles
$C'(E)$ and $C''(E)$ by
\begin{equation}\label{eq:G_actions_on_C'_and_C''}
  \begin{split}
    l^{C'(E)}_g(\epsilon_0,m_1) :=& (l^E_g(\epsilon_0),l^M_g(m_1))\\
    l^{C''(E)}_g((\epsilon_0,m_1),(\epsilon_1,m_2)) :=&
    (l^{C'(E)}_g(\epsilon_0,m_1), l^{C'(E)}_g(\epsilon_1,m_2)).
  \end{split}
\end{equation}
These actions are smooth and free because $l^E$ and $l^M$ have those
properties. In addition, the bundle projection maps of $C'(E)$ and
$C''(E)$ on $M$ are $\SG$-equivariant and $\pi^{M,\SG}:M\rightarrow
M/\SG$ is a principal $\SG$-bundle. In fact, it is easy to check that
$\SG$ acts on the fiber bundles $\phi\circ p_1:C'(E)\rightarrow M$ and
$\phi\circ p_3:C''(E)\rightarrow M$.

We can also define $\SG$-actions on $\ker(d\phi)$ and $p_3^*(TE)$ by
\begin{equation}\label{eq:p_1^*TE-action}
  \begin{split}
    l^{TE}_g(\epsilon_0,\delta \epsilon_0) :=&
    dl^E_g(\epsilon_0)(\delta \epsilon_0),\\
    l^{p_3^*(TE)}_g((\epsilon_0,m_1),(\epsilon_1,m_2),\delta \epsilon_1)
    :=& (l^{C''(E)}_g((\epsilon_0,m_1),(\epsilon_1,m_2)),
    dl^E_g(\epsilon_1)(\delta \epsilon_1)).
  \end{split}
\end{equation}
We denote the $\SG$-action on $\ker(d\phi)$ by $l^{TE}$ because it is
the restriction of the natural $\SG$-action on $TE$. The action
$l^{TE}$ is well defined by the $\SG$-equivariance of $\phi$.

\begin{lemma}\label{le:properties_Upsilon^2}
  Let $\SG$ be a Lie group acting on the fiber bundle
  $\phi:E\rightarrow M$ and $\DC$ be a discrete connection on the
  principal $\SG$-bundle $\pi^{M,\SG}:M\rightarrow M/\SG$. Define
  $\Upsilon_\DC^{(2)}:C''(E)\rightarrow C''(\ti{\SG}_E)$ as the
  restriction of $(\Upsilon_\DC\circ p_1)\times(\Upsilon_\DC\circ
  p_2): C'(E)\times C'(E)\rightarrow C'(\ti{\SG}_E)\times
  C'(\ti{\SG}_E)$ to the corresponding spaces, where $\Upsilon_\DC$ is
  the surjective submersion defined in~\eqref{eq:Upsilon_DC-def}. Then
  \begin{enumerate}
  \item \label{it:properties_Upsilon^2-well_def} $\Upsilon_\DC^{(2)}$
    is well defined.
  \item \label{it:properties_Upsilon^2-isomorphism} $D_1(p_1\circ
    \Upsilon_\DC)(\epsilon_0,m_1):T_{(\epsilon_0,m_1)}(E\times
    \{m_1\})\rightarrow T_{(p_1\circ
      \Upsilon_\DC)(\epsilon_0,m_1)}\ti{\SG}_E$ is an isomorphism of
    vector spaces for every $(\epsilon_0,m_1)\in C'(E)$.
  \item \label{it:properties_Upsilon^2-ppal_bundle}
    $\Upsilon_\DC^{(2)}:C''(E)\rightarrow C''(\ti{\SG}_E)$ is a
    principal $\SG$-bundle with structure group $\SG$. In particular,
    $C''(E)/\SG\simeq C''(\ti{\SG}_E)$.
  \item \label{it:properties_Upsilon^2-lift} For
    $((v_0,r_1),(v_1,r_2))\in C''(\ti{\SG}_E)$ and
    $(\epsilon_0,m_1)\in C'(E)$ such that
    $\Upsilon_\DC(\epsilon_0,m_1)=(v_0,r_1)$, there is a unique
    $(\epsilon_1,m_2)\in C'(E)$ such that
    $((\epsilon_0,m_1),(\epsilon_1,m_2))\in C''(E)$ and
    $\Upsilon_\DC^{(2)}((\epsilon_0,m_1),(\epsilon_1,m_2)) =
    ((v_0,r_1),(v_1,r_2))$.
  \end{enumerate}
\end{lemma}
\begin{proof} 
  A simple computation shows that, for
  $((\epsilon_0,m_1),(\epsilon_1,m_2))\in C''(E)$, we have
  $p_2(\Upsilon_\DC(\epsilon_0,m_1)) =
  p^{M/\SG}(p_1(\Upsilon_\DC(\epsilon_1,m_2)))$, proving
  point~\ref{it:properties_Upsilon^2-well_def}.

  Let $(v_0,r_1):=\Upsilon_\DC(\epsilon_0,m_1) =
  (\pi^{E\times\SG,\SG}(\epsilon_0,\DC(\phi(\epsilon_0),m_1)),\pi^{M,\SG}(m_1))$. It
  is easy to check that if $\delta \epsilon_0 \in \ker(D_1(p_1\circ
  \Upsilon_\DC)(\epsilon_0,m_1))$, then $(\delta \epsilon_0,0)\in
  \ker(d\Upsilon_\DC(\epsilon_0,m_1)) =
  \{(\xi_E(\epsilon_0),\xi_M(m_1))\in T_{(\epsilon_0,m_1)}(E\times M):
  \xi\in\jgsg\}$. But, being $\pi^{M,\SG}:M\rightarrow M/\SG$ a
  principal $\SG$-bundle, $\xi_M(m_1)=0$ implies that $\xi=0$, and we
  conclude that $\delta \epsilon_0 = 0$, so that $D_1(p_1\circ
  \Upsilon_\DC)(\epsilon_0,m_1):T_{(\epsilon_0,m_1)}(E\times
  \{m_1\})\rightarrow T_{v_0}\ti{\SG}_E$ is one to one. As, in
  addition, $\dim(T_{(\epsilon_0,m_1)}(E\times \{m_1\})) =
  \dim(T_{v_0}\ti{\SG}_E)$, we conclude that
  point~\ref{it:properties_Upsilon^2-isomorphism} is true.

  Consider the commutative diagram
  \begin{equation*}
    \xymatrix{
      {C''(E)} \ar[d]_{\Upsilon_\DC^{(2)}} \ar[r]^{F_E} & 
      {E\times E\times M} \ar[d]^{\ti{\Upsilon}_\DC^{(2)}}\\
      {C''(\ti{\SG}_E)} \ar[r]_(0.4){F_{\ti{\SG}_E}} & 
      {\ti{\SG}_E\times \ti{\SG}_E\times (M/\SG)}
    }
  \end{equation*}
  where $F_E$ and $F_{\ti{\SG}_E}$ are the diffeomorphisms introduced
  in Remark~\ref{rem:isomorphism_C''(E)_with_ExExM} and
  \begin{equation*}
    \ti{\Upsilon}_\DC^{(2)}(\epsilon_0,\epsilon_1,m_2) := ((p_1\circ
    \Upsilon_\DC)(\epsilon_0,\phi(\epsilon_1)),
    \Upsilon_\DC(\epsilon_1,m_2)).
  \end{equation*}
  It is clear that $\ti{\Upsilon}_\DC^{(2)}$ is smooth. Furthermore,
  as the projection to its last two components is simply
  $\Upsilon_\DC:E\times M\rightarrow \ti{\SG}_E\times (M/\SG)$, that
  is known to be a surjective submersion and applying
  point~\ref{it:properties_Upsilon^2-isomorphism} to the first
  component, we conclude that $\ti{\Upsilon}_\DC^{(2)}$ is a
  submersion. We check explicitly that $\ti{\Upsilon}_\DC^{(2)}$ is
  surjective. Let $(v_0,v_1,r_2)\in\ti{\SG}_E\times \ti{\SG}_E\times
  (M/\SG)$. Then, by definition of $\ti{\SG}_E$, there are
  $(\epsilon_0,m_1)\in E\times M$ such that
  $\Upsilon_\DC(\epsilon_0,m_1)=(v_0,p^{M/\SG}(v_1))$. Next, choose
  $(\epsilon_1',m_2')\in E\times M$ such that
  $\Upsilon_\DC(\epsilon_1',m_2')=(v_1,r_2)$. Notice that, using
  diagram~\eqref{eq:ExM_and_tiGxM/G}, $\pi^{M,\SG}(\phi(\epsilon_1'))
  = p^{M/\SG}(v_1) = \pi^{M,\SG}(m_1)$. Hence, as
  $\pi^{M,\SG}:M\rightarrow M/\SG$ is a principal $\SG$-bundle, there
  is $g'\in\SG$ such that $l^M_{g'}(\phi(\epsilon_1')) = m_1$. We
  define $(\epsilon_1,m_2) := l^{E\times
    M}_{g'}(\epsilon_1',m_2')$. By construction,
  $\Upsilon_\DC(\epsilon_1,m_2)=(v_1,r_2)$ and
  $\phi(\epsilon_1)=m_1$. All together,
  $\ti{\Upsilon}_\DC^{(2)}(\epsilon_0,\epsilon_1,m_2)=(v_0,v_1,r_2)$,
  showing that $\ti{\Upsilon}_\DC^{(2)}$ is onto. Using that
  $\Upsilon_\DC$ is a principal $\SG$-bundle, it follows easily that
  $(\ti{\Upsilon}_\DC^{(2)})^{-1}(v_0,v_1,r_2)=l^{E\times E\times
    M}_\SG\{(\epsilon_0,\epsilon_1,m_2)\}$, showing that
  $(\ti{\Upsilon}_\DC^{(2)})^{-1}(v_0,v_1,r_2)$ coincides with the
  orbit of the free ``diagonal'' action of $\SG$ on $E\times E\times
  M$. Theorem~\ref{thm:ppal_bundles_via_embedding} proves that
  $\ti{\Upsilon}_\DC^{(2)}:E\times E\times M\rightarrow
  \ti{\SG}_E\times\ti{\SG}_E\times(M/\SG)$ is a principal
  $\SG$-bundle. Finally, since the diffeomorphism $F_E$ is
  $\SG$-equivariant (when considering the $\SG$-actions $l^{C''(E)}$
  and $l^{E\times E\times M}$), we conclude that
  point~\ref{it:properties_Upsilon^2-ppal_bundle} holds.

  Notice that in the first step of the previous construction, we
  picked $(\epsilon_0,m_1)\in C'(E)$ such that
  $\Upsilon_\DC(\epsilon_0,m_1) = (v_0,p^{M/\SG}(v_1))$. In the
  context of point~\ref{it:properties_Upsilon^2-lift}, such
  $(\epsilon_0,m_1)$ is given. Hence, the rest of the construction
  produces $(\epsilon_1,m_2)$ so that
  $((\epsilon_0,m_1),(\epsilon_1,m_2)) \in C''(E)$ and
  $\Upsilon_\DC^{(2)}((\epsilon_0,m_1),(\epsilon_1,m_2)) =
  ((v_0,r_1),(v_1,r_2))$. The uniqueness of that pair follows from the
  fact this is the only element in the $\SG$-orbit that has
  $(\epsilon_0,m_1)$ as the first component. Hence,
  point~\ref{it:properties_Upsilon^2-lift} is valid.
\end{proof}

\begin{proposition}\label{prop:lifting_reduced_discrete_paths}
  Let $\SG$ be a Lie group acting on the fiber bundle
  $\phi:E\rightarrow M$ and $\DC$ a discrete connection on the
  principal $\SG$-bundle $\pi^{M,\SG}:M\rightarrow M/\SG$. Given a
  discrete path $(v_\cdot,r_\cdot) = ((v_0,r_1), \ldots, (v_{N-1},r_N))$
  in $C'(\ti{\SG}_E)$ and $(\ti{\epsilon}_0,\ti{m}_1)\in C'(E)$ such
  that $\Upsilon_\DC(\ti{\epsilon}_0, \ti{m}_1) = (v_0,r_1)$, there is
  a unique discrete path $(\epsilon_\cdot,m_\cdot)$ in $C'(E)$ such
  that $(\epsilon_0,m_1) = (\ti{\epsilon}_0,\ti{m}_1)$ and
  $\Upsilon_\DC(\epsilon_k,m_{k+1}) = (v_k, r_{k+1})$ for all $k$.
\end{proposition}

\begin{proof}
  The proof is by induction in the length of the reduced discrete
  path, $N$. If $N=0$, taking $(\epsilon_0,m_1) :=
  (\ti{\epsilon}_0,\ti{m}_1)$ solves the problem. Otherwise, assume
  that the result holds for all lengths $<N$ and $(v_\cdot,r_\cdot) =
  ((v_0,r_1), \ldots, (v_{N-1},r_N))$. Then, there is a discrete path
  $((\epsilon_0,m_1),\ldots,(\epsilon_{N-2},m_{N-1}))$ in $C'(E)$ that
  lifts $((v_0,r_1), \ldots, (v_{N-2},r_{N-1}))$ starting at
  $(\ti{\epsilon}_0,\ti{m}_1)$. In particular,
  $\Upsilon_\DC((\epsilon_{N-2},m_{N-1})) = (v_{N-2},r_{N-1})$. As, in
  addition, $((v_{N-2},r_{N-1}), (v_{N-1},r_{N}))\in C''(\ti{\SG}_E)$,
  by point~\ref{it:properties_Upsilon^2-lift} in
  Lemma~\ref{le:properties_Upsilon^2}, there is
  $(\epsilon_{N-1},m_N)\in C'(E)$ such that
  $((\epsilon_{N-2},m_{N-1}),(\epsilon_{N-1},m_N)) \in C''(E)$ and
  $\Upsilon_\DC^{(2)}((\epsilon_{N-2},m_{N-1}),(\epsilon_{N-1},m_N)) =
  ((v_{N-2},r_{N-1}), (v_{N-1},r_{N}))$. This proves that
  $((\epsilon_0,m_1),\ldots,(\epsilon_{N-2},m_{N-1}),(\epsilon_{N-1},m_N))$
  is a discrete path in $C'(E)$ starting at
  $(\ti{\epsilon}_0,\ti{m}_1)$ and that lifts $((v_0,r_1), \ldots,
  (v_{N-1},r_N))$. This proves that the statement holds for discrete
  paths of length $N$ so that, by the induction principle, it holds
  for arbitrary lengths.
\end{proof}

\begin{definition}
  A Lie group $\SG$ is a \jdef{symmetry group} of the DLPS
  $\mathcal{M}=(E,L_d,\IVCM)$ if
  \begin{enumerate}
  \item \label{it:symmetry_group-G_acts_on_FB} $\SG$ acts on the fiber
    bundle $\phi:E\rightarrow M$
    (Definition~\ref{def:group_acts_on_fiber_bundle}),
  \item considering the ``diagonal action'' of $\SG$ on $C'(E)$,
    $l^{C'(E)}$ defined in~\eqref{eq:G_actions_on_C'_and_C''}, $L_d$
    is $\SG$-invariant, and
  \item $\IVCM$ is a $\SG$-equivariant element of
    $\hom(p_3^*(TE),\ker(d\phi))$ for the $\SG$-actions $l^{TE}$ and
    $l^{p_3^*(TE)}$ defined in~\eqref{eq:p_1^*TE-action}. In other
    words,
    \begin{equation}
      \label{eq:G_equivariance_of_IVCM}
      \IVCM\circ l^{p_3^*(TE)}_g = l^{TE}_g \circ \IVCM =
      dl^E_g\circ \IVCM \stext{ for all } g\in\SG.
    \end{equation}
  \end{enumerate}
\end{definition}

\begin{example}\label{ex:symmetry_of_DMS_is_symmetry_of_DLPS}
  Let $(Q,L_d)$ be a DMS and $\mathcal{M}:=(id_Q:Q\rightarrow Q,
  L_d,0)$ the DLPS associated to $(Q,L_d)$ in
  Example~\ref{ex:mechanical_system_as_generalized_mechanical_system}. If
  $\SG$ is a symmetry group of $(Q,L_d)$ as in
  Section~\ref{sec:reduced_system_associated_to_a_symmetric_mechanical_system},
  then $\SG$ acts on the fiber bundle $id_Q:Q\rightarrow Q$ and $L_d$
  is $\SG$-invariant. Also, as $\IVCM=0$,
  condition~\eqref{eq:G_equivariance_of_IVCM} is trivially
  satisfied. Hence, $\SG$ is a symmetry group of $\mathcal{M}$.
\end{example}

\begin{lemma}\label{le:G_equivariance_of_IVCM_and_IVCM_morph_cond}
  Let $\mathcal{M} = (E,L_d,\IVCM) \in\Ob_{\DLPSC}$ and $\SG$ be a Lie
  group. Then, for $g\in \SG$,~\eqref{eq:G_equivariance_of_IVCM} holds
  if and only if~\eqref{eq:DLPSC_mor-IVCM} holds for
  $\Upsilon:=l^{C'(E)}_g$ and $\mathcal{M}'=\mathcal{M}$.
\end{lemma}
\begin{proof}
  Unravel the definitions.
\end{proof}

\begin{proposition}\label{prop:symmetry_group_and_morphisms}
  Let $\mathcal{M} = (E,L_d,\IVCM) \in\Ob_{\DLPSC}$ and $\SG$ a Lie
  group. Then $\SG$ is a symmetry group of $\mathcal{M}$ if and only
  if $\SG$ acts on the fiber bundle $\phi:E\rightarrow M$ and
  $l^{C'(E)}_g\in \Mor_{\DLPSC}(\mathcal{M}, \mathcal{M})$ for all
  $g\in \SG$.
\end{proposition}

\begin{proof}
  Assume that $\SG$ is a symmetry group of $\mathcal{M}$. Then, by
  definition, $\SG$ acts on the fiber bundle $\phi:E\rightarrow M$. We
  have to prove that $l^{C'(E)}_g\in
  \Mor_{\DLPSC}(\mathcal{M},\mathcal{M})$. It is immediate that
  $l^{C'(E)}_g$ is a diffeomorphism, so it has morphism's
  property~\ref{it:DLPSC_mor-surj_subm}. As $p_1\circ l^{C'(E)}_g =
  l^E_g\circ p_1$, we have $D_1(p_1\circ l^{C'(E)}_g) = dl^{E}_g$,
  that is an isomorphism; hence, $l^{C'(E)}_g$ has morphism's
  property~\ref{it:DLPSC_mor-D_1_onto}. As $p_2\circ l^{C'(E)}_g =
  l^M_g\circ p_2$, $D_1(p_2\circ l^{C'(E)}_g) = D_1(l^M_g\circ p_2) =
  0$, it follows that $l^{C'(E)}_g$ has morphism's
  property~\ref{it:DLPSC_mor-D_1_null}. Also, as on $C''(E)$ we have
  that
  \begin{equation*}
    \begin{split}
      p_2^{C'(E),M}\circ l^{C'(E)}_g \circ p_1^{C''(E),C'(E)} =& l^M_g
      \circ p_2^{C'(E),M} \circ p_1^{C''(E),C'(E)} \\=& \phi\circ
      p_1^{C'(E),E} \circ l^{C'(E)}_g \circ p_2^{C''(E),C'(E)},
    \end{split}
  \end{equation*}
  we see that $l^{C'(E)}_g$ has morphism's
  property~\ref{it:DLPSC_mor-C''_well_def}. As $L_d\circ l^{C'(E)}_g =
  L_d$, $l^{C'(E)}_g$ has morphism's property~\ref{it:DLPSC_mor-L_d}
  and Lemma~\ref{le:G_equivariance_of_IVCM_and_IVCM_morph_cond} shows
  that morphism's property~\ref{it:DLPSC_mor-IVCM} is valid for
  $l^{C'(E)}_g$. We conclude that $l^{C'(E)}_g\in
  \Mor_{\DLPSC}(\mathcal{M}, \mathcal{M})$.

  Conversely, if $\SG$ acts on the fiber bundle $\phi:E\rightarrow M$
  and $l^{C'(E)}_g\in \Mor_{\DLPSC}(\mathcal{M}, \mathcal{M})$, the
  first condition for being a symmetry group is met. The other two
  follow from morphism's properties~\ref{it:DLPSC_mor-L_d}
  and~\ref{it:DLPSC_mor-IVCM}, together with
  Lemma~\ref{le:G_equivariance_of_IVCM_and_IVCM_morph_cond}.
\end{proof}

Later on we will be interested in subgroups of a symmetry group of a
DLPS. The following results establish that closed subgroups of a
symmetry group of a system $\DLPS$ are symmetry groups of $\DLPS$.

\begin{lemma}\label{le:group_acts_on_FB_subgroup_acts_too}
  Let $\SG$ act on the fiber bundle $(E,M,\phi,F)$ and $H\subset \SG$
  be a closed Lie subgroup. Then $H$ acts on the fiber bundle
  $(E,M,\phi,F)$.
\end{lemma}
\begin{proof}
  We consider the $H$-actions on $E$, $M$ and $F$ obtained by
  restricting the $\SG$-actions $l^E$, $l^M$ and $r^F$ to $H$. Hence,
  all are smooth and the first two are free; also, $\phi$ is
  $H$-equivariant. As $\pi^{M,\SG}:M\rightarrow M/\SG$ is a principal
  $\SG$-bundle, by
  Lemma~\ref{le:G_actions_on_ppal_bundles_are_proper}, the
  $\SG$-action $l^M$ is proper and, being $H\subset \SG$ closed, the
  $H$-action $l^M$ obtained by restriction is proper. Then applying
  Corollary~\ref{cor:quotient_manifolds-fiber_bundle} to the
  $H$-action $l^M$, we see that $\pi^{M,H}:M\rightarrow M/H$ is a
  principal $H$-bundle. Given $m\in M$, there is a trivializing chart
  $(U,\Phi_U)$ with $m\in U$, an open $\SG$-invariant subset of $M$,
  and $\Phi_U$ $\SG$-equivariant. Thus, $U$ is $H$-invariant and
  $\Phi_U$ is $H$-equivariant, so that $(U,\Phi_U)$ is the type of
  trivializing chart required in
  point~\ref{it:group_acts_on_FB-trivializing_chart} of
  Definition~\ref{def:group_acts_on_fiber_bundle} to conclude that $H$
  acts on the fiber bundle $(E,M,\phi,F)$.
\end{proof}

\begin{proposition}\label{prop:subgroup_of_sym_group_is_sym_group}
  Let $\SG$ be a symmetry group of $\mathcal{M}\in\Ob_{\DLPSC}$. If
  $H\subset \SG$ is a closed Lie subgroup, then $H$ is a symmetry
  group of $\mathcal{M}$.
\end{proposition}

\begin{proof}
  Since $\SG$ is a symmetry group of $\mathcal{M} = (E,L_d,\IVCM)$,
  $\SG$ acts on the fiber bundle $(E,M,\phi,F)$ and, by
  Lemma~\ref{le:group_acts_on_FB_subgroup_acts_too}, the same happens
  to the closed subgroup $H$, when acting via the restricted
  $\SG$-actions. That $L_d$ is $H$-invariant and $\IVCM$ is
  $H$-equivariant, follow from the fact that they have those
  properties for $\SG$, and that $H$ acts by the restriction of the
  corresponding $\SG$-actions. Thus, $H$ is a symmetry group of
  $\mathcal{M}$.
\end{proof}


\begin{remark}\label{rem:discrete_momenta}
  When $\SG$ is a symmetry group of $\mathcal{M}=(E,L_d,\IVCM)
  \in\Ob_{\DLPSC}$ there are functions $J_d:C'(E)\rightarrow \jgsg^*$
  and $(J_d)_\xi: C'(E)\rightarrow \R$ defined as follows.
  $J_d(\epsilon_0,m_1)(\xi) :=
  -D_1L_d(\epsilon_0,m_1)(\xi_E(\epsilon_0))$ for $(\epsilon_0,m_1)\in
  C'(E)$ and $\xi \in \jgsg$, where $\xi_E$ is the infinitesimal
  generator associated to $\xi$ by the $\SG$-action on $E$. Then,
  $(J_d)_{\xi}(\epsilon_0,m_1) := J_d(\epsilon_0,m_1)(\xi)$. In some
  sense, these functions resemble the momentum mappings that appear in
  the context of DMS. It is easy to check that when
  $(\epsilon_\cdot,m_\cdot)$ is a trajectory of $\mathcal{M}$, for any
  $\xi\in\jgsg$,
  \begin{equation}\label{eq:momentum-evolution_on_trajectory}
    \begin{split}
      (J_d)_\xi(\epsilon_k,m_{k+1}) =& (J_d)_\xi(\epsilon_{k-1},m_k) \\ &+
      D_1L_d(\epsilon_{k-1},m_{k}) \circ 
      \IVCM((\epsilon_{k-1},m_{k}),(\epsilon_{k},m_{k+1}))(\xi_E(\epsilon_k))
    \end{split}
  \end{equation}
  for all $k=0,\ldots,N-1$. This last expression shows how $J_d$
  evolves on a given trajectory of $\mathcal{M}$. In particular, when
  the image of $\IVCM$ is contained in $\ker(D_1L_d)$, the momentum is
  conserved along the trajectories; this is the case of a DLPS arising
  from a discrete mechanical system (see
  Example~\ref{ex:mechanical_system_as_generalized_mechanical_system}).
  Equation~\eqref{eq:momentum-evolution_on_trajectory} can also be
  compared with the momentum evolution equation in the nonholonomic
  case: equation (35)
  in~\cite{ar:fernandez_tori_zuccalli-lagrangian_reduction_of_discrete_mechanical_systems}.
\end{remark}


\subsection{Reduced discrete Lagrange--Poincar\'e system}
\label{sec:reduced_generalized_discrete_mechanical_system}

Let $\SG$ be a symmetry group of $\mathcal{M}=(E,L_d,\IVCM) \in
\Ob_{\DLPSC}$. We want to construct a new DLPS that, as will be shown
later, will play the role of the reduced system of
$\mathcal{M}$. First of all, since $\SG$ acts on $(E,M,\phi,F)$, the
conjugate bundle $(\ti{\SG}_E,M/\SG,p^{M/\SG},F\times \SG)$, introduced
in Example~\ref{ex:extended_associated_bundle-fiber_bundle}, is a
fiber bundle.

Fix a discrete connection $\DC$ on the principal $\SG$-bundle
$\pi^{M,\SG}:M\rightarrow M/\SG$ and let $\Upsilon_\DC:E\times
M\rightarrow \ti{\SG}_E\times (M/\SG)$ be the map introduced
in~\eqref{eq:Upsilon_DC-def} that, by
Lemma~\ref{le:Upsilon_DC_is_ppal_bundle}, is a principal
$\SG$-bundle. Define $\check{L}_d:\ti{\SG}_E\times (M/\SG)\rightarrow
\R$ by $\check{L}_d(v_0,r_1):=L_d(\epsilon_0,m_1)$ for any
$(\epsilon_0,m_1)\in \Upsilon_\DC^{-1}(v_0,r_1)$; $\check{L}_d$ is
well defined by the $\SG$-invariance of $L_d$.

Next we define
$\check{\IVCM} \in\hom(p_3^*(T(\ti{\SG}_E)),\ker(dp^{M/\SG}))$. By
point~\ref{it:properties_Upsilon^2-ppal_bundle} of
Lemma~\ref{le:properties_Upsilon^2}, given any
$((v_0,r_1),(v_1,r_2))\in C''(\ti{\SG}_E)$, there are elements
$((\epsilon_0,m_1),(\epsilon_1,m_2))\in C''(E)$ such that
$\Upsilon_\DC^{(2)}((\epsilon_0,m_1),(\epsilon_1,m_2)) =
((v_0,r_1),(v_1,r_2))$.
In fact, those elements form a $\SG$-orbit in $C''(E)$; we fix one
element in the orbit. Also, by
point~\ref{it:properties_Upsilon^2-isomorphism} of
Lemma~\ref{le:properties_Upsilon^2},
$D_1(p_1\circ
\Upsilon_\DC)(\epsilon_1,m_2):T_{(\epsilon_1,m_2)}(E\times
\{m_2\})\rightarrow T_{v_1}\ti{\SG}_E$
is an isomorphism of vector spaces. Consequently, every element
$((v_0,r_1),(v_1,r_2),\delta v_1)\in p_3^*(T(\ti{\SG}_E))$ is
$\delta v_1 = D_1(p_1\circ \Upsilon_\DC)(\epsilon_1,m_2)(\delta
\epsilon_1)$
for a unique
$\delta \epsilon_1 \in T_{(\epsilon_1,m_2)}(E\times \{m_2\})$. Let
\begin{equation}
  \label{eq:PVMs_for_reduced_system-def}
  \begin{split}
    \check{\IVCM}((v_0,r_1),(v_1,r_2))(\delta v_1) :=& D_1(p_1\circ
    \Upsilon_\DC)(\epsilon_0,m_1)
    (\IVCM((\epsilon_0,m_1),(\epsilon_1,m_2))(\delta \epsilon_1)) \\
    &+ D_2(p_1\circ
    \Upsilon_\DC)(\epsilon_0,m_1)(d\phi(\epsilon_1)(\delta
    \epsilon_1)).
  \end{split}
\end{equation}

\begin{lemma}\label{le:reduced_IVCM_is_well_def}
  Under the previous conditions, the map defined
  by~\eqref{eq:PVMs_for_reduced_system-def} is a well defined
  homomorphism $\check{\IVCM}\in
  \hom(p_3^*(T(\ti{\SG}_E)),\ker(dp^{M/\SG}))$.
\end{lemma}

\begin{proof}
  Two things have to be checked: that the image of $\check{\IVCM}$ is
  contained in $\ker(dp^{M/\SG})$ and that the definition is
  independent of any choices involved in lifting the input data to
  $p_3^*(TE)$. Since the points $((\epsilon_0,m_1),(\epsilon_1,m_2))$
  lying over $((v_0,r_1),(v_1,r_2))$ form a $\SG$-orbit, any other
  such point would be of the form
  $((\epsilon_0',m_1'),(\epsilon_1',m_2')) =
  l^{C''(E)}_g((\epsilon_0,m_1),(\epsilon_1,m_2))$ for some
  $g\in\SG$. It follows from the $\SG$-invariance of $\Upsilon_\DC$
  that
  \begin{equation*}
    D_1(p_1\circ \Upsilon_\DC)(\epsilon_1,m_2)(\delta \epsilon_1) =
    D_1(p_1\circ
    \Upsilon_\DC)(l^{C'(E)}_g(\epsilon_1,m_2))(dl^E_g(\epsilon_1)(\delta
    \epsilon_1)).
  \end{equation*}
  Hence, a variation
  $((v_0,r_1),(v_1,r_2),\delta v_1) \in p_3^*(T(\ti{\SG}_E))$ lifts to
  the (unique for each $g$) variation
  $(l^{C''(E)}_g((\epsilon_0,m_1),(\epsilon_1,m_2)),dl^E_g(\epsilon_1)(\delta
  \epsilon_1))$
  for arbitrary $g\in\SG$. Then, for a given $g\in \SG$, using the
  $\SG$-equivariance of $\IVCM$ and the $\SG$-invariance of
  $\Upsilon_\DC$, we see that replacing
  $((\epsilon_0,m_1),(\epsilon_1,m_2))$ and $\delta \epsilon_1$ by
  $l^{C''(E)}_g((\epsilon_0,m_1),(\epsilon_1,m_2))$ and
  $(l^{C''(E)}_g((\epsilon_0,m_1),(\epsilon_1,m_2)),dl^E_g(\epsilon_1)(\delta
  \epsilon_1))$
  does not alter the value of the left side
  of~\eqref{eq:PVMs_for_reduced_system-def}. This proves
  $\check{\IVCM}$ is independent of the choices made.

  Direct computations show that the image of $\check{\IVCM}$ is
  contained in $\ker(dp^{M/\SG})$.
\end{proof}

\begin{definition}
  Let $\SG$ be a symmetry group of
  $\mathcal{M}=(E,L_d,\IVCM)\in\Ob_{\DLPSC}$ and $\DC$ a discrete
  connection on the principal $\SG$-bundle $\pi^{M,\SG}:M\rightarrow
  M/\SG$.  The DLPS $(\ti{\SG}_E,\check{L}_d,\check{\IVCM}) \in
  \Ob_{\DLPSC}$ defined above is called the \jdef{reduced discrete
    Lagrange--Poincar\'e system} obtained as the reduction of
  $\mathcal{M}$ by the symmetry group $\SG$ using the discrete
  connection $\DC$. We denote this system by $\mathcal{M}/\SG$ or
  $\mathcal{M}/(\SG,\DC)$.
\end{definition}

\begin{example}
  Given a DMS $(Q,L_d)$, let $\mathcal{M}:=(Q,L_d,0)$ be the DLPS
  constructed in
  Example~\ref{ex:mechanical_system_as_generalized_mechanical_system}. Let
  $\SG$ be a symmetry group of $(Q,L_d)$. By
  Example~\ref{ex:symmetry_of_DMS_is_symmetry_of_DLPS}, $\SG$ is a
  symmetry group of $\mathcal{M}$. Fix a discrete connection $\DC$ on
  the principal $\SG$-bundle $\pi^{Q,\SG}:Q\rightarrow Q/\SG$. The
  reduced DLPS $\mathcal{M}/(\SG,\DC)$ is
  $(\ti{\SG}_E,\check{L}_d,\check{\IVCM})$ where the fiber bundle
  $\phi:\ti{\SG}_E\rightarrow M/\SG$ is $p^{Q/\SG}:\ti{\SG}\rightarrow
  Q/\SG$, the lagrangian is determined by $\check{L}_d \circ
  \Upsilon_\DC = L_d$ and, according
  to~\eqref{eq:PVMs_for_reduced_system-def},
  \begin{equation*}
    \check{\IVCM}((v_{k-1},r_k),(v_k,r_{k+1}))(\delta v_k) =
    D_2(p_1\circ \Upsilon_\DC)(q_{k-1},q_k)(\delta q_k),
  \end{equation*}
  where $(v_{k-1},r_k) = \Upsilon_\DC(q_{k-1},q_k)$, $(v_{k},r_{k+1})
  = \Upsilon_\DC(q_{k},q_{k+1})$ and $\delta v_k = D_1(p_1\circ
  \Upsilon_\DC)(q_k,q_{k+1})(\delta q_k)$. Notice that
  this DLPS coincides with the DLPS associated to the reduction of
  $(Q,L_d)$ in
  Section~\ref{sec:reduced_system_associated_to_a_symmetric_mechanical_system}. In
  other words, the reduced system $\mathcal{M}/(\SG,\DC)$ extends the
  reduction construction of DMSs introduced
  in~\cite{ar:fernandez_tori_zuccalli-lagrangian_reduction_of_discrete_mechanical_systems}.
\end{example}

\begin{proposition}\label{prop:Upsilon_DC_is_morphism}
  Let $\SG$ be a symmetry group of
  $\mathcal{M}=(E,L_d,\IVCM)\in \Ob_{\DLPSC}$ and $\DC$ a discrete
  connection on the principal $\SG$-bundle
  $\pi^{M,\SG}:M\rightarrow M/\SG$. Then $\Upsilon_\DC$ defined
  by~\eqref{eq:Upsilon_DC-def} is in
  $\Mor_{\DLPSC}(\mathcal{M}, \mathcal{M}/(\SG,\DC))$.
\end{proposition}

\begin{proof}
  We have already noticed that
  $\Upsilon_\DC:C'(E)\rightarrow \ti{\SG}_E$ is a surjective
  submersion, so that morphism's property~\ref{it:DLPSC_mor-surj_subm}
  holds. By point~\ref{it:properties_Upsilon^2-isomorphism} of
  Lemma~\ref{le:properties_Upsilon^2}, morphism's
  property~\ref{it:DLPSC_mor-D_1_onto} holds. As
  $p_2\circ \Upsilon_\DC = \pi^{M,\SG} \circ p_2$, if
  $i_1:TE\rightarrow T(C'(E)) = TE \oplus TM$ is the natural
  inclusion, we have that
  $D_1(p_2\circ \Upsilon_\DC) = d\pi^{M,\SG} \circ dp_2 \circ i_1 =
  0$,
  as $\im(i_1)\subset \ker(dp_2)$, so that morphism's
  property~\ref{it:DLPSC_mor-D_1_null} is valid. As
  \begin{equation*}
    (p_2^{C'(\ti{\SG}_E),M/\SG} \circ \Upsilon_\DC \circ
    p_1^{C''(E),C'(E)}) ((\epsilon_0,m_1),(\epsilon_1,m_2)) = \pi^{M,\SG}(m_1)
  \end{equation*}
  and
  \begin{equation*}
    \begin{split}
      (p^{M,\SG}\circ p_1^{C'(\ti{\SG}_E),\ti{\SG}_E}\circ
      &\Upsilon_\DC \circ
      p_2^{C''(E),C'(E)})((\epsilon_0,m_1),(\epsilon_1,m_2)) \\=&
      p^{M,\SG}(\pi^{E\times\SG,\SG}(\epsilon_1,\DC(\phi(\epsilon_1),m_2)))
      = \pi^{M,\SG}(\phi(\epsilon_1)) = \pi^{M,\SG}(m_1),
    \end{split}
  \end{equation*}
  morphism's property~\ref{it:DLPSC_mor-C''_well_def} is
  satisfied. Morphism's property~\ref{it:DLPSC_mor-L_d} is satisfied
  by $\SG$ being a symmetry group of $\mathcal{M}$ and, by definition
  of $\check{\IVCM}$~\eqref{eq:PVMs_for_reduced_system-def}, we see
  that~\eqref{eq:DLPSC_mor-IVCM} holds, proving that morphism's
  property~\ref{it:DLPSC_mor-IVCM} holds for $\Upsilon_\DC$.
\end{proof}

When a DLPS is symmetric, constructing the associated reduced system
requires the choice of a discrete connection. The following result
proves that all reduced DLPSs obtained from a DLPS by this procedure
are isomorphic in $\DLPSC$, independently of the discrete connection
chosen.

\begin{proposition}\label{prop:reduction_different_connections_isomorphis_sys}
  Let $\SG$ be a symmetry group of the $\mathcal{M} = (E,L_d,\IVCM)
  \in \Ob_{\DLPSC}$ and $\DCp{1},\DCp{2}$ two discrete connections on the
  underlying principal $\SG$-bundle $\pi^{M,\SG}:M\rightarrow
  M/\SG$. Then, the reduced systems $\mathcal{M}/(\SG,\DCp{1})$ and
  $\mathcal{M}/(\SG,\DCp{2})$ are isomorphic in $\DLPSC$.
\end{proposition}

\begin{proof}
  By Lemma~\ref{le:Upsilon_DC_is_ppal_bundle}, $\Upsilon_{\DCp{1}},
  \Upsilon_{\DCp{2}} : C'(E)\rightarrow C'(\ti{\SG}_E)$ are principal
  $\SG$-bundles. Then, we have the following commutative diagrams of
  smooth maps, where the horizontal arrows are diffeomorphisms (see
  Proposition~\ref{prop:generalized_isomorphisms_associated_to_DC})
  \begin{equation*}
    \xymatrix{ {} & {C'(E)} \ar[d]^{\pi^{C'(E),\SG}} \ar[dl]_{\Upsilon_{\DCp{1}}}\\
      {C'(\ti{\SG}_E)} \ar[r]_{\Phi_{\DCp{1}}^{-1}} & {C'(E)/\SG}
    } 
    \stext{ and }
    \xymatrix{ {C'(E)} \ar[d]_{\pi^{C'(E),\SG}} \ar[dr]^{\Upsilon_{\DCp{2}}} & {}\\
      {C'(E)/\SG} \ar[r]_{\Phi_{\DCp{2}}}  & {C'(\ti{\SG}_E)} 
    }
  \end{equation*}
  Joining the two diagrams we obtain the commutative diagram of smooth
  maps
  \begin{equation*}
    \xymatrix{ {} & {C'(E)} \ar[dl]_{\Upsilon_{\DCp{1}}} \ar[dr]^{\Upsilon_{\DCp{2}}} & 
      {}\\
      {C'(\ti{\SG}_E)} \ar[rr]_{\Phi_{\DCp{2}} \circ \Phi_{\DCp{1}}^{-1}} & {} & 
      {C'(\ti{\SG}_E)}
    }
  \end{equation*}
  The result then follows from
  Lemma~\ref{le:comm_triangle_side_function_imp_side_morph} because
  the horizontal arrow is a diffeomorphism and, by
  Proposition~\ref{prop:Upsilon_DC_is_morphism}, the non-horizontal
  arrows are morphisms in $\DLPSC$.
\end{proof}


\subsection{Dynamics of the reduced discrete Lagrange--Poincar\'e
  system}
\label{sec:dynamics_of_the_reduced_generalized_discrete_mechanical_system}

The following result compares the dynamics of a reduced DLPS to that
of the original symmetric system.

\begin{theorem}\label{thm:two_point_theorem-unconstrained-generalized}
  Let $\SG$ be a symmetry group of the DLPS
  $\mathcal{M}=(E,L_d,\IVCM)$, $\DC$ a discrete connection on the
  principal $\SG$-bundle $\pi^{M,\SG}:M\rightarrow M/\SG$ and
  $\mathcal{M}/(\SG,\DC) =(\ti{G}_E,\check{L}_d,\check{\IVCM})$ the
  corresponding reduced DLPS. If
  $(\epsilon_\cdot,m_\cdot) = ((\epsilon_0,m_1), \ldots,
  (\epsilon_{N-1},m_N))$
  is a discrete path in $C'(E)$, we define a discrete path
  $(v_\cdot,r_\cdot)$ in $C'(\ti{\SG}_E)$ by
  $(v_k,r_{k+1}) := \Upsilon_\DC(\epsilon_k,m_{k+1})$ for
  $k=0,\ldots,N-1$. Then, $(\epsilon_\cdot,m_\cdot)$ is a trajectory
  of $\mathcal{M}$ if and only if $(v_\cdot,r_\cdot)$ is a trajectory
  of $\mathcal{M}/(\SG,\DC)$.
\end{theorem}

\begin{proof}
  As, by Proposition~\ref{prop:Upsilon_DC_is_morphism},
  $\Upsilon_\DC\in\Mor_{\DLPSC}(\mathcal{M},\mathcal{M}/(\SG,\DC))$,
  the result follows from
  Theorem~\ref{thm:morphisms_and_trajectories}.
\end{proof}

\begin{corollary}\label{cor:four_point_thm-generalized}
  In the same setting of
  Theorem~\ref{thm:two_point_theorem-unconstrained-generalized}, the
  following assertions are equivalent.
  \begin{enumerate}
  \item $(\epsilon_\cdot,m_\cdot)$ is a trajectory of $\mathcal{M}$.
  \item
    Equation~\eqref{eq:equation_of_motion-generalized-in_epsilon_m} is
    satisfied.
  \item $(v_\cdot,r_\cdot)$ is a trajectory of
    $\mathcal{M}/(\SG,\DC)$.
  \item For all $k=0,\ldots,N-1$,
    \begin{equation}
      \label{eq:red_equation_of_motion-generalized-in_v_r}
      \begin{split}
        D_1\check{L}_d(v_k,r_{k+1}) + D_1\check{L}_d(v_{k-1},r_{k})
        \check{\IVCM}&((v_{k-1},r_{k}),(v_{k},r_{k+1})) \\ &+
        D_2\check{L}_d(v_{k-1},r_k) dp^{M/\SG}(v_k) = 0
      \end{split}
    \end{equation}
  \end{enumerate}
\end{corollary}
\begin{proof}
  Use Theorem~\ref{thm:two_point_theorem-unconstrained-generalized}
  and Proposition~\ref{prop:eqs_of_motion_gdms}, applied to
  $\mathcal{M}$ and $\mathcal{M}/(\SG,\DC)$.
\end{proof}

The following reconstruction result shows how, knowing the discrete
trajectories of a reduced system, the trajectories of the original
system can be recovered.

\begin{theorem}\label{rem:reconstruction}
  Let $\SG$ be a symmetry group of the DLPS
  $\mathcal{M}=(E,L_d,\IVCM)$, $\DC$ a discrete connection on the
  principal $\SG$-bundle $\pi^{M,\SG}:M\rightarrow M/\SG$ and
  $\mathcal{M}/(\SG,\DC)$ the corresponding reduced DLPS. Let
  $(v_\cdot,r_\cdot)$ be a trajectory of $\mathcal{M}/(\SG,\DC)$ and
  $(\ti{\epsilon}_0,\ti{m}_1)\in C'(E)$ such that
  $\Upsilon_\DC(\ti{\epsilon}_0,\ti{m}_1)=(v_0,r_1)$. Then, there
  exists a unique trajectory $(\epsilon_\cdot, m_\cdot)$ of
  $\mathcal{M}$ such that
  $(\epsilon_0,m_1)=(\ti{\epsilon}_0,\ti{m}_1)$ and
  $\Upsilon_\DC(\epsilon_k, m_{k+1}) = (v_k, r_{k+1})$ for all $k$.
\end{theorem}

\begin{proof}
  By Proposition~\ref{prop:lifting_reduced_discrete_paths}, the
  discrete path $(v_\cdot,r_\cdot)$ lifts to a unique discrete path
  $(\epsilon_\cdot, m_\cdot)$ in $C'(E)$, starting at
  $(\ti{\epsilon}_0,\ti{m}_1)$. Then,
  $(\epsilon_0,m_1)=(\ti{\epsilon}_0,\ti{m}_1)$ and
  $(v_k, r_{k+1}) = \Upsilon_\DC(\epsilon_k, m_{k+1})$ for all $k$.
  By Theorem~\ref{thm:two_point_theorem-unconstrained-generalized},
  $(\epsilon_\cdot, m_\cdot)$ is a trajectory of $\mathcal{M}$.
\end{proof}

\begin{remark}
  Theorem~\ref{rem:reconstruction} asserts that all trajectories of a
  reduced DLPS $\DLPS/(\SG,\DC)$ come from trajectories of the
  original system $\DLPS$. It is possible to give a direct description
  of the reconstruction process in terms of lifting discrete paths
  (see Lemma~\ref{le:properties_Upsilon^2} and
  Proposition~\ref{prop:lifting_reduced_discrete_paths}). This process
  is inductive, so it suffices to describe the initial step, as we do
  next.

  Given a discrete path
  $\rho := ((v_0,r_1),(v_1,r_2))\in C''(\ti{\SG}_E)$ and
  $(\epsilon_0,m_1)\in C'(E)$ such that
  $\Upsilon_\DC(\epsilon_0,m_1) = (v_0,r_1)$, the discrete lift of $\rho$
  starting at $(\epsilon_0,m_1)$ is
  $\hat{\rho}:=((\epsilon_0,m_1), (l^E_g(\epsilon_1'),l^M_g(m_2')))$
  where $(\epsilon_1',m_2')\in \Upsilon_\DC^{-1}(v_1,r_2)$ is
  arbitrary ---as $\Upsilon_\DC$ is onto, it is always possible to
  find such pairs $(\epsilon_1',m_2')$--- and $g\in\SG$ is the unique
  element making $l^M_g(\phi(\epsilon_1')) = m_1$.

  Using~\eqref{eq:ExM_and_tiGxM/G} we see that
  $\pi^{M,\SG}(\phi(\epsilon_1')) =
  p^{M/\SG}(p_1(\Upsilon_\DC(\epsilon_1',m_2'))) = p^{M/\SG}(v_1) =
  r_1 = \pi^{M,\SG}(m_1)$,
  so that $m_1$ and $\phi(\epsilon_1')$ are in the same $\SG$-orbit
  and $g$ is well defined. Furthermore, as
  $\phi(l^E_g(\epsilon_1')) = l^M_g(\phi(\epsilon_1')) = m_1$, we have
  that $\hat{\rho}\in C''(E)$. Finally, as $\Upsilon_\DC$ is
  $\SG$-invariant
  $\Upsilon_\DC(l^E_g(\epsilon_1'),l^M_g(m_2')) =
  \Upsilon_\DC(l^{C'(E)}_g(\epsilon_1',m_2')) =
  \Upsilon_\DC(\epsilon_1',m_2') = (v_1,r_2)$.
  Hence $\Upsilon_\DC^{(2)}(\hat{\rho})=\rho$, and $\hat{\rho}$ is,
  indeed, the corresponding lifted path.
\end{remark}


\section{Example}
\label{sec:example_T2}

In this section we illustrate the reduction techniques introduced so
far with the reduction of an explicit symmetric DLPS and give a
description of the resulting system.


\subsection{The system and a symmetry group}
\label{sec:ex_T2_system_and_symmetry_group}

The starting point is the DMS $(Q,L_d)$, where
$Q:=\C^2-\Delta_{xy}$, for $\Delta_{xy}$ the diagonal in
$\C^2$ and   
\begin{equation}\label{eq:Ld_2body-direct_approach}
  \Lagr(q_0,q_1) := \frac{1}{2h}(\norm{q_1^x-q_0^x}^2 + 
  \norm{q_1^y-q_0^y}^2) - 
  \frac{h}{2} V(\norm{(q_0^y-q_0^x)}^2)
\end{equation}
where $h\neq 0$ is a real constant. This DMS arises as a simple
discretization of the mechanical system consisting of two distinct
unit-mass particles in the plane that interact via a potential $V$,
that only depends on the distance between the particles.

Following
Example~\ref{ex:mechanical_system_as_generalized_mechanical_system},
we associate a DLPS $\DLPS$ to $(Q,\Lagr)$. Take the fiber bundle
$\FBM:\FBTS\rightarrow\FBBS$ to be $id_Q:Q\rightarrow Q$, the
Lagrangian function $\Lagr$ and $\IVCM:=0$. Define the DLPS $\DLPS :=
(Q, \Lagr, \IVCM)$.

Recall that $SE(2)$ is the group of special Euclidean symmetries of
$\R^2\simeq \C$. We can identify $SE(2)$ with
$\{(A,v)\in \C^2 : \abs{A}=1\} = U(1)\times \C$.  The product
operation is $(A_1,v_1)\cdot (A_2,v_2) = (A_1 A_2, A_1 v_2 + v_1)$
with null element $e_{SE(2)} = (1,0)$ and inverse
$(A,v)^{-1} = (A^{-1},-A^{-1}v)$. The subset
$T_2:=\{(1,v)\in U(1)\times \C\}\subset SE(2)$ is a closed normal
subgroup that is isomorphic (as a group) to $\C$.

$SE(2)$ acts naturally on $\C$ by $l^{\C}_{(A,v)}(z) := Az+v$. This
action induces the diagonal action of $SE(2)$ on $Q\times Q$ by
$l^{\C^2}_{(A,v)}(q) :=
(l^{\C}_{(A,v)}(q^x),l^{\C}_{(A,v)}(q^y)) = (Aq^x+v,Aq^y+v)$. Since
$Q$ is preserved by $l^{\C^2}$, $SE(2)$ acts smoothly on $Q$ by the
restricted action, that we denote by $l^Q$.  

It is immediate that $l^Q$ is a free action. Being $U(1)$ compact,
$l^Q$ is a proper action. Then, by
Corollary~\ref{cor:quotient_manifolds-fiber_bundle},
$\pi^{Q,SE(2)}:Q\rightarrow Q/SE(2)$ is a principal $SE(2)$-bundle.

From the previous discussion and the fact that $\FBM=id_Q$ is an
$SE(2)$-invariant trivialization of $\FBM:\FBTS\rightarrow\FBBS$, we
conclude that $SE(2)$ acts on the fiber bundle
$\FBM:\FBTS\rightarrow\FBBS$. As
$\Lagr \circ l^{Q\times Q}_{(A,v)} = \Lagr$ for all $(A,v)\in SE(2)$
and $\IVCM:=0$ is $SE(2)$-equivariant, we conclude that $SE(2)$ is a
symmetry group of $\DLPS$. Being $T_2\subset SE(2)$ a closed subgroup,
it is also a symmetry group of $\DLPS$ by
Proposition~\ref{prop:subgroup_of_sym_group_is_sym_group}.


\subsection{A discrete connection}
\label{sec:reduction_example_T2_discrete_connection}

In this section we use the canonical real inner products in $\C^2$ and
$\C$ to produce a discrete connection $\DCT$ on the principal
$T_2$-bundle $\pi^{Q,T_2}:Q\rightarrow Q/T_2$, following the
construction given in Section 5
of~\cite{ar:fernandez_zuccalli-a_geometric_approach_to_discrete_connections_on_principal_bundles}. The
idea of that construction (in the current setting) is as follows. As
$T_2$ acts by isometries on $\C$ (with the canonical real product),
$T_2$ acts by isometries on $\C^2$ via the diagonal action (with the
canonical real inner product on $\C^2$). This last inner product
induces a $T_2$-invariant riemannian metric on $Q$. The horizontal
subspace for the discrete connection is an open subset of the set of
pairs $(q_0,q_1)\in Q\times Q$ such that $q_1 = \exp^Q(v)$ for some
$v\in T_{q_0}Q$ that is orthogonal to
$T_{q_0}\mathcal{V}(q_0) = \ker(d\pi^{Q,T_2}(q_0))$, the tangent space
to the $l^Q$-orbit through $q_0$.

The previous construction gives that 
\begin{equation*}
  Hor_\DCT:=\{(q_0,q_1)\in Q\times Q: q_0^x+q_0^y=q_1^x+q_1^y\}
\end{equation*}
is a discrete connection on the principal $T_2$-bundle
$\pi^{Q,T_2}:Q\rightarrow Q/T_2$. Straightforward computations show
that the discrete connection form is
\begin{equation}\label{eq:DCT-def}
  \begin{split}
    \DCT(q_0,q_1) =& \left(1,\frac{1}{2}\big((q_1^x+q_1^y) -
      (q_0^x+q_0^y) \big) \right).
  \end{split}
\end{equation}

\begin{remark}
  Other discrete connections can be considered on the principal
  $T_2$-bundle $\pi^{Q,T_2}:Q\rightarrow Q/T_2$. For instance, one
  can define an affine discrete connection whose horizontal space
  consists of a level manifold of the discrete momentum function (see
  Remark~\ref{rem:discrete_momenta}), This would lead to the reduction
  procedure considered in Section 11
  of~\cite{ar:fernandez_tori_zuccalli-lagrangian_reduction_of_discrete_mechanical_systems}
  for DMS with horizontal symmetries.
\end{remark}


\subsection{The reduced system}
\label{sec:reduction_example_T2_reduced_system}

Using the discrete connection $\DCT$ we construct the reduced system
$\DLPS/(T_2,\DCT) = (\ti{(T_2)_Q},\check{\Lagr},\check{\IVCM})$, in
such a way that
\begin{equation*}
  \DLPSMor_{\DCT}(q_0,q_1) = (\pi^{Q\times T_2, T_2}(q_0,\DCT(q_0,q_1)),
  \pi^{Q,T_2}(q_1))
\end{equation*}
is in $\Mor_{\DLPS}(\DLPS,\DLPS/(T_2,\DCT))$.  Below we give an
explicit DLPS $\DLPS'$, isomorphic to $\DLPS/(T_2,\DCT)$.

\subsubsection{An alternative model}
\label{sec:an_alternative_model}

Define $\FBM' : \FBTS' \rightarrow \FBBS'$ by
$p_1: \C^* \times T_2\rightarrow \C^*$, so that
$(\FBTS', \FBBS', \FBM', T_2)$ is a trivial fiber bundle. Define the
map $\Upsilon:C'(\FBTS)\rightarrow C'(\FBTS')$ by
\begin{equation*}
  \begin{split}
    \Upsilon(q_0,q_1):=& \bigg(\bigg(\frac{1}{\sqrt{2}} (q_0^x-q_0^y),
    \bigg(1,\frac{1}{2} ((q_1^x+q_1^y)- (q_0^x+q_0^y)) \bigg)\bigg),
    \frac{1}{\sqrt{2}} (q_1^x-q_1^y) \bigg).
  \end{split}
\end{equation*}
Clearly $\Upsilon$ is smooth and $T_2$-invariant.

We intend to define a DLPS structure on $\FBM':\FBTS' \rightarrow \FBBS'$
in such a way that $\Upsilon$ is a morphism. This forces us to define
\begin{equation}\label{eq:reduction_example-SE(2)-reduced_L}
  \begin{split}
    \Lagr'((r_0,z_0),r_1) =& \frac{1}{2h}\big( 2\norm{z_0^s}^2 +
    \norm{r_1 - r_0}^2\big) - \frac{h}{2} V\big(2\norm{r_0}^2\big)
  \end{split}
\end{equation}
and
\begin{equation}\label{eq:reduction_example-SE(2)-reduced_IVCM}
  \begin{split}
    \IVCM'\big(((r_0,z_0),r_1),((r_1,z_1),r_2)\big)\left(b \pd{}{r_1}
      + c\pd{}{z_1}\right) = - c \pd{}{z_0},
  \end{split}
\end{equation}
where $(r,z) \in \C^*\times \C$. A number of computations confirm that
$\DLPS' := (\FBTS', \Lagr', \IVCM')$ is a DLPS and that
$\Upsilon\in \Mor_{\DLPSC}(\DLPS, \DLPS')$.

By the $T_2$-invariance of $\Upsilon$, there is a smooth map
$\ti{\Upsilon}$ such that the diagram
\begin{equation*}
  \xymatrix{
    {C'(\FBTS')} 
    & {Q\times Q} \ar[l]_{\Upsilon} \ar[d]^{\pi^{Q\times Q,T_2}}
    \ar@/^/[dr]^{\Upsilon_\DCT} 
    & {} \\
    {} & {(Q\times Q)/T_2} \ar[lu]^{\ti{\Upsilon}} & 
    {\ti{(T_2)}_Q} \ar[l]^(.35){\Psi_{\DCT}}_(.35){\sim}
  }
\end{equation*}
is commutative. Define the smooth map
$\check{\Upsilon}:=\ti{\Upsilon}\circ \Psi_\DCT$. As $\Upsilon$ is
onto and satisfies
$\Upsilon^{-1}(\Upsilon(q_0,q_1)) = l^{Q\times Q}_{T_2}(q_0,q_1)$ for
all $(q_0,q_1)\in Q\times Q$, which easily implies that
$\ti{\Upsilon}$ is one to one, $\check{\Upsilon}$ is a
diffeomorphism. By
Lemma~\ref{le:comm_triangle_side_function_imp_side_morph},
$\check{\Upsilon}$ is an isomorphism in $\DLPSC$.

All together, $\DLPS'$ is an explicit model for the reduced DLPS
$\DLPS/(T_2,\DCT)$.


\subsubsection{Equations of motion}
\label{sec:equations_of_motion-T2}

Trajectories in $\DLPS'$ are found
using~\eqref{eq:equation_of_motion-generalized-in_epsilon_m}.
Evaluating the left side
of~\eqref{eq:equation_of_motion-generalized-in_epsilon_m} on an
arbitrary tangent vector $b\pd{}{r_1}+c\pd{}{z_1} \in T_{(r_1,z_1)}E'$
and computing the corresponding derivatives we obtain the equations
\begin{equation*}
  \re((z_1-z_0)\conj{c}) = 0 \stext{ and } 
  \re(((r_1-r_0) - (r_2-r_1) - 2h^2 V'(2\norm{r_1}^2)r_1)\conj{b}) = 0
\end{equation*}
which, due to de arbitrariness of $b,c\in \C$, lead to
\begin{equation*}
  z_1=z_0 \stext{ and } r_2 = 2r_1 -r_0 - 2h^2 V'(2\norm{r_1}^2)r_1.
\end{equation*}
It should be noticed that the $z_k$ are (proportional to) the velocity
of the center of mass of the original system, which explains the fact
that $z_k$ is constant for a trajectory, while $r_k$ gives the
position of one particle relative to the other.

\begin{remark}\label{rem:example_T2-U(1)_residual_action}
  There is a $U(1)$ action on $\DLPS'$ given by
  $l^{E'}_{(A,0)}(r,z) := (Az,Az)$. This action is a ``residue'' of
  the original $SE(2)$ action on $\DLPS$. This action is, indeed, a
  symmetry of $\DLPS'$ and can be reduced using the same
  techniques. During the rest of the paper we will show that, under
  appropriate conditions, this second reduction produces a system that
  is isomorphic to $\DLPS/SE(2)$.
\end{remark}


\section{Reduction in two stages}
\label{sec:reduction_in_two_stages}

Let $\SG$ be a symmetry group of the DLPS $\mathcal{M}$ and $H\subset
\SG$ a normal closed subgroup.  In this section we apply the reduction
theory of $\mathcal{M}$ by $H$ and, provided that $\SG/H$ is a
symmetry group of $\mathcal{M}/H$, perform a second reduction. Last,
we compare the two step reduction with the reduction $\mathcal{M}/\SG$
performed in one step.


\subsection{Residual symmetry group}
\label{sec:residual_symmetry_group}

\begin{lemma}\label{le:G_acts_on_FB_imp_G/H_acts_on_FB/H}
  Let $\SG$ act on the fiber bundle $(E,M,\phi,F)$ and $H\subset \SG$
  be a closed normal subgroup. Define mappings
  \begin{equation}\label{eq:action_of_residual_symmetry-def}
    \begin{split}
      l^{\ti{H}_E}_{\pi^{\SG,H}(g)}(\pi^{E\times H,H}(\epsilon,w))
      :=& \pi^{E\times H,H}(l^E_g(\epsilon),l^\SG_g(w)),\\
      l^{M/H}_{\pi^{\SG,H}(g)}(\pi^{M,H}(m)) :=& \pi^{M,H}(l^M_g(m)).
    \end{split}
  \end{equation}
  Then $l^{\ti{H}_E}$, $l^{M/H}$ and the trivial right action on
  $F\times H$ define a $\SG/H$-action on the fiber bundle
  $(\ti{H}_E,M/H,p^{M/H},F\times H)$.
\end{lemma}

\begin{proof}
  By Lemma~\ref{le:group_acts_on_FB_subgroup_acts_too}, $H$ acts on
  $(E,M,\phi,F)$. Hence, by
  Example~\ref{ex:extended_associated_bundle-fiber_bundle}, the
  conjugate bundle $(\ti{H}_E,M/H,p^{M/H},F\times H)$ is a fiber bundle.
  Consider the $\SG$-action on $E\times H$ defined by
  \begin{equation}
    \label{eq:G-action_on_ExH}
    l^{E\times H}_g(\epsilon,w) :=(l^E_g(\epsilon),l^\SG_g(w)), 
  \end{equation}
  where $l^\SG_g(w) = g w g^{-1}$. Being a product of smooth actions,
  it is a smooth action. Also, as $l^E$ is a free and proper
  $\SG$-action (see
  Lemma~\ref{le:G_actions_on_ppal_bundles_are_proper}), the same is
  true for $l^{E\times H}$. Therefore, by
  Lemma~\ref{le:quotient_action_properties}, the function
  $l^{\ti{H}_E}$ given in~\eqref{eq:action_of_residual_symmetry-def}
  defines a smooth, free and proper $\SG/H$-action on $\ti{H}_E =
  (E\times H)/H$.

  Analogously, as the $\SG$-action $l^M$ on $M$ is smooth, free and
  proper (this last fact by
  Lemma~\ref{le:G_actions_on_ppal_bundles_are_proper} ),
  Lemma~\ref{le:quotient_action_properties} proves that $l^{M/H}$
  defined in~\eqref{eq:action_of_residual_symmetry-def} is a smooth,
  free and proper $\SG/H$-action on $M/H$. Then,
  Corollary~\ref{cor:quotient_manifolds-fiber_bundle} proves that the
  quotient map $\pi^{M/H,\SG/H}:M/H\rightarrow (M/H)/(\SG/H)$ is a
  principal $\SG/H$-bundle.

  It follows from the $\SG$-equivariance of $\phi:E\rightarrow M$ that
  $p^{M/H}:\ti{H}_E\rightarrow M/H$ is $\SG/H$-equivariant for the
  $\SG/H$-actions defined above.

  For the rest of the proof we construct $\SG/H$-equivariant
  trivializing charts of the fiber bundle
  $(\ti{H}_E,M/H,p^{M/H},F\times H)$. Since this is a local problem
  (in the base of the bundle) we will assume that
  $\phi:E\rightarrow M$ is $p_1:M\times F\rightarrow M$ with the $\SG$
  action on $E$ given by $l^E_g(m,f) = (l^M_g(m),r^F_{g^{-1}}(f))$.
  Even more, as $\pi^{M,\SG}:M\rightarrow M/\SG$ is a principal
  $\SG$-bundle, by shrinking $M$ further we may assume that
  $\pi^{M,\SG}$ is a trivial principal $\SG$-bundle, that is,
  $M=(M/\SG)\times \SG$, where the $\SG$-action is given by left
  multiplication on $\SG$. All together, we have that
  $\phi:E\rightarrow M$ is given by the projection on the first two
  components
  $p_{12}:(M/\SG)\times \SG\times F\rightarrow (M/\SG)\times \SG$ and
  the $\SG$-actions are
  $l^E_g(\pi^{M,\SG}(m),g',f)=(\pi^{M,\SG}(m),g g',r^F_{g^{-1}}(f))$
  and $l^M_g(\pi^{M,\SG}(m),g')=(\pi^{M,\SG}(m),g g')$.

  Define the map
  \begin{equation*}
    \sigma:\underbrace{(M/\SG)\times \SG\times F}_{=E}\times H\rightarrow 
    \underbrace{(M/\SG)\times \SG}_{=M} \times F\times H
  \end{equation*}
  by
  \begin{equation*}
    \sigma(\pi^{M,\SG}(m),g',f,h) := (\pi^{M,\SG}(m), g', 
    r^F_{g'}(f), (g')^{-1} h g').
  \end{equation*}
  A quick check shows that $\sigma$ is a $\SG$-equivariant
  diffeomorphism for the left actions $l^{E\times H}$ and
  $l^{M\times F\times H}_g(m,f,h):=(l^M_g(m),f,h)$, for all
  $g\in\SG$. Restricting those actions to $H\subset \SG$ and applying
  Corollary~\ref{cor:quotient_maps-G_equivariant_map}, $\sigma$
  induces a map
  $\Phi^{\ti{H}_E}: \ti{H}_E\rightarrow (M/H)\times F\times H$. It is
  easy to see that $\Phi^{\ti{H}_E}$ provides a (local)
  $\SG/H$-equivariant trivialization of
  $p^{M/H}:\ti{H}_E\rightarrow M/H$, concluding the proof of the fact
  that $\SG/H$ acts on the fiber bundle
  $(\ti{H}_E,M/H,p^{M/H},F\times H)$.
\end{proof}

\begin{lemma}\label{le:G-equivariance_H_connection_form}
  Let $\SG$ be a Lie group acting on $Q$ by the action $l^Q$ in such a
  way that $\pi^{Q,\SG}:Q\rightarrow Q/\SG$ is a principal
  $\SG$-bundle. Assume that $H\subset\SG$ is a closed and normal
  subgroup and that $\DC$ is a discrete connection on the principal
  $H$-bundle $\pi^{Q,H}:Q\rightarrow Q/H$ whose domain is
  $\SG$-invariant for the diagonal $\SG$-action $l^{Q\times Q}$. Then,
  the following assertions are equivalent.
  \begin{enumerate}
  \item For each $g\in \SG$ and $(q_0,q_1)$ in the domain of $\DC$,
    \begin{equation}
      \label{eq:G-equivariance_H_connection_form}
      \DC(l^Q_g(q_0),l^Q_g(q_1)) = g\DC(q_0,q_1)g^{-1}.
    \end{equation}
  \item The submanifold $Hor_\DC\subset Q\times Q$ is $\SG$-invariant
    for the $\SG$-action $l^{Q\times Q}$.
  \end{enumerate}
\end{lemma}

\begin{proof}
  Recall that $(q_0,q_1)$ in the domain of $\DC$ is in $Hor_{\DC}$ if
  and only if $\DC(q_0,q_1)=e$. Assume
  that~\eqref{eq:G-equivariance_H_connection_form} holds, for each
  $g\in\SG$. Let $(q_0,q_1)\in Hor_\DC$. Then, for any $g\in \SG$,
  \begin{equation*}
    \DC(l^Q_g(q_0),l^Q_g(q_1)) = g \DC(q_0,q_1) g^{-1} = g g^{-1} = e, 
  \end{equation*}
  showing that $l^{Q\times Q}_g(q_0,q_1) = (l^Q_g(q_0),l^Q_g(q_1))\in
  Hor_\DC$.
  
  Conversely, if $(q_0,q_1)$ is in the domain of $\DC$, which is
  $\SG$-invariant, we have that $(l^Q_g(q_0),l^Q_g(q_1))$ is also in
  the domain of $\DC$. Then $\DC(l^Q_g(q_0),l^Q_g(q_1)) = h\in H$
  if and only if 
  \begin{equation}
    \label{eq:condition_for_h_to_be_DC_value}
    (l^Q_g(q_0),l^Q_{h^{-1}}(l^Q_g(q_1)))\in Hor_\DC.
  \end{equation}
  But, as $(q_0,l^Q_{\DC(q_0,q_1)^{-1}}(q_1))\in Hor_\DC$ and
  $Hor_\DC$ is $\SG$-invariant, we have that
  \begin{equation*}
    (l^Q_g(q_0),l^Q_{g\DC(q_0,q_1)^{-1}g^{-1}}(l^Q_g(q_1))) = 
    l^{Q\times Q}_g(q_0,l^Q_{\DC(q_0,q_1)^{-1}}(q_1)) \in Hor_\DC
  \end{equation*}
  so that $h:=g\DC(q_0,q_1)^{-1}g^{-1}$
  satisfies~\eqref{eq:condition_for_h_to_be_DC_value}. As the element
  of $\SG$ with this property is unique, we conclude
  that~\eqref{eq:G-equivariance_H_connection_form} holds.
\end{proof}

\begin{proposition}\label{prop:symmetry_group_induces_residual_action}
  Let $\SG$ be a symmetry group of
  $\mathcal{M}=(E,L_d,\IVCM)\in\Ob_{\DLPSC}$ and $H\subset \SG$ be a
  closed and normal subgroup. Choose a discrete connection $\DC$ of
  the principal $H$-bundle $\pi^{Q,H}:Q\rightarrow Q/H$ such that
  either one of the conditions in
  Lemma~\ref{le:G-equivariance_H_connection_form} holds.  Then $\SG/H$
  is a symmetry group of the DLPS
  $\mathcal{M}^H := \mathcal{M}/(H,\DC)
  =(\ti{H}_E,\check{L}_d,\check{\IVCM})$
  obtained by the reduction of $\mathcal{M}$ using $\DC$.
\end{proposition}

\begin{proof}
  By Lemma~\ref{le:G_acts_on_FB_imp_G/H_acts_on_FB/H}, $\SG/H$ acts on
  the fiber bundle $(\ti{H}_E,M/H,p^{M/H},F\times H)$. Recall that
  $\Upsilon_\DC:C'(E)\rightarrow C'(\ti{H}_E)$ is defined by
  \begin{equation*}
    \Upsilon_\DC(\epsilon_0,m_1) := (\pi^{E\times
    H,H}(\epsilon_0,\DC(\phi(\epsilon_0),m_1)),\pi^{M,H}(m_1)).
  \end{equation*}
  Unraveling the definitions and
  taking~\eqref{eq:G-equivariance_H_connection_form} into account, we
  have that
  \begin{equation}\label{eq:equivariance_of_Upsilon_DC}
    \begin{split}
      \Upsilon_\DC \circ l^{C'(E)}_g  =
      l^{\ti{H}_E\times
        (M/H)}_{\pi^{\SG,H}(g)} \circ \Upsilon_\DC \stext{ for all } g\in\SG.
    \end{split}
  \end{equation}
  Then, as $\check{L}_d$ satisfies $\check{L}_d\circ \Upsilon_\DC =
  L_d$, and $L_d$ is $\SG$-invariant for the $\SG$-action $l^{E\times
    M}$, we have that
  \begin{equation*}
    \begin{split}
      \check{L}_d \circ l^{\ti{H}_E\times
        (M/H)}_{\pi^{\SG,H}(g)} \circ \Upsilon_\DC =
      \check{L}_d \circ \Upsilon_\DC \circ l^{E\times M}_g =
      L_d \circ l^{E\times M}_g = L_d =
      \check{L}_d\circ \Upsilon_\DC.
    \end{split}
  \end{equation*}
  Hence, as $\Upsilon_\DC$ is onto, $\check{L}_d$ is $\SG/H$-invariant
  for the $\SG/H$-action $l^{\ti{H}_E\times (M/H)}$.

  Differentiating the first component
  of~\eqref{eq:equivariance_of_Upsilon_DC} we see that, for
  $v_0=\Upsilon_\DC(\epsilon_0,m_1)$,
  \begin{equation}\label{eq:consequence_equivariance_Upsilon_DC-D1}
    D_1(p_1\circ \Upsilon_\DC)(l^{C'(E)}_g(\epsilon_0,m_1)) 
    dl^{E}_g(\epsilon_0)(\delta \epsilon_0) =
    d l^{\ti{H}_E}_{\pi^{\SG,H}(g)}(v_0) D_1(p_1\circ \Upsilon_\DC)(\epsilon_0,m_1) 
    (\delta \epsilon_0)
  \end{equation}
  and 
  \begin{equation}\label{eq:consequence_equivariance_Upsilon_DC-D2}
    D_2(p_1\circ \Upsilon_\DC)(l^{C'(E)}_g(\epsilon_0,m_1))
    dl^{M}_g(m_1)(\delta m_1) =
    d l^{\ti{H}_E}_{\pi^{\SG,H}(g)}(v_0) D_2(p_1\circ \Upsilon_\DC)(\epsilon_0,m_1)
    (\delta m_1).
  \end{equation}

  Now, fix $\nu=((v_0,r_1),(v_1,r_2))\in C''(\ti{H}_E)$ and take
  $\eta=((\epsilon_0,m_1),(\epsilon_1,m_2))\in C''(E)$ such that
  $\Upsilon_\DC^{(2)}(\eta) = \nu$. Then, by
  point~\ref{it:properties_Upsilon^2-isomorphism} of
  Lemma~\ref{le:properties_Upsilon^2}, any $\delta v_1\in
  T_{v_1}(\ti{H}_E)$ is of the form $\delta v_1 = D_1(p_1\circ
  \Upsilon_\DC)(\epsilon_1,m_2)(\delta \epsilon_1)$ for a unique
  $\delta \epsilon_1\in T_{\epsilon_1}E$. Then,
  using~\eqref{eq:consequence_equivariance_Upsilon_DC-D1}, we obtain
  \begin{equation*}
    \begin{split}
      D_1(p_1\circ
      \Upsilon_\DC)(l^{C'(E)}_g(\epsilon_1,m_2))dl^E_g(\epsilon_1)
      (\delta \epsilon_1) =&  dl^{\ti{H}_E}_{\pi^{\SG,H}(g)}(v_1)(\delta v_1),
    \end{split}
  \end{equation*}
  which shows that the unique element of $T_{l^E_g(\epsilon_1)}E$ that
  represents $dl^{\ti{H}_E}_{\pi^{\SG,H}(g)}(v_1)(\delta v_1) \in
  T_{l^{\ti{H}_E}_{\pi^{\SG,H}(g)}(v_1)}\ti{H}_E$ is $dl^E_g(\epsilon_1)
  (\delta \epsilon_1)$. Also, notice that
  using~\eqref{eq:equivariance_of_Upsilon_DC} we obtain
  \begin{equation*}
    \begin{split}
      \Upsilon_\DC^{(2)}(l^{C''(E)}_g(\eta))
      =&
      l^{C''(\ti{H}_E)}_{\pi^{\SG,H}(g)}(\Upsilon_\DC^{(2)}(\eta)) =
      l^{C''(\ti{H}_E)}_{\pi^{\SG,H}(g)}(\nu).
    \end{split}
  \end{equation*}
  We use this information to compute $\check{\IVCM}\circ
  l^{p_3^*T(\ti{H}_E)}_{\pi^{\SG,H}(g)}$. For any $g\in\SG$,
  \begin{equation}\label{eq:equivariance_reduced_IVCM-1}
    \begin{split}
      (\check{\IVCM}\circ
      l^{p_3^*T(\ti{H}_E)}_{\pi^{\SG,H}(g)})(\nu,&\delta v_1) =
      \check{\IVCM}(l^{C''(\ti{H}_E)}_{\pi^{\SG,H}(g)}(\nu))
      (dl^{\ti{H}_E}_{\pi^{\SG,H}(g)}(v_1)(\delta v_1)) \\=&
      D_1(p_1\circ
      \Upsilon_\DC)(l^{C'(E)}_g(\epsilon_0,m_1))(\IVCM(l^{C''(E)}_g(\eta))
      (dl^E_g(\epsilon_1) (\delta \epsilon_1))) \\
      &+ D_2(p_1\circ \Upsilon_\DC)(l^{C'(E)}_g(\epsilon_0,m_1))
      (d\phi(l^E_g(\epsilon_1))(dl^E_g(\epsilon_1) (\delta
      \epsilon_1))).
    \end{split}
  \end{equation}
  Using the $\SG$-equivariance of $\IVCM$
  and~\eqref{eq:consequence_equivariance_Upsilon_DC-D1}, we have
  \begin{equation*}
    \begin{split}
      D_1(p_1\circ \Upsilon_\DC)(l^{C'(E)}_g(\epsilon_0,m_1))&
      (\IVCM(l^{C''(E)}_g(\eta)(dl^E_g(\epsilon_1)(\delta
      \epsilon_1)))) \\=& D_1(p_1\circ
      \Upsilon_\DC)(l^{C'(E)}_g(\epsilon_0,m_1))
      (l^{TE}_g(\epsilon_0)(\IVCM(\eta)(\delta \epsilon_1)))\\=&
      D_1(p_1\circ \Upsilon_\DC)(l^{C'(E)}_g(\epsilon_0,m_1))
      (dl^{E}_g(\epsilon_0)(\IVCM(\eta)(\delta \epsilon_1)))\\=& d
      l^{\ti{H}_E}_{\pi^{\SG,H}(g)}(v_0) D_1(p_1\circ
      \Upsilon_\DC)(\epsilon_0,m_1)
      (\IVCM(\eta)(\delta \epsilon_1))
    \end{split}
  \end{equation*}
  and, using~\eqref{eq:consequence_equivariance_Upsilon_DC-D2},
  \begin{equation*}
    \begin{split}
      D_2(p_1\circ
      \Upsilon_\DC)&(l^{C'(E)}_g(\epsilon_0,m_1))(d\phi(l^E_g(\epsilon_1))
      (dl^E_g(\epsilon_1)(\delta \epsilon_1))) \\=& D_2(p_1\circ
      \Upsilon_\DC)(l^{C'(E)}_g(\epsilon_0,m_1))(dl^M_g(m_1)(d\phi(\epsilon_1)
      (\delta \epsilon_1))) \\=& d l^{\ti{H}_E}_{\pi^{\SG,H}(g)}(v_0)
      D_2(p_1\circ \Upsilon_\DC)(\epsilon_0,m_1)(d\phi(\epsilon_1)
      (\delta \epsilon_1)).
    \end{split}
  \end{equation*}
  Going back to~\eqref{eq:equivariance_reduced_IVCM-1}, we obtain
  \begin{equation*}
    \begin{split}
      \check{\IVCM}\circ
      l^{p_3^*T(\ti{H}_E)}_{\pi^{\SG,H}(g)}(\nu,\delta v_1) =&
      \check{\IVCM}(l^{C''(\ti{H}_E)}_{\pi^{\SG,H}(g)}(\nu))(
      dl^{\ti{H}_E}_{\pi^{\SG,H}(g)}(v_1)(\delta v_1)) \\=&
      dl^{\ti{H}_E}_{\pi^{\SG,H}(g)}(v_0)\big( D_1(p_1\circ
      \Upsilon_\DC)(\epsilon_0,m_1)
      (\IVCM(\eta)(\delta \epsilon_1)) \\
      & \phantom{dl^{\ti{H}_E}_{\pi^{\SG,H}(g)}(v_0)\big(} +
      D_2(p_1\circ \Upsilon_\DC)(\epsilon_0,m_1)(d\phi(\epsilon_1)
      (\delta \epsilon_1)) \big) \\=&
      dl^{\ti{H}_E}_{\pi^{\SG,H}(g)}(v_0)\big(
      \check{\IVCM}(\nu)(\delta v_1)\big),
    \end{split}
  \end{equation*}
  showing that $\check{\IVCM}$ is $\SG/H$-equivariant. Hence,
  $\SG/H$ is a symmetry group of $\mathcal{M}^H$.
\end{proof}

\begin{example}
  In Section~\ref{sec:ex_T2_system_and_symmetry_group} we introduced a
  DLPS $\DLPS$ and saw that $SE(2)$ was one of its symmetry groups.
  As $T_2\subset SE(2)$ is a closed normal subgroup, it was also a
  symmetry group of $\DLPS$. A simple verification shows that the
  discrete connection form $\DCT$ defined in~\eqref{eq:DCT-def}
  satisfies~\eqref{eq:G-equivariance_H_connection_form} for
  $\SG:=SE(2)$ and $H:=T_2$ so that, by
  Proposition~\ref{prop:symmetry_group_induces_residual_action},
  $SE(2)/T_2$ is a symmetry group of $\DLPS/(T_2,\DCT)\simeq \DLPS'$.
  As $SE(2)/T_2 \simeq U(1)$, we see that this fact is already
  suggested in Remark~\ref{rem:example_T2-U(1)_residual_action}.
\end{example}


\subsection{Comparison with reduction by the full symmetry group}
\label{sec:comparison_with_reduction_by_the_full_symmetry_group}

Here we consider a symmetry group $\SG$ of
$\mathcal{M}=(E,L_d,\IVCM)\in\Ob_{\DLPSC}$. Fixing a discrete
connection $\DCp{\SG}$ on the principal $\SG$-bundle
$\pi^{M,\SG}:M\rightarrow M/\SG$ we have the reduced system
$\mathcal{M}^\SG := \mathcal{M}/(\SG,\DCp{\SG})$. When $H\subset \SG$
is a closed and normal subgroup, by
Proposition~\ref{prop:subgroup_of_sym_group_is_sym_group}, $H$ is a
symmetry group of $\mathcal{M}$ and, when $\DCp{H}$ is a discrete
connection on the principal $H$-bundle $\pi^{M,H}:M\rightarrow M/H$ we
have the reduced system $\mathcal{M}^H :=
\mathcal{M}/(H,\DCp{H})$.
Furthermore, when $\DCp{H}$ satisfies any one of the conditions in
Lemma~\ref{le:G-equivariance_H_connection_form}, by
Proposition~\ref{prop:symmetry_group_induces_residual_action}, $\SG/H$
is a symmetry group of $\mathcal{M}^H$. Fixing a discrete connection
$\DCp{\SG/H}$ on the principal $\SG/H$-bundle
$\pi^{M/H,\SG/H}:M/H\rightarrow \frac{M/H}{\SG/H}$, we have the
reduced system
$\mathcal{M}^{\SG/H} := \mathcal{M}^H/(\SG/H,\DCp{\SG/H})$. The
following diagram depicts the relation between the different DLPSs and
morphisms.
\begin{equation*}
  \xymatrix{
    {} & {} & {\mathcal{M}} \ar[ddll]_{\Upsilon_{\DCp{\SG}}} 
    \ar[dr]^{\Upsilon_{\DCp{H}}} & {} & {}\\
    {} & {} & {} & {\mathcal{M}^H} \ar[dr]^{\Upsilon_{\DCp{\SG/H}}} & {}\\
    {\mathcal{M}^{\SG}} & {} & {} & {} & {\mathcal{M}^{\SG/H}}
  }
\end{equation*}
At the ``geometric level'', the corresponding spaces and smooth maps
are
\begin{equation*}
  \xymatrix{
    {} & {} & {C'(E)} \ar[ddll]_{\Upsilon_{\DCp{\SG}}} 
    \ar[dr]^{\Upsilon_{\DCp{H}}} & {} & {}\\
    {} & {} & {} & {C'(\ti{H}_E)} \ar[dr]^{\Upsilon_{\DCp{\SG/H}}} & {}\\
    {C'(\ti{\SG}_E)} & {} & {} & {} & 
    {C'(\ti{\SG/H}_{\ti{H}_E})}
  }
\end{equation*}
We can enlarge the previous diagram by adding the different
diffeomorphisms $\Phi_\DC$ introduced in
Proposition~\ref{prop:generalized_isomorphisms_associated_to_DC} and
by taking into account the commutative
diagram~\eqref{eq:diagram_ExM_to_reduced}. The resulting diagram
follows.
\begin{equation}\label{eq:two_step_reduction-geometric_diagram}
  \xymatrix{
    {} & {} & {C'(E)} \ar@/_/[ddll]_{\Upsilon_{\DCp{\SG}}} 
    \ar[ddl]_(.7){\pi^{C'(E),\SG}} 
    \ar[d]^{\pi^{C'(E),H}} \ar@/^/[dr]^{\Upsilon_{\DCp{H}}} & {} & {}\\
    {} & {} & {\frac{C'(E)}{H}} \ar[d]^{\pi^{\frac{C'(E)}{H},\SG/H}} 
    \ar[dl]^(.45){F_1} \ar[r]^(.4){\Phi_{\DCp{H}}}_(.4){\sim} & 
    {C'(\ti{H}_E)} \ar[d]_(.25){\pi^{C'(\ti{H}_E),\SG/H}} 
    \ar[dr]^{\Upsilon_{\DCp{\SG/H}}} & {}\\
    {C'(\ti{\SG}_E)} & 
    {\frac{C'(E)}{\SG}} \ar[l]^(.4){\Phi_{\DCp{\SG}}}_(.4)\sim & 
    {\frac{\frac{C'(E)}{H}}{\SG/H}} \ar[l]^{F_2}_{\sim} 
    \ar[r]_{\check{\Phi_{\DCp{H}}}}^(.45){\sim} & 
    {\frac{C'(\ti{H}_E)}{\SG/H}} \ar[r]_(.45){\Phi_{\DCp{\SG/H}}}^(.4)\sim & 
    {C'(\ti{\SG/H}_{\ti{H}_E})}\\
  }
\end{equation}

The following result introduces the new functions that appear in
diagram~\eqref{eq:two_step_reduction-geometric_diagram} and explores
their basic properties.

\begin{lemma}\label{le:two_step_diagram-properties}
  Under the previous conditions,
  \begin{enumerate}
  \item \label{it:two_step_diagram-Phi_check}
    $\Phi_{\DCp{H}}:\frac{C'(E)}{H}\rightarrow C'(\ti{H}_E)$ is a
    $\SG/H$-equivariant diffeomorphism. Hence it induces a smooth
    diffeomorphism
    $\check{\Phi_{\DCp{H}}}:\frac{\frac{C'(E)}{H}}{\SG/H} \rightarrow
    \frac{C'(\ti{H}_E)}{\SG/H}$.
  \item \label{it:two_step_diagram-F_1_def}
    $\pi^{C'(E),\SG}:C'(E)\rightarrow \frac{C'(E)}{\SG}$ is a smooth
    $H$-invariant map, hence it induces a smooth map
    $F_1:\frac{C'(E)}{H}\rightarrow \frac{C'(E)}{\SG}$.
  \item\label{it:two_step_diagram-F_2_def}
    $F_1:\frac{C'(E)}{H}\rightarrow \frac{C'(E)}{\SG}$ is a smooth
    $\SG/H$-invariant map, hence it induces a smooth map
    $F_2:\frac{\frac{C'(E)}{H}}{\SG/H} \rightarrow
    \frac{C'(E)}{\SG}$. Furthermore, $F_2$ is a diffeomorphism.
  \item \label{it:two_step_diagram-commutative} The
    diagram~\eqref{eq:two_step_reduction-geometric_diagram} is
    commutative.
  \end{enumerate}
\end{lemma}

\begin{proof}
  Unraveling the definitions and recalling that $\DCp{H}$
  satisfies~\eqref{eq:G-equivariance_H_connection_form}, we see that
  $\Phi_{\DCp{H}}$ is $\SG/H$-equivariant. As, by
  Proposition~\ref{prop:generalized_isomorphisms_associated_to_DC},
  $\Phi_{\DCp{H}}$ is smooth, we conclude from
  Corollary~\ref{cor:quotient_maps-G_equivariant_map} that the induced
  map $\check{\Phi_{\DCp{H}}}$ is smooth. Furthermore, as
  $\Phi_{\DCp{H}}$ is also a diffeomorphism by
  Proposition~\ref{prop:generalized_isomorphisms_associated_to_DC},
  its inverse is also $\SG/H$-equivariant, so that
  $\check{\Phi_{\DCp{H}}}$ is a diffeomorphism. Hence
  point~\ref{it:two_step_diagram-Phi_check} in the statement is
  proved. By construction the square in
  diagram~\eqref{eq:two_step_reduction-geometric_diagram} is
  commutative.

  Point~\ref{it:two_step_diagram-F_1_def} follows immediately using
  the $H$-invariance of $\pi^{C'(E),\SG}$ and
  Corollary~\ref{cor:quotient_maps-G_invariant_map}. Furthermore, it
  is immediate that $F_1$ is $\SG/H$-invariant, and the same argument
  proves that $F_2$ is a well defined smooth map. It is easy to check
  that the map $\pi^{\frac{C'(E)}{H},\SG/H} \circ \pi^{C'(E),H} :
  C'(E)\rightarrow \frac{\frac{C'(E)}{H}}{\SG/H}$ is smooth and
  $\SG$-invariant, so it induces a smooth inverse of $F_2$, showing
  that $F_2$ is a diffeomorphism. This proves
  point~\ref{it:two_step_diagram-F_2_def}. By definition, the two
  triangles involving $F_1$ in
  diagram~\eqref{eq:two_step_reduction-geometric_diagram} are
  commutative.

  The commutativity of the three remaining triangles in
  diagram~\eqref{eq:two_step_reduction-geometric_diagram} is due to
  the commutativity of diagram~\eqref{eq:diagram_ExM_to_reduced}.
\end{proof}

\begin{theorem}\label{th:isomorphism_reduced_spaces_in_stages}
  Consider the data given at the beginning of this section.  Let
  $F:C'(\ti{\SG/H}_{\ti{H}_E})\rightarrow C'(\ti{\SG}_E)$ be defined
  by the bottom row of
  diagram~\eqref{eq:two_step_reduction-geometric_diagram}, that is,
  $F:= \Phi_{\DCp{\SG}} \circ F_2 \circ (\check{\Phi_{\DCp{H}}})^{-1}
  \circ (\Phi_{\DCp{\SG/H}})^{-1}$.
  Then, the following statements are true.
  \begin{enumerate}
  \item \label{it:isomorphism_reduced_spaces_in_stages-diffeo} $F$ is
    a diffeomorphism.
  \item \label{it:isomorphism_reduced_spaces_in_stages-morphism}
    $F\in\Mor_{\DLPSC}(\mathcal{M}^\SG,\mathcal{M}^{\SG/H})$.
  \item \label{it:isomorphism_reduced_spaces_in_stages-iso} $F$ is an
    isomorphism in $\DLPSC$.
  \end{enumerate}
\end{theorem}

\begin{proof}
  By Proposition~\ref{prop:generalized_isomorphisms_associated_to_DC},
  $\Phi_{\DCp{\SG}}$ and $\Phi_{\DCp{\SG/H}}$ are diffeomorphisms
  and, by Lemma~\ref{le:two_step_diagram-properties}, the same
  happens to $F_2$ and $\check{\Phi_{\DCp{H}}}$. Hence, $F$ is
  a diffeomorphism, proving
  point~\ref{it:isomorphism_reduced_spaces_in_stages-diffeo}.

  Next, as $\Upsilon_{\DCp{\SG/H}}\in
  \Mor_{\DLPSC}(\mathcal{M}^H,\mathcal{M}^{\SG/H})$ and
  $\Upsilon_{\DCp{H}}\in \Mor_{\DLPSC}(\mathcal{M},\mathcal{M}^H)$, we
  have that $\Upsilon_{\DCp{\SG/H}}\circ \Upsilon_{\DCp{H}}\in
  \Mor_{\DLPSC}(\mathcal{M},\mathcal{M}^{\SG/H})$. Also,
  $\Upsilon_{\DCp{\SG}}\in\Mor_{\DLPSC}(\mathcal{M},\mathcal{M}^\SG)$,
  $F$ is smooth and $F\circ \Upsilon_{\DCp{\SG}} =
  \Upsilon_{\DCp{\SG/H}}\circ \Upsilon_{\DCp{H}}$, so that, by
  Lemma~\ref{le:comm_triangle_side_function_imp_side_morph}
  point~\ref{it:isomorphism_reduced_spaces_in_stages-morphism} is
  true. Using
  point~\ref{it:isomorphism_reduced_spaces_in_stages-diffeo} and
  Lemma~\ref{le:morphism_DLPSC_and_diffeo_imp_iso_DLPSC},
  point~\ref{it:isomorphism_reduced_spaces_in_stages-iso} follows.
\end{proof}

\begin{theorem}\label{thm:4_pts_in_stages}
  Consider the data given at the beginning of this section. 
  \begin{enumerate}
  \item \label{it:4_pts_in_stages-trajectories} Let
    $(\epsilon_\cdot,m_\cdot) = ((\epsilon_0, m_1), \ldots,
    (\epsilon_{N-1},m_N))$ be a discrete path in $C'(E)$. For
    $k=0,\ldots,N-1$ define the discrete paths
    $(v^H_k,r^H_{k+1}):=\Upsilon_{\DCp{H}}(\epsilon_k,m_{k+1})$,
    $(v^\SG_k,r^\SG_{k+1}):=\Upsilon_{\DCp{\SG}}(\epsilon_k,m_{k+1})$
    and
    $(v^{\SG/H}_k,r^{\SG/H}_{k+1}):=\Upsilon_{\DCp{H}}(v^H_k,r^H_{k+1})$
    in $C'(\ti{H}_E)$, $C'(\ti{\SG}_E)$ and
    $C'(\ti{\SG/H}_{\ti{H}_E})$ respectively. Then, the following
    assertions are equivalent.
    \begin{enumerate}
    \item \label{it:4_pts_in_stages-1} $(\epsilon_\cdot, m_\cdot)$ is
      a trajectory of $\mathcal{M}$.
    \item \label{it:4_pts_in_stages-G} $(v^\SG_\cdot,r^\SG_\cdot)$ is
      a trajectory of $\mathcal{M}^\SG$.
    \item \label{it:4_pts_in_stages-H} $(v^H_\cdot,r^H_\cdot)$ is a
      trajectory of $\mathcal{M}^H$.
    \item \label{it:4_pts_in_stages-G/H}
      $(v^{\SG/H}_\cdot,r^{\SG/H}_\cdot)$ is a trajectory of
      $\mathcal{M}^{\SG/H}$.
    \end{enumerate}
  \item \label{it:4_pts_in_stages-trajectories_F} Let
    $F:C'(\ti{\SG/H}_{\ti{H}_E}) \rightarrow C'(\ti{\SG}_E)$ be the
    diffeomorphism defined in
    Theorem~\ref{th:isomorphism_reduced_spaces_in_stages}. Then
    $F(v^{\SG/H}_k, r^{\SG/H}_{k+1}) = (v^\SG_k,r^\SG_{k+1})$ for all
    $k$.
  \item \label{it:4_pts_in_stages-isomorphism} The DLPSs
    $\mathcal{M}^\SG$ and $\mathcal{M}^{\SG/H}$ are isomorphic in
    $\DLPSC$.
  \end{enumerate}
\end{theorem}

\begin{proof}
  By Proposition~\ref{prop:Upsilon_DC_is_morphism}
  $\Upsilon_{\DCp{\SG}}$, $\Upsilon_{\DCp{H}}$ and
  $\Upsilon_{\DCp{\SG/H}}$ are morphisms in $\DLPSC$. Then,
    point~\ref{it:4_pts_in_stages-trajectories} follows from
    Theorem~\ref{thm:morphisms_and_trajectories}.

  Point~\ref{it:4_pts_in_stages-trajectories_F} is true by the
  following computation.
  \begin{equation*}
    \begin{split}
      (v^\SG_k,r^\SG_{k+1}) =&
      \Upsilon_{\DCp{\SG}}(\epsilon_k,m_{k+1}) = (F\circ
      \Upsilon_{\DCp{\SG/H}} \circ
      \Upsilon_{\DCp{H}})(\epsilon_k,m_{k+1}) \\=& (F\circ
      \Upsilon_{\DCp{\SG/H}})(v^H_k,r^H_{k+1}) =
      F(v^{\SG/H}_k,r^{\SG/H}_{k+1}).
    \end{split}
  \end{equation*}

  Point~\ref{it:4_pts_in_stages-isomorphism} is immediate from
  point~\ref{it:isomorphism_reduced_spaces_in_stages-iso} in
  Theorem~\ref{th:isomorphism_reduced_spaces_in_stages}.
\end{proof}


\subsection{Discrete connections derived from a Riemannian metric}
\label{sec:discrete_connections_derived_from_a_riemannian_metric}

The conditions stated at the beginning of
Section~\ref{sec:comparison_with_reduction_by_the_full_symmetry_group}
require the choice of three discrete connections $\DCp{H}$,
$\DCp{\SG}$ and $\DCp{\SG/H}$ on the corresponding principal
bundles. One case where such connections are known to exist is when
the total space of the corresponding principal bundle carries a
Riemannian metric and the structure group acts by isometries; this is
the content of Theorem 5.2
in~\cite{ar:fernandez_zuccalli-a_geometric_approach_to_discrete_connections_on_principal_bundles}. In
addition, $\DCp{H}$ is required to satisfy either one of the
conditions in Lemma~\ref{le:G-equivariance_H_connection_form}. In this
Section we prove that when the total space $Q$ of a principal
$\SG$-bundle $\pi^{Q,\SG}:Q\rightarrow Q/\SG$ is equipped with a
$\SG$-invariant Riemannian metric, it is possible to apply Theorem 5.2
in~\cite{ar:fernandez_zuccalli-a_geometric_approach_to_discrete_connections_on_principal_bundles}
to construct discrete connections $\DCp{H}$ satisfying the conditions
in Lemma~\ref{le:G-equivariance_H_connection_form} on the principal
$H$-bundle $\pi^{Q,H}:Q\rightarrow Q/H$ for any closed and normal
subgroup $H\subset \SG$.

The construction analyzed in Theorem 5.2
in~\cite{ar:fernandez_zuccalli-a_geometric_approach_to_discrete_connections_on_principal_bundles}
is as follows. When $Q$ is a Riemannian manifold and a Lie group $H$
acts on $Q$ by isometries, the \jdef{vertical bundle}
$\mathcal{V}^H\subset TQ$ defined by
$\mathcal{V}^H_q:=T_q(l^Q_H(q)) \subset T_qQ$ has an orthogonal
complement, the \jdef{horizontal bundle} $\mathcal{H}^H$. This
horizontal bundle determines a connection $\CCp{H}$ on the principal
$H$-bundle $\pi^{Q,H}:Q\rightarrow Q/H$. In addition, there is a
unique Riemannian metric on $Q/H$ that makes $Q/H$ a Riemannian
manifold and $\pi^{Q,H}$ a Riemannian submersion. Standard results of
Riemannian Geometry show that, for any $r\in Q/H$, there are open sets
$W_r\subset Q/H$ containing $r$ and such that any two points in $W_r$
can be joined by a unique length-minimizing geodesic that, also, is
contained in $W_r$ (see Theorem 3.6 on page 166
of~\cite{bo:kobayashi_nomizu-foundations-v1}); we call these sets
\jdef{geodesically convex}. Using such a collection
$\{W_r:r\in Q/H\}$, the open set
\begin{equation}\label{eq:domain_riemannian_connection-def}
  \mathcal{U}:=\cup_{r\in Q/H} \left((\pi^{Q,H})^{-1}(W_r) \times 
    (\pi^{Q,H})^{-1}(W_r) \right) \subset Q\times Q
\end{equation}
is defined. Then, a function $\DCp{H}:\mathcal{U}\rightarrow H$ is
constructed as follows. Given $(q_0,q_1)\in \mathcal{U}$, there is
$r\in Q/H$ such that $\pi^{Q,H}(q_0), \pi^{Q,H}(q_1)\in W_r$. Let
$\gamma:[0,1]\rightarrow Q/H$ be the unique length-minimizing geodesic
contained in $W_r$ and joining $\pi^{Q,H}(q_0)$ to
$\pi^{Q,H}(q_1)$. Let $\ti{\gamma}$ be the $\CCp{H}$-horizontal lift
of $\gamma$ to $Q$, starting at $q_0$. Finally, let
\begin{equation}\label{eq:riemannian_connection-discrete_connection_form}
  \DCp{H}(q_0,q_1) := \kappa_{q_1}(\ti{\gamma(1)},q_1),
\end{equation}
where $\kappa_{q_1}:Q_{\pi^{Q,H}(q_1)} \rightarrow H$ is the smooth
map defined by $\kappa_{q_1}(l^Q_h(q_1),q_1):=h$.  Theorem 5.2
in~\cite{ar:fernandez_zuccalli-a_geometric_approach_to_discrete_connections_on_principal_bundles}
asserts that there is a discrete connection $\DCp{H}$ on the principal
$H$-bundle $\pi^{Q,H}:Q\rightarrow Q/H$ whose domain is $\mathcal{U}$
and whose associated discrete form is given
by~\eqref{eq:riemannian_connection-discrete_connection_form}.

Below, we consider the case where $\SG$ is a Lie group and $H\subset
\SG$ is a closed normal subgroup. $\SG$ acts on the Riemannian
manifold $Q$ by isometries and in such a way such that
$\pi^{Q,\SG}:Q\rightarrow Q/\SG$ is a principal $\SG$-bundle. Then, by
restricting the $\SG$-action to an $H$-action, $H$ acts by isometries
on $Q$ and $\pi^{Q,H}:Q\rightarrow Q/H$ is a principal
$H$-bundle. But, still, $\SG/H$ acts on $Q/H$ by isometries and making
$\pi^{Q/H,\SG/H}:Q/H\rightarrow (Q/H)/(\SG/H)$ a principal
$\SG/H$-bundle. 

\begin{lemma}\label{le:existence_of_G/H_invariant_W_r}
  Under the previous conditions, there is a collection of open subsets
  $\{W_r\subset Q/H:r\in Q/H\}$ that are geodesically convex as above
  that, in addition, satisfies
  \begin{equation}\label{eq:G/H_invariant_W_r}
    l^{Q/H}_{\pi^{\SG,H}(g)}(W_{\pi^{Q,H}(q)}) = W_{\pi^{Q,H}(l^Q_g(q))} 
    \stext{ for all } q\in Q \text{ and } g\in\SG.
  \end{equation}
\end{lemma}

\begin{proof}
  In the current context, we have the commutative diagram
  \begin{equation*}
    \xymatrix{
      {Q} \ar[r]^{\pi^{Q,H}} \ar[d]_{\pi^{Q,\SG}} & {Q/H} \ar[d]^{\pi^{Q/H,\SG/H}} \\
      {Q/\SG} \ar[r]_{\phi}^{\sim} & {\frac{Q/H}{\SG/H}}
    }
  \end{equation*}
  where all the $\pi$-mappings are principal bundles and $\phi$ is a
  diffeomorphism. Let $\ti{\sigma}$ be a section of $\pi^{Q/H,\SG/H}$
  that may be discontinuous, and define $\sigma:Q\rightarrow Q/H$ by
  $\sigma:= \ti{\sigma}\circ \phi\circ \pi^{Q,\SG}$. It is easy to see
  that $\sigma$ is $\SG$-invariant, that its image intersects each
  $\SG/H$-orbit in $Q/H$ in exactly one point and that, for each $q\in
  Q$, $\sigma(q)$ and $\pi^{Q,H}(q)$ are on the same $\SG/H$-orbit.

  For each $\sigma(q)\in Q/H$, let $W_{\sigma(q)} \subset Q/H$ be any
  geodesically convex open subset. Using the $G/H$-action $l^{Q/H}$,
  for each $q\in Q$, we define
  \begin{equation*}
    W_{\pi^{Q,H}(q)} := l^{Q/H}_{\pi^{\SG,H}(g)}(W_{\sigma(q)}) \stext{ where }
    \pi^{\SG,H}(g):= \kappa(\sigma(q),\pi^{Q,H}(q)).
  \end{equation*}
  Since $l^{Q/H}_{\pi^{\SG,H}(g)}$ is an isometry in $Q/H$, the open
  sets $W_{\pi^{Q,H}(q)}$ are also geodesically convex. A direct
  computation shows that the collection $\{W_{\pi^{Q,H}(q)}: q\in Q\}$
  satisfies~\eqref{eq:G/H_invariant_W_r}.
\end{proof}

\begin{proposition}\label{prop:existence_of_G_adapted_Ad^H}
  With the same conditions as above, let $\mathcal{U}$ be defined
  by~\eqref{eq:domain_riemannian_connection-def}, for a collection of
  geodesically convex open subsets $\{W_r\subset Q/H:r\in Q/H\}$
  satisfying~\eqref{eq:G/H_invariant_W_r}. Then
  \begin{enumerate}
  \item \label{it:existence_of_G_adapted_Ad^H-open} $\mathcal{U}$ is
    $\SG$-invariant for the diagonal $\SG$-action $l^{Q\times Q}$.
  \item \label{it:existence_of_G_adapted_Ad^H-} The discrete
    connection with domain $\mathcal{U}$ and discrete connection form
    $\DCp{H}$ defined above satisfies
    condition~\eqref{eq:G-equivariance_H_connection_form} in
    Lemma~\ref{le:G-equivariance_H_connection_form}.
  \end{enumerate}
\end{proposition}

\begin{proof}
  Let $(q_0,q_1)\in \mathcal{U}$ and $g\in \SG$. By definition of
  $\mathcal{U}$, there is $\pi^{Q,H}(q)\in Q/H$ such that
  $\pi^{Q,H}(q_0), \pi^{Q,H}(q_1)\in W_{\pi^{Q,H}(q)}$. Hence, for
  $j=0,1$,
  \begin{equation*}
    \begin{split}
      \pi^{Q,H}(l^Q_g(q_j)) =
      l^{Q/H}_{\pi^{\SG,H}(g)}(\underbrace{\pi^{Q,H}(q_j)}_{\in
        W_{\pi^{Q,H}(q)}}) \in
      l^{Q/H}_{\pi^{\SG,H}(g)}(W_{\pi^{Q,H}(q)}) =
      W_{\pi^{Q,H}(l^Q_g(q))},
    \end{split}
  \end{equation*}
  Hence $l^{Q\times Q}_g(q_0,q_1) = (l^Q_g(q_0),l^Q_g(q_1)) \in
  \mathcal{U}$, proving
  part~\ref{it:existence_of_G_adapted_Ad^H-open}.

  Given $(q_0,q_1)\in \mathcal{U}$ and $\pi^{Q,H}(q)$ as above, let
  $\gamma_0$ and $\gamma_1$ be the unique length-minimizing geodesics
  contained in $W_{\pi^{Q,H}(q)}$ and $W_{\pi^{Q,H}(l^Q_g(q))}$ going
  from $\pi^{Q,H}(q_0)$ to $\pi^{Q,H}(q_1)$ and from
  $\pi^{Q,H}(l^Q_g(q_0))$ to $\pi^{Q,H}(l^Q_g(q_1))$. Let
  $\ti{\gamma}_0$ and $\ti{\gamma}_1$ be the $\CCp{H}$-horizontal
  lifts starting at $q_0$ and $l^Q_g(q_0)$ respectively.

  Notice that, by the uniqueness of the length-minimizing geodesics in
  $W_{\pi^{Q,H}(l^Q_g(q))}$ and since $l^{Q/H}_{\pi^{Q,H}(g)}$ is an
  isometry in $Q/H$, we have that $\gamma_1 = l^{Q/H}_{\pi^{Q,H}(g)}
  \circ \gamma_0$. 

  Let $\rho:=l^Q_g \circ \ti{\gamma}_0$. It is easy to check that
  $\rho$ is a lift of $\gamma_1$ starting at $l^Q_g(q_0)$. It is also
  $\CCp{H}$-horizontal, a fact that follows from the $\SG$-invariance
  of $\mathcal{H}^H$, that is, from
  $dl^Q_g(q')(\mathcal{H}^H_{q'})\subset \mathcal{H}^H_{l^Q_g(q')}$
  for all $q'\in Q$.  By the uniqueness of the horizontal lifts, we
  conclude that $\ti{\gamma}_1 = \rho$.

  Finally,
  using~\eqref{eq:riemannian_connection-discrete_connection_form}, we
  have
  \begin{equation*}
    \begin{split}
      \DCp{H}(l^Q_g(q_0),l^Q_g(q_1)) =&
      \kappa_{l^Q_g(q_1)}(\ti{\gamma}_1(1),l^Q_g(q_1)) =
      \kappa_{l^Q_g(q_1)}(l^Q_g(\ti{\gamma}_0(1)),l^Q_g(q_1)) \\=& g
      \kappa_{q_1}(\ti{\gamma}_0(1),q_1) g^{-1} = g \DCp{H}(q_0,q_1)
      g^{-1},
    \end{split}
  \end{equation*}
  that is, identity~\eqref{eq:G-equivariance_H_connection_form} holds,
  concluding the proof of part~\ref{it:existence_of_G_adapted_Ad^H-}.
\end{proof}


\section{Poisson structures}
\label{sec:poisson_structures}

It is a well known and used fact that if $(Q,L_d)$ is a regular DMS,
there is a symplectic structure $\omega_{L_d}$ defined in (an open
subset containing the diagonal of) $Q\times Q$. Furthermore, the
discrete Lagrangian flow $F_{L_d}$ is symplectic for
$\omega_{L_d}$. This structure is important both for the theoretical
as well as the numerical applications of DMSs. Still, the dynamical
system obtained as the reduction of a DMS may not carry a symplectic
structure: an obvious reason could be that
$\dim(\ti{\SG}\times(Q/\SG)) = 2\dim(Q) -\dim(\SG)$ could be odd,
making it impossible for the reduced space $\ti{\SG}\times(Q/\SG)$ to
be a symplectic manifold. 

The purpose of this section is to show that when a symmetric DLPS has
a Poisson structure, in a sense to be defined below, and the symmetry
group acts by Poisson maps, then its reduction also carries a Poisson
structure and the reduction morphism is a Poisson map. In principle,
these structures could be uninteresting ---for instance, the trivial
Poisson structure is always a Poisson structure for a DLPS.  Still,
when a DLPS has an interesting structure, as is the case of those DLPSs
obtained from DMSs, the natural Poisson structure arising from the
symplectic structure is inherited by all reductions, as we see
below.

\begin{definition}
  Let $\mathcal{M}=(E,L_d,\IVCM)$ be a DLPS. We say that a Poisson
  structure $\{,\}_{C'(E)}$ on $C'(E)$ is a \jdef{Poisson structure of
    $\mathcal{M}$} if the flow map $F_\mathcal{M}$ is a Poisson map for
  $\{,\}_{C'(E)}$.
\end{definition}

\begin{proposition}\label{prop:reduced_poisson_structure}
  Let $\{,\}_{C'(E)}$ be a Poisson structure of
  $\mathcal{M}=(E,L_d,\IVCM)$. If $\SG$ is a symmetry group of
  $\mathcal{M}$ that preserves $\{,\}_{C'(E)}$ and $\DC$ is a discrete
  connection on $\pi^{M,\SG}:M\rightarrow M/\SG$, then there is a
  Poisson structure $\{,\}_{C'(\ti{\SG}_E)}$ of the reduced system
  $\mathcal{M}/(\SG,\DC)$ such that the reduction morphism
  $\Upsilon_\DC$ is a Poisson map, \ie,
  \begin{equation}\label{eq:reduced_poisson_structure}
    \Upsilon_\DC^*(\{f_1,f_2\}_{C'(\ti{\SG}_E)}) = \{\Upsilon_\DC^*(f_1), 
    \Upsilon_\DC^*(f_2)\}_{C'(E)} \stext{ for all } 
    f_1, f_2 \in C^\infty(\ti{\SG}_E).
  \end{equation}
\end{proposition}

\begin{proof}
  Being $\SG$ a symmetry group of $\mathcal{M}$, by
  Lemma~\ref{le:Upsilon_DC_is_ppal_bundle},
  $\Upsilon_\DC:C'(E)\rightarrow C'(\ti{\SG}_E)$ is a principal
  $\SG$-bundle. Then, the $\SG$-action on $C'(E)$ is free and proper
  and $\Upsilon_\DC$ is a surjective submersion. As $\SG$ acts on
  $C'(E)$ by Poisson maps, it follows from Theorem 10.5.1
  in~\cite{bo:MR-mechanics_symmetry} that there is a unique Poisson
  structure $\{,\}_{C'(\ti{\SG}_E)}$ on $C'(\ti{\SG}_E)$ such that
  $\Upsilon_\DC$ becomes a Poisson map,
  hence~\eqref{eq:reduced_poisson_structure} holds. By
  Theorem~\ref{thm:two_point_theorem-unconstrained-generalized}, we
  have the commutative diagram of manifolds and smooth maps:
  \begin{equation*}
    \xymatrix{
      {C'(E)} \ar[r]^{F_{\mathcal{M}}} \ar[d]_{\Upsilon_\DC} & 
      {C'(E)} \ar[d]^{\Upsilon_\DC}\\
      {C'(\ti{\SG}_E)} \ar[r]_{F_{\mathcal{M}/\SG}} & {C'(\ti{\SG}_E)}
    }
  \end{equation*}
  where $F_{\mathcal{M}/\SG}$ is the flow of the reduced system. As
  $\Upsilon_\DC$ and $F_{\mathcal{M}}$ are Poisson maps, with
  $\Upsilon_\DC$ onto, it follows from
  Lemma~\ref{le:triangles_and_poisson_maps} below that,
  $F_{\mathcal{M}/\SG}$ is a Poisson map. All together, we have seen
  that $\{,\}_{C'(\ti{\SG}_E)}$ is a Poisson structure of
  $\mathcal{M}/(\SG,\DC)$.
\end{proof}

\begin{lemma}\label{le:triangles_and_poisson_maps}
  Let $\phi_1:M\rightarrow M_1$ and $\phi_2:M\rightarrow M_2$ be
  Poisson maps and assume that $\phi_1$ is onto. If
  $f:M_1\rightarrow M_2$ is a smooth map such that
  $f\circ \phi_1 =\phi_2$, then $f$ is a Poisson map.
\end{lemma}

\begin{proof}
  As $f\circ \phi_1 =\phi_2$, a direct computation shows that, for
  $h_1,h_2\in C^\infty(M_2)$, $\phi_1^*(f^*(\{h_1,h_2\}_{M_2})) =
  \phi_1^*( \{ f^*(h_1), f^*(h_2)\}_{M_1})$. The result follows by
  noticing that, as $\phi_1$ is onto, $\phi_1^*$ is one to one.
\end{proof}

Recall that a regular DMS $(Q,L_d)$ carries a natural closed $2$-form
$\omega_{L_d}$ that is symplectic in, at least, an open subset of
$Q\times Q$ containing the diagonal $\Delta_Q$. We have the following
result.

\begin{lemma}\label{le:diagonal_actions_are_symplectic_for_DMS}
  Let $\SG$ be a symmetry group of the regular discrete mechanical
  system $(Q,L_d)$. Then, the diagonal $\SG$-action on $Q\times Q$ is
  symplectic for the symplectic form $\omega_{L_d}$.
\end{lemma}

\begin{proof}
  See the argument at the beginning of page 375
  in~\cite{ar:marsden_west-discrete_mechanics_and_variational_integrators}.
\end{proof}

In particular, a DLPS $\mathcal{M}=(Q,L_d,\IVCM)$ that comes from a
DMS $(Q,L_d)$ as in
Example~\ref{ex:mechanical_system_as_generalized_mechanical_system},
carries a natural Poisson structure $\{,\}_{Q\times Q}$ arising from
the symplectic structure $\omega_{L_d}$ on $Q\times Q$. It is well
known that $F_{\mathcal{M}}^*(\omega_{L_d}) = \omega_{L_d}$ (see
Section 1.3.2
in~\cite{ar:marsden_west-discrete_mechanics_and_variational_integrators}).
Hence $F_{\mathcal{M}}$ is a Poisson map and, consequently,
$\{,\}_{Q\times Q}$ is a Poisson structure of $\mathcal{M}$.

When $\SG$ is a symmetry group of $(Q,L_d)$, it is a symmetry group of
$\mathcal{M}$ and, by
Lemma~\ref{le:diagonal_actions_are_symplectic_for_DMS}, it acts on
$C'(Q)=Q\times Q$ by Poisson maps for $\{,\}_{Q\times Q}$. Fixing a
discrete connection $\DC$ on $\pi^{Q,\SG}:Q\rightarrow Q/\SG$, by
Proposition~\ref{prop:reduced_poisson_structure}, the reduced system
$\mathcal{M}/(\SG,\DC)$ has a natural Poisson structure induced by
$\{,\}_{Q\times Q}$ and $\Upsilon_\DC$ is a Poisson map.

We conclude that all DLPSs obtained from a DMS by a finite number of
reductions have natural Poisson structures that make the corresponding
reduction morphisms Poisson maps.


\section{Appendix}
\label{sec:appendix}

The purpose of this Appendix is to review some basic definitions and
standard results, using a notation that is compatible with the rest of
the paper. Sections~\ref{sec:group_actions_on_manifolds}
and~\ref{sec:bundles} contain well known
material. Section~\ref{sec:group_actions_on_bundles} contains some
nonstandard material.


\subsection{Group actions on manifolds}
\label{sec:group_actions_on_manifolds}

A continuous map $f:X\rightarrow Y$ between topological spaces is
\jdef{proper} if $f^{-1}(K)$ is compact for every $K\subset Y$
compact. A $\SG$-action $l^M$ of a Lie group $\SG$ on a manifold $M$
is \jdef{proper} if the map $L^M:\SG\times M\rightarrow M\times M$
defined by $L^M(g,m) := (l^M_g(m),m)$ is proper. The following result
gives a characterization of properness in terms of sequences that is
very convenient in practice.

\begin{proposition}\label{prop:properness_and_sequences}
  Let $M$ be a manifold and $\SG$ be a Lie group acting on $M$ by
  $l^M$. Assume that $l^M$ has the property that for any convergent
  sequence $(m_j)_{j\in\N}$ in $M$ and sequence $(g_j)_{j\in\N}$ in
  $\SG$ such that the sequence $(l^M_{g_j}(m_j))_{j\in\N}$ is
  convergent, there exists a convergent subsequence of
  $(g_j)_{j\in\N}$. Then $l^M$ is a proper action. Conversely, if the
  action $l^M$ is proper, then the property holds.
\end{proposition}
\begin{proof}
  See Proposition 9.13
  in~\cite{bo:lee-introduction_to_smooth_manifolds}.
\end{proof}

\begin{theorem}\label{thm:quotient_manifolds}
  Let $l^M$ be a smooth, free and proper action of the Lie group $\SG$
  on $M$. Then, the quotient space $M/\SG$ is a topological manifold
  of dimension $\dim(M)-\dim(\SG)$. In addition, $M/\SG$ has a unique
  smooth structure with the property that the quotient map
  $\pi^{M,\SG}:M\rightarrow M/\SG$ is a smooth
  submersion. Furthermore, $\pi^{M,\SG}:M\rightarrow M/\SG$ is a
  principal $\SG$-bundle (Definition~\ref{def:principal_G_bundle}).
\end{theorem}
\begin{proof}
  See Theorem 9.16 in~\cite{bo:lee-introduction_to_smooth_manifolds}.
\end{proof}

\begin{proposition}\label{prop:quotient_of_equivariant_maps}
  Let $\SG$ be a Lie group acting smoothly on the manifolds $M$ and
  $N$ in such a way that $\pi^{M,\SG}:M\rightarrow M/\SG$ and
  $\pi^{N,\SG}:N\rightarrow N/\SG$ are smooth submersions (in
  particular, $M/\SG$ and $N/\SG$ are smooth manifolds). If
  $f:M\rightarrow N$ is a smooth $\SG$-equivariant map, then there is
  a unique smooth map $\check{f}:M/\SG \rightarrow N/\SG$ such that
  $\pi^{N,\SG}\circ f = \check{f}\circ \pi^{M,\SG}$.
\end{proposition}

\begin{proof}
  An application of the local description of submersions.
\end{proof}

\begin{corollary}\label{cor:quotient_maps-G_equivariant_map}
  Let $\SG$ be a Lie group acting smoothly, freely and properly on the
  manifolds $M$ and $N$. If $f:M\rightarrow N$ is a smooth
  $\SG$-equivariant map, then there is a unique smooth map
  $\check{f}:M/\SG \rightarrow N/\SG$ such that $\pi^{N,\SG}\circ f =
  \check{f}\circ \pi^{M,\SG}$.
\end{corollary}

\begin{corollary}\label{cor:quotient_maps-G_invariant_map}
  Let $\SG$ be a Lie group acting smoothly, freely and properly on the
  manifold $M$. If $f:M\rightarrow N$ is a smooth $\SG$-invariant map,
  then there is a unique smooth map $\check{f}:M/\SG \rightarrow N$ such
  that $f = \check{f}\circ \pi^{M,\SG}$.
\end{corollary}


\subsection{Bundles}
\label{sec:bundles}

\begin{definition}\label{def:fiber_bundle}
  A \jdef{fiber bundle} is a quadruple $(E,M,\phi,F)$ where $E$, $M$
  and $F$ are smooth manifolds and $\phi:E\rightarrow M$ is a smooth
  map such that each $m\in M$ has a neighborhood $U\subset M$ and a
  diffeomorphism $\Phi_U:\phi^{-1}(U)\rightarrow U\times F$ that makes
  the following diagram commutative.
  \begin{equation}\label{eq:fiber_bundle_diagram}
    \xymatrix{ {\phi^{-1}(U)} \ar[r]^{\Phi_U} \ar[d]_{\phi} & {U\times
        F} \ar[dl]^{p_1} \\ {U} & {}}
  \end{equation}
  In this case, $E$, $M$ and $F$ are called the \jdef{total space},
  \jdef{base space} and \jdef{fiber} of the fiber bundle. A pair
  $(U,\Phi_U)$ as above is called a \jdef{trivializing chart} of the
  bundle. It is convenient to denote a fiber bundle $(E,M,\phi,F)$ by
  $E$ or $\phi$.
\end{definition}

If $(E,M,\phi,F)$ is a fiber bundle, given two of its trivializing
charts $(U_\alpha,\Phi_\alpha)$ and $(U_\beta,\Phi_\beta)$ such that
$U_{\alpha\beta} := U_\alpha\cap U_\beta\neq \emptyset$, we can write
$(\Phi_\alpha \circ \Phi_\beta^{-1})(m,f) =
(m,\Phi_{\alpha\beta}(m)(f))$ for all $m\in U_{\alpha\beta}$ and $f\in
F$, for a smooth map $\Phi_{\alpha\beta}:U_{\alpha\beta}\rightarrow
\Diff(F)$ known as a \jdef{transition function} of the bundle. The
fiber bundle is called a \jdef{$\SG$-bundle} for a Lie group $\SG$ if
there is a right $\SG$-action on $F$ denoted by $r^F$ such that all
transition functions are of the form $\Phi_{\alpha\beta}(m) =
r^F_{\chi_{\alpha\beta}(m)}$ for a family of smooth functions
$\chi_{\alpha\beta}:U_{\alpha\beta}\rightarrow \SG$ that satisfy
$\chi_{\beta\gamma}(m)\chi_{\alpha\beta}(m) = \chi_{\alpha\gamma}(m)$
for all $m\in U_\alpha\cap U_\beta \cap U_\gamma \neq \emptyset$.

\begin{definition}\label{def:principal_G_bundle}
  Let $(E,M,\phi,\SG)$ be a $\SG$-bundle such that $\SG$ acts on the
  fiber $\SG$ by right multiplication. Then, $E$ is called a
  \jdef{principal $\SG$-bundle} over $M$.
\end{definition}

\begin{theorem}\label{thm:ppal_bundles_via_embedding}
  Let $(E,M,\phi,\SG)$ be a principal $\SG$-bundle. Then $\phi$ is a
  surjective submersion, $\SG$ acts freely on the left on $E$ and the
  $\SG$-orbits for this action are of the form $\phi^{-1}(m)$ for
  $m\in M$. Conversely, if $\phi:E\rightarrow M$ is a surjective
  submersion and the Lie group $\SG$ acts freely on the left on $E$ in
  such a way that the $\SG$ orbits are of the form $\phi^{-1}(m)$ for
  $m\in M$, then $(E,M,\phi,\SG)$ is a principal $\SG$-bundle.
\end{theorem}

\begin{proof}
  The first part is direct computation. See Lemma 18.3
  in~\cite{bo:michor-topics_in_differential_geometry} for the converse
  (in the right action case).
\end{proof}

\begin{remark}
  All principal $\SG$-bundles are left $\SG$-spaces by
  Theorem~\ref{thm:ppal_bundles_via_embedding}. This follows,
  eventually, from the fact that our $\SG$-spaces have right
  $\SG$-actions on the fibers. The opposite choices are common in most
  of the fiber bundle literature
  (see~\cite{bo:husemoller-fibre_bundles}). Our choice is the standard
  one in Geometric Mechanics
  (see~\cite{bo:cendra_marsden_ratiu-lagrangian_reduction_by_stages}).
\end{remark}

\begin{corollary}\label{cor:quotient_manifolds-fiber_bundle}
  In the context of Theorem~\ref{thm:quotient_manifolds},
  $\pi^{M,\SG}:M\rightarrow M/\SG$ is a principal $\SG$-bundle.
\end{corollary}

\begin{lemma}\label{le:G_actions_on_ppal_bundles_are_proper}
  Let $\psi:M\rightarrow R$ be a principal $\SG$-bundle. Then the
  $\SG$-action on $M$ is proper. If, furthermore, $(E,M,\phi,F)$ is a
  fiber bundle and $\SG$ acts on $E$ by $l^E$ making $\phi$
  equivariant, then $l^E$ is proper.
\end{lemma}

\begin{proof}
  The first statement follows from
  Proposition~\ref{prop:properness_and_sequences}, using the local
  triviality of the bundles. The second repeats the same argument
  building on the properness of the $\SG$-action on $M$.
\end{proof}

\begin{lemma}\label{le:quotient_action_properties}
  Let $\SG$ be a Lie group acting smoothly, properly and freely on the
  manifold $Q$. Let $H\subset \SG$ be a closed and normal Lie
  subgroup. The $\SG$-action on $Q$ induces an $H$-action on $Q$. Then
  \begin{enumerate}
  \item \label{it:quotient_action-well_defined} $\SG/H$ acts on $Q/H$
    by the induced action
    $l^{Q/H}_{\pi^{\SG,H}(g)}(\pi^{Q,H}(q)):=\pi^{Q,H}(l^Q_g(q))$.
  \item \label{it:quotient_action-free_and_proper} The $\SG/H$-action
    $l^{Q/H}$ is free and proper.
  \end{enumerate}
\end{lemma}

\begin{definition}\label{def:bundle_map}
  Let $(E_j,M_j,\phi_j,F_j)$ be fiber bundles for $j=1,2$. A
  \jdef{bundle map} from $E_1$ to $E_2$ is a pair $(\Psi,\psi)$ of
  smooth maps $\Psi:E_1\rightarrow E_2$ and $\psi:M_1\rightarrow M_2$
  such that the following diagram is commutative.
  \begin{equation*}
    \xymatrix{{E_1} \ar[r]^\Psi \ar[d]_{\phi_1} & {E_2}
      \ar[d]^{\phi_2} \\ {M_1} \ar[r]_{\psi} & {M_2}}
  \end{equation*}
\end{definition}


\subsection{Group actions on bundles}
\label{sec:group_actions_on_bundles}

The following definition introduces what we mean by the action of a
Lie group on a fiber bundle. We warn the reader that it may not be
completely standard.

\begin{definition}\label{def:group_acts_on_fiber_bundle}
  Let $\SG$ be a Lie group and $(E,M,\phi,F)$ a fiber bundle. We say
  that $\SG$ \jdef{acts on the fiber bundle $E$} if there are free
  left $\SG$-actions $l^E$ and $l^M$ and a right $\SG$-action $r^F$ on
  $F$ such that
  \begin{enumerate}
  \item $l^M$ induces a principal $\SG$-bundle structure
    $\pi^{M,\SG}:M\rightarrow M/\SG$,
  \item $\phi$ is a $\SG$-equivariant map for the given actions,
  \item \label{it:group_acts_on_FB-trivializing_chart} for every
    $m\in M$ there is a trivializing chart $(U,\Phi_U)$ of $E$ such
    that $U$ is $\SG$-invariant, $m\in U$ and, when considering the
    left $\SG$-action $l^{U\times F}$ on $U\times F$ given by
    $l^{U\times F}_g(m,f) := (l^M_g(m),r^F_{g^{-1}}(f))$, the map
    $\Phi_U$ is $\SG$-equivariant.
  \end{enumerate}
\end{definition}

\begin{example}\label{ex:extended_associated_bundle-G_acts}
  Let $\SG$ act on the fiber bundle $(E,M,\phi,F_1)$ by the left
  actions $l^E$ and $l^M$ and the right action $r^{F_1}$ on $F_1$, and
  let $F_2$ be a right $\SG$-manifold for the action
  $r^{F_2}$. Consider the left $\SG$-action $l^{E\times F_2}$ on the
  fiber bundle $(E\times F_2,M,\phi\circ p_1,F_1\times F_2)$ defined
  by $l^{E\times F_2}_g(\epsilon,f_2) :=
  (l^E_g(\epsilon),r^{F_2}_{g^{-1}}(f_2))$. Then $\SG$ acts on the
  fiber bundle $E\times F_2$. The only part of the verification that
  requires some work is the existence of local $\SG$-equivariant
  trivializations. This is done by taking the (Cartesian) product of
  $\SG$-equivariant trivialization of $E$ and the identity mapping on
  $F_2$. Using the right $\SG$-action $r^{F_1\times F_2}_g :=
  r^{F_1}_g \times r^{F_2}_g$ on $F_1\times F_2$ makes the resulting
  mapping $\SG$-equivariant, in the sense of
  point~\ref{it:group_acts_on_FB-trivializing_chart} of
  Definition~\ref{def:group_acts_on_fiber_bundle}.
\end{example}

\begin{proposition}\label{prop:quotient_of_fiber_bundles}
  Let $\SG$ be a Lie group that acts on the fiber bundle
  $(E,M,\phi,F)$. Then $\phi$ induces a smooth map
  $\check{\phi}:E/\SG\rightarrow M/\SG$ such that
  $(E/\SG,M/\SG,\check{\phi},F)$ is a fiber bundle.
\end{proposition}

\begin{proof}
  Since the $\SG$-actions on $E$ and $M$ are free and proper and
  $\phi$ is equivariant, by Theorem~\ref{thm:quotient_manifolds} and
  Corollary~\ref{cor:quotient_maps-G_equivariant_map}, we have that
  $E/\SG$ and $M/\SG$ are manifolds, the quotient mappings
  $\pi^{E,\SG}:E\rightarrow E/\SG$ and $\pi^{M,\SG}:M\rightarrow
  M/\SG$ are smooth submersions and $\check{\phi}$ is smooth.

  An outline of the proof of the local triviality of
  $(E/\SG,M/\SG,\check{\phi},F)$ goes as follows.  Since the existence
  of local trivializations is a local matter, we can assume that
  $\pi^{M,\SG}:M\rightarrow M/\SG$ is a trivial $\SG$-principal
  bundle, that is, it is $p_1:R\times \SG\rightarrow R$ for some
  manifold $R$ and the $\SG$-action on $M$ is $l^{R\times
    \SG}_g(r,g'):=(r,g g')$. Similarly, we can assume that
  $\phi:E\rightarrow M$ is $p_1:(R\times\SG)\times F\rightarrow
  R\times\SG$ and the $\SG$-action on $E$ is $l^{(R\times\SG)\times
    F}_g((r,g'),f) := ((r,g g'),r^F_{g^{-1}}(f))$.

  Using Corollary~\ref{cor:quotient_maps-G_equivariant_map}, $p_1$
  induces a map $\check{p_1}:((R\times\SG)\times F)/\SG \rightarrow
  (R\times \SG)/\SG = R$. In addition, define
  $\sigma:(R\times\SG)\times F\rightarrow R\times F$ by
  $\sigma(r,g,f):=(r,r^F_g(h))$. As $\sigma$ is smooth and
  $\SG$-invariant, it induces a smooth map $\check{\sigma}:((R\times
  \SG)\times F)/\SG\rightarrow R\times F$. In fact, $\check{\sigma}$
  is a diffeomorphism and satisfies $p_1\circ \check{\sigma} =
  \check{p_1}$. Thus, we have the following commutative diagram
  \begin{equation*}
    \xymatrix{
      {E/\SG = ((R\times\SG)\times F)} \ar[r]^(.7){\check{\sigma}} 
      \ar[d]_{\check{\phi} = \check{p_1}} & {R\times F} \ar[dl]^{p_1}\\
      {M/\SG = R} & {}
    }
  \end{equation*}
  showing the (local) triviality of the bundle
  $(E/\SG,M/\SG,\check{\phi},F)$, ending the proof.
\end{proof}

\begin{example}\label{ex:extended_associated_bundle-fiber_bundle}
  Applying Proposition~\ref{prop:quotient_of_fiber_bundles} to the
  setting of Example~\ref{ex:extended_associated_bundle-G_acts} we
  conclude that if $\SG$ acts on the fiber bundle $(E,M,\phi,F_1)$ and
  $F_2$ is a right $\SG$-manifold, then $((E\times
  F_2)/\SG,M/\SG,\check{\phi\circ p_1},F_1\times F_2)$ is a fiber
  bundle that we call the \jdef{associated bundle} and denote by
  $\ti{F_2}_E$.  A special case of this construction is the so called
  \jdef{conjugate bundle}, denoted by $\ti{\SG}_E$, that corresponds to
  the case when $F_2 = \SG$ and the right action is $r^{F_2}_g(h) :=
  l^{\SG}_{g^{-1}}(h) = g^{-1} h g$. For the conjugate bundle, we define
  $p^{M/\SG} := \check{\phi \circ p_1}$.

  When the Lie group $\SG$ acts on a manifold $Q$ in such a way that
  $\pi^{Q,\SG}:Q\rightarrow Q/\SG$ is a principal $\SG$-bundle,
  $(Q,Q,id_Q,\{pt\})$ is a fiber bundle with a $\SG$-action. The
  conjugate bundle in this case, $\ti{\SG}_Q$ coincides with the
  conjugate bundle $p^{Q/\SG}:\ti{\SG}\rightarrow Q/\SG$ considered in
  Section~\ref{sec:reduced_system_associated_to_a_symmetric_mechanical_system}
  and
  in~\cite{ar:fernandez_tori_zuccalli-lagrangian_reduction_of_discrete_mechanical_systems}.
\end{example}




\def\cprime{$'$} \def\polhk#1{\setbox0=\hbox{#1}{\ooalign{\hidewidth
  \lower1.5ex\hbox{`}\hidewidth\crcr\unhbox0}}} \def\cprime{$'$}
  \def\cprime{$'$}
\providecommand{\bysame}{\leavevmode\hbox to3em{\hrulefill}\thinspace}
\providecommand{\MR}{\relax\ifhmode\unskip\space\fi MR }
\providecommand{\MRhref}[2]{%
  \href{http://www.ams.org/mathscinet-getitem?mr=#1}{#2}
}
\providecommand{\href}[2]{#2}


\end{document}